\documentclass[12pt]{article}
\usepackage[margin =1in]{geometry}

\usepackage{pifont}
\usepackage{wrapfig}
\usepackage{comment}
\usepackage{subfig}
\usepackage[font=footnotesize]{caption}
\usepackage{amssymb,amsthm}
\usepackage{amsmath}
\usepackage{enumerate}
\usepackage{hyperref}
\usepackage{cite}
\usepackage{color}
\usepackage{cleveref}
\usepackage{algorithmic}
\usepackage{tablefootnote}

\numberwithin{equation}{section}

\newcommand{\norm}[1] {\left \| #1 \right \|}
\newcommand{\inclu}[0] {\ar@{^{(}->}}

\newcommand{\diag}{{\rm diag}}

\newcommand{\PP}{\mathbb{P}}

\newcommand{\EE}{\mathbb{E}}

\newcommand{\sign}{\mathrm{sign}}

\newcommand{\RR}{\mathbb{R}}

\newcommand{\cT}{\mathcal{T}}

\newcommand{\cD}{\mathcal{D}}
\newcommand{\cM}{\mathcal{M}}
\newcommand{\cN}{\mathcal{N}}

\newcommand{\rank}{\mathrm{rank}}

\newcommand{\abs}[1]{\left| #1 \right|}

\newcommand{\fnorm}[1]{\| #1 \|_{\mathrm{F}}}
\newcommand{\fknorm}[2]{\| #1 \|_{\mathrm{F,#2}}}

\newcommand{\ceil}[1]{\left \lceil #1 \right \rceil }

\newcommand{\argmax}{\text{argmax}}

\newcommand{\expect}[1]{\mathbb{E}\left[#1\right]}

\newcommand{\dotp}[1]{\left\langle #1\right\rangle}

\newtheorem{thm}{Theorem}[section]
\newtheorem{mdl}{Model}
\newtheorem{lemma}{Lemma}[section]

\newtheorem{definition}[thm]{Definition}
\newtheorem{proposition}[thm]{Proposition}
\newtheorem{lem}[thm]{Lemma}

\theoremstyle{remark}

\newif\ifcomment
\commentfalse 

\newcommand{\qq}[1]{\ifcomment{\color{blue}{Qing: #1}}\else\fi}
\newcommand{\ld}[1]{\ifcomment{\color{red}{Lijun: #1}}\else\fi}
\newcommand{\yc}[1]{\ifcomment{\color{magenta}{Yudong: #1}}\else\fi}
\newcommand{\lj}[1]{\ifcomment{\color{brown}{Liwei: #1}}\else\fi}
\newcommand{\zz}[1]{\ifcomment{\color{purple}{#1}}\else\fi}
\newcommand{\nct}[1]{{#1}}
\usepackage{mathtools}
\usepackage[boxruled]{algorithm2e}

\newcommand{\bb}{\mathbb}

\newcommand{\Brac}[1]{\left\lbrace #1 \right\rbrace}

\usepackage[inline]{enumitem}
\numberwithin{equation}{section}
\usepackage[boxruled]{algorithm2e}
\newcommand{\bS}{\mathbb{S}}

\newcommand{\truX}{X_\natural}

\usepackage[utf8]{inputenc} 
\usepackage[T1]{fontenc}    
\usepackage{hyperref}       
\usepackage{url}            
\usepackage{booktabs}       
\usepackage{amsfonts}       
\usepackage{nicefrac}       
\usepackage{microtype}      
\usepackage{xcolor}         
\title{Rank Overspecified Robust Matrix Recovery: Subgradient Method and Exact Recovery \thanks{The first two authors contributed equally to this paper.}}
\author{Lijun Ding\thanks{School of ORIE, Cornell University. Ithaca, NY 14850, USA;~\url{https://people.orie.cornell.edu/ld446}} \qquad 
Liwei Jiang\thanks{School of ORIE, Cornell University. Ithaca, NY 14850, USA;~\url{https://www.orie.cornell.edu/research/grad-students/liwei-jiang}}\qquad
	Yudong Chen\thanks{Department of Computer Sciences,
University of Wisconsin-Madison. Madison, WI 53706, USA;~\texttt{https://pages.cs.wisc.edu/{\raise.17ex\hbox{$\scriptstyle\mathtt{\sim}$}}yudongchen}} \qquad
	Qing Qu\thanks{Department of Electrical Engineering and Computer Science, University of Michigan-Ann Arbor. Ann Arbor, MI 48109, USA;~ \url{https://qingqu.engin.umich.edu/}}\qquad 
	Zhihui Zhu\thanks{Department of Electrical and Computer Engineering,
University of Denver. Denver, CO 80208, USA;~\texttt{https://mysite.du.edu/{\raise.17ex\hbox{$\scriptstyle\mathtt{\sim}$}}zzhu61}}}
\date{Last updated: September 23, 2021}
\begin{document}
\maketitle
\begin{abstract}
We study the robust recovery of a low-rank matrix from \nct{sparsely and} grossly corrupted Gaussian measurements, with no prior knowledge on the intrinsic rank. We consider the robust matrix factorization approach. We employ a robust $\ell_1$ loss function and deal with the challenge of the unknown rank by using an overspecified factored representation of the matrix variable. We then solve the associated nonconvex nonsmooth problem using a subgradient method with diminishing stepsizes. We show that under a regularity condition on the sensing matrices and corruption, which we call restricted direction preserving property (RDPP), even with rank overspecified, the subgradient method converges to the exact low-rank solution at a sublinear rate. Moreover, our result is more general in the sense that it automatically speeds up to a linear rate once the factor rank matches the unknown rank. On the other hand, we show that the RDPP condition holds under generic settings, such as Gaussian measurements under independent or adversarial sparse corruptions, where the result could be of independent interest. Both the exact recovery and the convergence rate of the proposed subgradient method are numerically verified in the overspecified regime. Moreover, our experiment further shows that our particular design of diminishing stepsize effectively prevents overfitting for robust recovery under overparameterized models, such as robust matrix sensing and learning robust deep image prior. This regularization effect is worth further investigation. 
\end{abstract}
\newpage
\tableofcontents

\section{Introduction}
Robust low rank matrix recovery problems are ubiquitous in  computer vision \cite{candes2011robust}, signal processing \cite{li2016low}, quantum state tomography \cite{rambach2021robust}, etc. In the vanilla setting, the task is to recover a rank $r$ positive semidefinite (PSD) matrix
$\truX \in \RR^{d\times d}$ ($r \ll d $) from a few corrupted linear measurements $\{(y_i,A_i)\}_{i=1}^n$ of the form
\begin{equation}\label{eq: observation model} 
y_i =\dotp{A_i, \truX} + s_i, \quad i=1,\dots,m,
\end{equation}
where $\Brac{s_i}_{i=1}^m$ are the corruptions. This model can be written succinctly as $y = \mathcal{A}(\truX) + s$, where $\mathcal{A}:\RR^{d\times d} \to \RR^{m}$ is the measurement operator. Recently, tremendous progress has been made on this problem, leveraging  advances on convex relaxation \cite{candes2011robust,li2016low}, nonconvex optimization \cite{li2020non,li2020nonconvex,charisopoulos2021low,tong2021low}, and high dimension probability \cite{vershynin2010introduction,tropp2012user,vershynin2018high,wainwright2019high}. However, several practical challenges remain:
\begin{itemize}
    \item \textbf{Nonconvex and nonsmooth optimization for sparse corruptions.} For computational efficiency, we often exploit the low-rank structure via the factorization $X = FF^\top$ with $F\in \bb R^{d \times r}$, and solve the problem with respect to $F$ via nonconvex optimization. This problem is challenging since the corruption can be arbitrary. Fortunately, \nct{the corruptions $s_i$ are \emph{sparse}} \cite{candes2011robust}. This motivates us to consider the $\ell_1$ loss function and lead to a nonsmooth nonconvex problem.
    \item \textbf{Unknown rank and overspecification.} Another major challenge in practice is that the \emph{exact} true rank $r$ is usually \emph{unknown} apriori.\footnote{There are very few exceptions, such as the case of robust phase retrieval with $r=1$ \cite{duchi2019solving}.} To deal with this issue, researchers often solve the problem in \emph{rank overspecified regime}, using a factor variable $F \in \bb R^{d \times k}$ with a larger inner dimension $k > r$. Most existing analysis relies on the $k$-th eigenvalue of $\truX$ being nonzero and is inapplicable in this setting. 
    Recent work studied the problem either in some special case (e.g., rank one) \cite{ma2021implicit} or under unrealistic communitability conditions \cite{you2020robust}.  
\end{itemize}

\paragraph{Our approach and contributions.} 
To deal with these challenges, we consider the following nonconvex and nonsmooth formulation with overspecified rank $k \geq r$,
\begin{equation}\label{opt: maintextmain}
	\min_{F\colon F \in \RR^{d\times k}} f(F) := \frac{1}{2m}\sum_{i=1}^{m} \abs{\dotp{A_i, FF^\top} -y_i} = \frac{1}{2m} \norm{ \mathcal{A}(FF^\top) - y }_1,
\end{equation}
and solve the problem using a subgradient method with \emph{adaptive} stepsize. We summarize our contributions as follows.

\begin{itemize}
    \item \textbf{Exact recovery with rank overspecification}. In \Cref{thm: exact recovery under models}, we show that \nct{under Gaussian measurements and sparse corruptions,}  with $m=\tilde{\mathcal{O}}(dk^3)$ samples,\footnote{Here $\tilde{\mathcal{O}}$ hides the condition number of $\truX$ and other log factors, \nct{and we require the corruption rate, the fraction of corrupted measurements, to be sufficient small.}} even with overspecified rank $k>r$, our subgradient method with spectral initialization converges to the {ground truth} $\truX$ \emph{exactly}. When $k>r$, the convergence rate is $\mathcal{O}(\frac{1}{t})$ in the operator norm error $\norm{F_tF_t^\top-\truX}$. When $k=r$, the convergence rate boosts to a linear (geometric) rate.
    \item \textbf{Convergence under a regularity condition.} Our convergence result is established under a regularity condition termed \emph{restricted direction preserving property} (RDPP); see Definition \ref{defn: RDPP}. RDPP ensures that the \emph{direction} of the finite-sample subgradient stays close to its population version along the algorithm's iteration path. The key to our proof is to use this property to leverage a recent result on rank overspecification under \emph{smooth} formulations \cite{zhuo2021computational}. Moreover, we show that RDPP holds with high probability under arbitrary sparse corruptions and Gaussian measurement matrices. 
    \item 
    \textbf{Adaptive stepsize and implicit regularization}. As shown in \cite{you2020robust,ma2021implicit}, for overparametrized robust matrix recovery problems, subgradient methods with constant stepsizes suffer from overfitting, even when initialized near the origin. In contrast, numerical experiments demonstrate that our theory-inspired diminishing stepsize rule has an implicit regularization effect that prevents overfitting. Our stepsize rule is easy to implement and almost tuning-free, and performs better than other non-adaptive stepsize rules (e.g., sublinear or geometric decaying stepsizes). We demonstrate similar good performance of our method (\nct{with a fixed decaying stepsize}) in image recovery with a robust deep image prior (DIP) \cite{ulyanov2018deep}.
\end{itemize}

\paragraph{Related work.} The nonconvex, factorized formulation~\eqref{opt: maintextmain} is computationally more efficient than  convex relaxation approaches \cite{recht2010guaranteed,candes2011robust,chen2013low,liu2019recovery}, as it involves a smaller size optimization variable and only requires one partial singular value decomposition (SVD) in the initialization phase. In the corruption-free ($s_i \equiv 0$) and exact rank ($k=r$) setting, provably efficient methods have been studied for the \emph{smooth}, least squares version of \eqref{opt: maintextmain} \cite{sun2016guaranteed,bhojanapalli2016global,ge2017no,zhu2018global}; in this case, the loss function has no spurious local minimum and is locally strongly convex\cite{chi2019nonconvex,zhang2020symmetry}. For smooth formulations with rank overspecification ($k\ge r$ or even $k = d$), gradient descent initialized near the origin is shown to \emph{approximately} recover the ground truth matrix \cite{gunasekar2018implicit,li2018algorithmic}.  

When the rank $r$ is known and measurement corruption is present, the work \cite{li2020non} considers the smooth formulation and proposes a gradient descent method with a robust gradient estimator using median truncation. The nonsmooth formulation \eqref{opt: maintextmain} is considered in a recent line of work \cite{li2020nonconvex,charisopoulos2021low,tong2021low}, again in the known rank setting ($k=r$), where they show that the loss function satisfies certain sharpness \cite{burke1993weak} properties and hence subgradient methods converge locally at a linear rate. 
\nct{The number of samples required in these works for exact recovery of $\truX$ is $\mathcal{O}(dr^2)$ when the corruption rate is sufficiently small. The $\mathcal{O}(r^2)$ dependence is due to their initialization procedure.}
We emphasize that the above results crucially rely on knowing and using the exact true rank.
When the rank is unknown, the work \cite{you2020robust} considers a doubly over-parameterized formulation but requires a stringent, commutability assumptions on the measurement matrices; moreover, their analysis concerns continuous dynamic of gradient descent and is hard to generalize to discrete case. Our results only impose standard measurement assumptions and apply to the usual discrete dynamic of subgradient methods.

\begin{table}[]
    \centering
    \begin{tabular}{c|c|c|c}
    \toprule
    &  \cite{li2020non,li2020nonconvex,charisopoulos2021low,tong2021low} & \cite{ma2021implicit} & Ours\\
    \hline
       Rank Specification  &  $k=r$ & $k\geq  r =1$  &  $k\geq r$ \\
         \hline 
      Exact Recovery & $\checkmark$  & \ding{55} & $\checkmark$ \\ 
      \hline 
      \nct{Sample Complexity}  & \nct{$\tilde{\mathcal{O}}(dr^2)$} & 
        \nct{$\tilde{\mathcal{O}}\left(\frac{dk}{\delta^4}\right)$}
        & \nct{$\tilde{\mathcal{O}}(dk^3)$} \\
      \bottomrule
    \end{tabular}
    \caption{\textbf{Comparison of results on robust matrix sensing.}}
    \label{tab:comparison}
\end{table}

The work most related to ours is \cite{ma2021implicit}, which studies subgradient methods for the nonsmooth formulation \eqref{opt: maintextmain} in the rank-one case ($r=1$). Even in this special case, only approximate recovery is guaranteed (i.e., finding an $F$ with $\fnorm{FF^\top -\truX}\lesssim \delta \log (\frac{1}{\delta})$ for a fixed $\delta>0$). Additionally, \nct{the number of samples required in \cite{ma2021implicit} for this approximate recovery is $\tilde{\mathcal{O}}\left(\frac{dk}{\delta^4}\right)$ when the corruption rate is sufficiently small.} We note that the lack of exact recovery is inherent to their choice of stepsizes, and spectral initialization already guarantees $\fnorm{FF^\top -\truX}\lesssim \delta$ \cite[Equation (85)]{ma2021implicit}. In contrast, our result applies beyond rank-1 case and guarantees exact recovery thanks to the proposed adaptive diminishing stepsize rule. We summarize the above comparisons in \Cref{tab:comparison}.

\section{Models, regularity, and identifiability }\label{sec:models}
In this section, we formally describe our model and assumptions. We then introduce the restricted direction preserving property and establish a basic identifiability result under rank overspecification, both key to the convergence analysis of our algorithm given in \Cref{sec:alg} to follow.

\subsection{Observation and corruption models}\label{subsec:models}

In this work, we study subgradient methods for the robust low-rank matrix recovery problem \eqref{opt: maintextmain} under the observation model \eqref{eq: observation model}. Throughout this paper, we assume that the sensing matrices  $A_i\in \RR^{d\times d}, i=1,\dots,m$ are generated \emph{i.i.d.} from the Gaussian Orthogonal Ensemble (GOE).\footnote{Our results can be easily extended sensing matrices with i.i.d. standard normal entries since $X_\natural$ is symmetric.} That is, each $A_i$ is symmetric, with the diagonal entries being \emph{i.i.d.}\ standard Gaussian $N(0,1)$, and the upper triangular entries are \emph{i.i.d.}\ $N(0,\frac{1}{2})$ independent of the diagonal. We consider two models for the corruption $s = (s_1, \ldots, s_m )^\top $: (\emph{i}) arbitrary corruption (AC) and (\emph{ii}) random corruption (RC).

\begin{mdl}[AC model]\label{md: arbitraryCorruption}
Under the model AC($\truX,\mathcal{A},p,m$), the index set $S\subset\{1,\dots,m\}$ of nonzero entries of $s$ is uniformly random with size $\lfloor pm \rfloor$, independently of $\mathcal{A}$, and the values of the nonzero entries are arbitrary.  
\end{mdl} 

\begin{mdl}[RC model]\label{md: randomCorruption}
Under the model RC($\truX,\mathcal{A},p,m$), each $s_i$, $i=1,\dots,m$, is a random variable, independent of $\mathcal{A}$ and each other, with probability $p\in(0,1)$ being nonzero.
\yc{Does $s_i$ need to be mean zero? Can the mean depend on anything? Arbitrary mean is equivalent to the AC model below.} \lj{The mean can be non-zero, but they can't be chosen based on $y_i$ or $A_i$ because this will make $s_i$ depend on $A_i$. }
\end{mdl}

In the AC model, the values of the nonzero entries of $s$ can be chosen in a coordinated and potentially {adversarial} manner dependent on $\{A_i\}$ and $\truX$, whereas their locations are independent. Therefore, this model is more general than those in \cite{candes2011robust,chen2020bridging}, which assume random signs.\footnote{Our results can be extended to the even more general setting with adversarial locations \cite{li2020non,li2020nonconvex,charisopoulos2021low,tong2021low}, with an extra $\textup{poly}(k,\kappa)$ factor in the requirement of the corruption rate $p$ compared to the one we present in \Cref{sec:alg}.} The RC model puts no restriction on the random variables $\{s_i\}$ except for independence, and hence is more general than the model in \cite{ma2021implicit}. In particular, the distributions of $\{s_i\}$ can be heavy-tailed, non-identical and have non-zero means (where the means cannot depend on $\{A_i\}$). Under the RC model, we can allow the corruption rate $p$ to be close to $\frac{1}{2}$ and still achieve exact recovery (see \Cref{thm: exact recovery under models}).
\yc{The AC model above does NOT allow arbitrary locations, right?}\ld{Right, location is not arbitrary} \lj{Allowing arbitrary location is doable, but it will make the assumption on $p$ worse. We picked the better one to present.}

\subsection{Restricted direction preserving property (RDPP)}

When the true rank $r$ is known, existing work \cite{li2020nonconvex,charisopoulos2021low} on subgradient methods typically makes use of the so-called $\ell_1/\ell_2$-restricted isometry property (RIP), which ensures that the function $g(X)= \frac{1}{m} \norm{ \mathcal{A}(X) - y }_1$ concentrates uniformly around its expectation. Note that the factorized objective $f$ in \eqref{opt: maintextmain} is related to $g$ by $f(F)=\frac{1}{2}g(FF^\top)$. 
With $\ell_1/\ell_2$-RIP, existing work establishes the sharpness of $g$ when $k=r$, i.e., linear growth around the minimum \cite[Proposition 2]{li2020nonconvex}. In the rank-overspecification setting, however, $g$ is \emph{not} sharp, as it grows more slowly than linear in the direction of superfluous eigenvectors of $F$. 
We address this challenge by directly analyzing the subgradient dynamics, detailed in Section \ref{sec:alg}, and establishing concentration of the \emph{subdifferential} of $g$ rather than $g$ itself.

Recall the regular subdifferential of $g(X)$ \cite{clarke2008nonsmooth} is 
\begin{align}\label{eqn: cDsubdifferential}
  \cD(X) := \frac{1}{m}\sum_{i=1}^{m}\overline{\sign}(\dotp{A_i,X} -s_i)A_i,\quad \text{and}\quad  \overline{{\sign}}(x) \;:=\; 
  \begin{cases} \{-1\}, & x<0 \\
    [-1,1], &  x=0 \\
    \{1\}, & x>0
    \end{cases}. 
\end{align}
Since our goal is to study the convergence of the subgradient method, which utilizes only one member $D(X)$ and will be illustrated in \Cref{subsec:subgradient}, it is enough to consider this member $D(X)$ by setting $\sign(x) =\begin{cases} -1, & x<0 \\
    1, & x\geq 0
    \end{cases}$. We are now ready to state the RDPP condition. 
%

\begin{definition}[RDPP]\label{defn: RDPP}
	The measurements and corruption $\{(A_i,s_i)\}_{i=1}^m$ are said to satisfy  restricted direction preserving property with parameters $(k',\delta)$ and a scaling function $\psi: \RR^{d\times d} \rightarrow \RR$ if for every symmetric matrix $X \in \RR^{d\times d}$ with rank no more than $k'$, we have 
	\begin{equation}\label{eq: RDPP}
	D(X):=\frac{1}{m}\sum_{i=1}^{m}{\sign}(\dotp{A_i,X} -s_i)A_i,\quad \text{and}\quad
		\norm{D(X) - \psi(X) \frac{X}{\fnorm{X}}} \le \delta.
	\end{equation}
\end{definition}

Here, the scaling function accounts for the fact that the expectation is not exactly $\frac{X}{\fnorm{X}}$, e.g., when no corruption, $\nabla \EE(g(X))=\sqrt{\frac{2}{\pi}}\frac{X}{\fnorm{X}}$.
%
How does RDPP relate to the $\ell_1/\ell_2$-RIP? In the absence of corruption, as shown in \Cref{sec: RDPP and l1l2RIP}, RDPP implies $\ell_1/\ell_2$-RIP for $\delta \lesssim \frac{1}{\sqrt{r}}$. 
 Conversely, we know if $\ell_1/\ell_2$-RIP condition holds for all rank $d$ matrices, then RDPP holds when $p=0$ \cite[Proposition 9]{ma2021implicit}.  

Our RDPP is inspired by the condition \emph{sign-RIP} introduced in
the recent work \cite[Definition 1]{ma2021implicit}.
The two conditions are similar, 
with the difference that RDPP is based on the operator norm (see the inequality in equation~\eqref{eq: RDPP}), whereas sign-RIP uses the Frobenius norm, hence a more stringent condition. In fact, the numerical results in Table~\ref{tab: RDPP} suggest that such a Frobenius norm bound \emph{cannot} hold if $m$ is on the order $\mathcal{O}(dr)$ even for a fixed rank-1 $\truX$---in our experiments $\fnorm{D(\truX) - \sqrt{\frac{2}{\pi}} \frac{\truX}{\fnorm{\truX}}}$ is always larger than $1$.%
\footnote{\nct{To the best of our knowlegde, the proof of \cite[Proposition 1]{ma2021implicit} actually establishes a bound on
the $\ell_2$ norm of top $r$ singular values, called partial Frobenius norm, rather than on Frobenius norm (i.e., the $\ell_2$ norm of all singular values). This is consistent with our empirical observations in Table~\ref{tab: RDPP}.}}
\footnote{\nct{After the first version of this paper was posted, the work \cite{ma2021implicitV3} became available, which is an updated version of \cite{ma2021implicit}. In this updated paper, the definition and associated results for sign-RIP (see \cite[Definition 2 and Theorem 1]{ma2021implicitV3}) use the rank-$r$ partial Frobenius norm, instead of the Frobenius norm as used in the original \cite[Definition 1]{ma2021implicit}.}
}
%
%
In contrast, standard Gaussian/GOE sensing matrices satisfy RDPP for any rank $k'\geq 1$, and it suffices for our analysis. Indeed, we can show the following.

\begin{table}[]
    \centering
    \begin{tabular}{c|c|c|c|c}
    \toprule
    $d$ & $d=50$ & $d=100$ & $d=200$ & $d=300$\\
    \hline
       RDPP: $\norm{D(\truX) - \sqrt{\frac{2}{\pi}} \frac{\truX}{\fnorm{\truX}}}$  &  $0.62/0.27$ & $0.62/0.28$  &  $0.63/0.28$ & $0.63/0.28$ \\
         \hline 
    sign-RIP \cite{ma2021implicit}: $\fnorm{D(\truX) - \sqrt{\frac{2}{\pi}} \frac{\truX}{\fnorm{\truX}}}$ & $2.26/1.01$  & $3.17/1.42$  & 
    $4.48/2.00$ & $5.49/2.45$\\ 
      \bottomrule
    \end{tabular}
    \smallskip 
    
    \caption{
    \textbf{Comparison of RDPP and sign-RIP} for $d\in\{50,100,200,300\}$ and $r\in\{1,5\}$. For each entry with two numbers, the left corresponds to $r=1$ and the right corresponds to $r=5$. 
    We set $p=0$ and $m=5dr$. Note sign-RIP is never satisfied as  the above entries for sign-RIP has value larger than or equal to $1$. }
    \label{tab: RDPP}
\end{table}

\begin{proposition}\label{prop: main RDPP for two models}
  Let $m \gtrsim \frac{dk' \left(\log(\frac{1}{\delta}) \vee 1\right)}{\delta^4}$.  For Model  \ref{md: arbitraryCorruption},  $(k',\delta+ 3\sqrt{\frac{dp}{m}} +3p)$-RDPP holds with a scaling function $\psi(X)=\sqrt{\frac{2}{\pi}}$ with probability at least $1-\exp(-(pm+d)) - \exp(-c'm\delta^4)$.  For Model \ref{md: randomCorruption},  $(k',\delta)$-RDPP holds with a scaling function $\psi(X) = \frac{1}{m}\sum_{i=1}^{m}\sqrt{\frac{2}{\pi}}\left(1-p + p\EE_{s_i}\left[\exp(-\frac{s_i^2}{2\fnorm{X}^2})\right]\right)$ with probability at least $1-Ce^{-cm\delta^4}$. Here $c,c',C$ are universal constants.
\end{proposition}

Our proof of Proposition~\ref{prop: main RDPP for two models} \nct{follows that of \cite[Proposition 5]{ma2021implicit} with minor modifications accounting for our corruption models and the definition of RDPP. For completeness we provide the proof in Appendix \ref{sec: Proof of RDPP random}}. The operator norm bound in the definition~\eqref{eq: RDPP} of RDPP turns out to be sufficient for our purpose in approximating the population subgradient of $f$ for our factor $F$. 

\subsection{Identifiability of the ground truth $\truX$}

If the exact rank is given (i.e., $(k=r)$), it is known that our objective function \eqref{opt: maintextmain} 
indeed identifies $\truX$, in the sense that any global optimal solution $F_\star$ coincides with $\truX = F_\star F_\star^\top$\cite{li2020nonconvex}. When $k >r$, one may suspect that there might be an idenitfiability or overfitting issue due to rank oversepcification and the presence of corruptions, i.e., there might be a factor $F$ that has better objective value than $F_\star$, where $F_\star F_\star^\top = X_\natural$. However, so long as the number of samples are sufficiently large $m=\Omega(dk)$, we show that the identifiability of $\truX$ is ensured. We defer the proof to \Cref{sec: proofofidentifiability}.

\begin{thm} \label{thm: identifiability} 
Fix any $p< \frac{1}{2}$. Under either Model \ref{md: arbitraryCorruption} or Model  \ref{md: randomCorruption} with corruption rate $p$, if $m\geq cdk$, then  with probability at least $1-e^{-c_1m}$, any global optimizer $F_\star$ of \eqref{opt: maintextmain}  satisfies $F_\star F_\star ^\top = \truX$. Here, the constants $c,c_1>0$ only depend on $p$.
\yc{Our main theorem allows $p\le 1/2 - \epsilon$ under the RC model. Do we have the landscape result for this $p$?}\ld{actually both models are ok.}
\end{thm}

Here, overfitting does not occur since sample size $m$ matches the degree of freedom of our formulation $\mathcal{O}(dk)$, making it possible to guarantee success of the algorithm described in Section \ref{sec:alg}. On the other hand, if we are in the heavily overparameterized regime with $m=\mathcal{O}(dr) \ll dk$  (e.g., $k=d$),\footnote{A few authors \cite{ma2021implicit,zhuo2021computational} refer to the setting with $k>r$ and $m \asymp dk $ as ``overparametrization'' since $k$ is larger than necessary. Others reserve the term for the setting $m<dk$, i.e., the number of parameters exceeds the sample size. We generally refer to the former setting as ``(rank-)overspecifcation'' and the latter as ``overparametrization''.} the above identifiability result should not be expected to hold. Perhaps surprisingly, in \Cref{sec:exp} we empirically show that the proposed method works even in this heavy overparametrization regime with $k=d$. This is in contrast to the overfitting behavior of subgradient methods with constant stepsizes observed in \cite{you2020robust}, and it indicates an implicit regularization effect of our adaptive stepsize rule.

\vspace{-0.05in}
\section{Main results and analysis}\label{sec:alg}
In this section, we first introduce the subgradient method with implementation details. We then show our main convergence results, that subgradient method with diminishing stepsizes converge to the exact solutions even under the setting of overspecified rank. Finally, we sketch the high-level ideas and intuitions of our analysis. All the technical details are postponed to the appendices.

\subsection{Subgradient method with diminishing stepsizes}\label{subsec:subgradient}

A natural idea of optimizing the problem \eqref{opt: maintextmain} is to deploy the subgradient method starting from some initialization  $F_0\in \RR^{d\times k}$, and iterate\footnote{Recall the definition of the subgradient $D(X)=\frac{1}{m}\sum_{i=1}^{m}{\sign}(\dotp{A_i,X} -s_i)A_i$ in \eqref{eq: RDPP} for the objective, $g(X)= \frac{1}{2m} \norm{ \mathcal{A}(X) - y }_1$ in the original space $X\in \RR^{d\times d}$.}
\begin{align}\label{eqn: mainalgorithm}
	F_{t+1}= F_t -\eta_t g_t, \quad g_t = D(FF^\top -\truX)F, \quad,t=0,1,2,\dots. 
\end{align}
 with certain stepsize rule $\eta_t\geq 0$. The vector $g_t$ indeed belongs to the regular subdifferential of $\partial f(F_t) = \cD(F_tF_t^\top -\truX)F_t$ by chain rule (recall $\cD$ in \eqref{eqn: cDsubdifferential}).
 
 %
%
%
For optimizing a nonconvex and nonsmooth problem with subgradient methods, it should be noted that the choice of initialization and stepsize is critical for algorithmic success. As observed in \Cref{sec:exp-matrix}, a fixed stepsize rule, i.e., stepsize determined before seeing iterates, such as geometric or sublinear decaying stepsize, cannot perform uniformly well in various rank specification settings, e.g., $k=r,r\leq k=\mathcal{O}(r), k=d$.  Hence it is important to have an adaptive stepsize rule. 

\paragraph{Initialization} We initialize $F_0$ by a spectral method similar to \cite{ma2021implicit}, yet we do not require $F_0F_0^\top$ to be close to $\truX$ even in operator norm. We discuss the initialization in more detail in \Cref{subsec:init}.

\paragraph{Stepsize choice} To obtain exact recovery, we deploy a simple and adaptive diminishing stepsize rule for $\eta_t$, which is inspired by our analysis and is easy to tune. As we shall see in Theorem \ref{eqn: mainalgorithm} and its proof intuition in \Cref{sec: intuition}, the main requirement for stepsize is that it scales with $\fnorm{FF^\top -\truX}$ for an iterate $F$, and hence diminishing and adaptive as $FF^\top \rightarrow \truX$. To estimate $\fnorm{FF^\top -\truX}$ when corruption exists, we  define an operator  $\tau_{\mathcal{A},y}(F):\RR^{d\times k}\rightarrow \RR$ and the corresponding stepsize: 
\begin{equation} \label{eqn: stepsizerule}
 \tau_{\mathcal{A},y}(F) =\xi_{\frac{1}{2}}\left(\{|\dotp{A_i,FF^\top}-y_i|\}_{i=1}^{m}\right), \quad \text{and} \quad \eta_t = C_\eta \tau_{\mathcal{A},y}(F_t),
\end{equation}
where $\xi_{\frac{1}{2}}(\cdot)$ is the median operator, and $C_\eta >0$ is a free parameter to be chosen.
As without corruption \footnote{\nct{Here we consider a simple setting that the matrix $F\in \RR^{d\times k}$ is fixed and independent of $A_1$, and we take the expectation with respect to the randomness in $A_1$ only.}} $\EE{|\dotp{A_1,FF^\top}-y_1|} = \sqrt{\frac{2}{\pi}} \fnorm{FF^\top -\truX}$, we expect $\tau_{\mathcal{A},y}(F)$ to estimate $\fnorm{FF^\top -\truX}$ up to some scalar when corruption exists since the median is robust to corruptions.

\subsection{Main theoretical guarantees}

First, we present our main convergence result, that ensures that the iterate $F_tF_t^\top$ in subgradient method \eqref{eqn: mainalgorithm} converges to the rank $r$ ground-truth $\truX$, based on conditions on initialization, stepsize, and RDPP. Let the condition number of $\truX$ be $\kappa := \frac{\sigma_1(\truX)}{\sigma_r(\truX)}$, our result can be stated as follows.

\begin{thm}\label{lem: mainlemma}
Suppose the following conditions hold: 
\begin{itemize}
      \item[(i)] The initialization $F_0$ satisfies
    \begin{align}\label{eqn: initialconditionmainlemma}
		\|  F_0 F_0^\top - c^*\truX \| \;\leq\; c_0 \sigma_r/\kappa.
	\end{align}
	with some constant $c^* \in [\epsilon, \frac{1}{\epsilon}]$ and $c_0 = \tilde c_0\epsilon$ for some sufficiently small $\tilde c_0$, which only depends on $c_1$ below and $0<\epsilon <1$ is a fixed parameter. 
	Moreover, all the constants $c_i,i\geq 3$ in the following depend only on $\epsilon$ and $c_1$.  \lj{Does this $c_1$ look strange?}
 
    \item[(ii)] The stepsize satisfies $0<c_1\frac{\sigma_r}{\sigma_1^2}\leq \frac{\eta_t}{\fnorm{F_tF_t^\top -\truX}} \leq  c_2\frac{\sigma_r}{\sigma_1^2}$ for some small numerical constants $c_1 < c_2 \le 0.01$ and all $t\ge 0$. 
    \qq{consistent} 
    \item[(iii)] The $(r+k,\delta)$-RDPP holds for $\{A_i,s_i\}_{i=1}^m$ with $\delta \leq \frac{c_3}{\kappa^3\sqrt{k}}$ and a scaling function $\psi\in \big[\sqrt{\frac{1}{2\pi}}, \sqrt{\frac{2}{\pi}}\big].$
\end{itemize}
Define $\mathcal{T} = c_4\kappa^2 \log \kappa$. 
Then we have a sublinear convergence in the sense that for any $t \ge 0$,
\[
\norm{F_{t+\mathcal{T}}F_{t+\mathcal{T}}^\top -\truX} \leq c_5\sigma_1\frac{\kappa}{\kappa^3+t}.
\]
Moreover, if $k=r$, then under the same set of condition, we have convergence at a linear rate
\[
\norm{F_{t+\mathcal{T}}F_{t+\mathcal{T}}^\top -\truX}  \leq \frac{c_6\sigma_r}{\kappa}\left(1-\frac{c_7}{\kappa^2}\right)^t, \quad t\ge 0.
\]
\end{thm} 

\lj{Indeed, all the constants $c_i,i\geq2$ depend on $\epsilon$ and $c_1$. Once $\epsilon$ and $c_1$ are fixed, all the other constants can be determined.}

We sketch the high-level ideas of the proof in \Cref{sec: intuition}, and postpone all the details to \Cref{sec: analysisOfMainTheorem}. 
In the last, we show that the conditions of \Cref{lem: mainlemma} can be met with our model assumptions in \Cref{sec:models}, our initialization scheme in \Cref{subsec:init}, and our stepsize choice \eqref{eqn: stepsizerule}, with sufficient number of samples $m$ and proper sparsity level $p$ of the corruption $s$. Define $\theta_{\epsilon}$ to be the left quantile of the folded normal distribution $|Z|$, $Z\sim N(0,1)$, i.e., $\PP(Z\leq\theta_{\epsilon})=\epsilon$. The quantile function is used for the stepsize condition.  For both RC and AC models, we now state our results as follows. 

\begin{thm}\label{thm: exact recovery under models}
 Suppose under Model \ref{md: arbitraryCorruption}, we have $m\geq c_1 dk^3\kappa^{12}(\log \kappa + \log k)\log d$ and  $p\le \frac{c_2}{\kappa^3 \sqrt{k}}$ for some universal constants $c_1,c_2$. On the other hand, suppose under Model \ref{md: randomCorruption}, we have fixed $p<\frac{1}{2}$ and $m\geq c_3 dk^3\kappa^{12}(\log \kappa + \log k)\log d$ for some $c_3>0$ depending only on $p$. Then for both models, with probability at least $1- c_4 \exp(-c_5 \frac{m}{\kappa^{12} k^2})-\exp(-(pm+d))$ for some universal constants $c_4,c_5$, our subgradient method \eqref{eqn: mainalgorithm} 
 with the initialization in Algorithm \ref{alg: initialization}, 
 and the adaptive stepsize choice \eqref{eqn: stepsizerule} with  $C_\eta \in [ \frac{c_6}{\theta_{1-\frac{0.5-p}{3}}}\frac{\sigma_r}{\sigma_1^2}, \frac{c_7}{\theta_{1-\frac{0.5-p}{3}}}\frac{\sigma_r}{\sigma_1^2},]$ with some universal $c_6,c_7\leq 0.001$, converges as stated in \Cref{lem: mainlemma}.
\end{thm}

The proof follows by verifying the three conditions in \Cref{lem: mainlemma}. The initialization scheme is verified in \Cref{prop: initialization for models} in Section \ref{subsec:init}. The RDPP is verified in \Cref{prop: main RDPP for two models}. We defer the verification of the validity of our stepsize rule \eqref{eqn: stepsizerule} to \Cref{sec: Choice of Stepsize}.

Note that the initialization condition in \Cref{lem: mainlemma} is weaker than those in the literature \cite{li2020non,li2020nonconvex}, which requires $F_0F_0^\top$ to be close to $\truX$. Instead, we allow a constant multiple $c^*$ in the initialization condition \eqref{eqn: initialconditionmainlemma}. This means we only require $F_0F_0^\top$ to be accurate in the direction of $\truX$ and allows a crude estimate of the scale $\fnorm{\truX}$. This in turn allows us to have $p$ close to $\frac{1}{2}$ in \Cref{thm: exact recovery under models} for the RC Model.\footnote{The work 
\cite{ma2021implicit} actually allows for $p$ to be close to $1$ under a simpler RC model. The main reason is that they assume the scale information $\fnorm{\truX}$ is known a priori and their guarantee is only approximate $\fnorm{FF^\top -\truX}\lesssim \delta\log \delta$ for a rank-$1$ $\truX$. However, estimating $\fnorm{\truX}$ accurately requires $p$ to be close to $0$.} 
If we require $p \lesssim \frac{1}{\kappa \sqrt{r}}$ in both models, which is assumed in \cite{li2020non,li2020nonconvex}, then the sample complexity can be improved to $O(dk^3\kappa^{4}(\log \kappa + \log k)\log d)$ and the convergence rate can also be improved. See \Cref{sec: better convergence result} for more details.

\subsection{A sketch of analysis for \Cref{lem: mainlemma}}\label{sec: intuition}

In the following, we briefly sketch the high-level ideas of the proof for \Cref{lem: mainlemma}.

\paragraph{Closeness of the subgradient dynamics to its smooth counterpart.}
To get the main intuition of our analysis, for simplicity let us first consider the noiseless case, i.e., $s_i=0$ for all $i$. Now the objective $f(F)$ in \eqref{opt: maintextmain} has expectation 
 $\EE f(F)=\frac{1}{2}\EE\abs {
\dotp {A_1,FF^\top-\truX}}=\sqrt{\frac{1}{2\pi}}\fnorm{FF^\top -\truX}$.  Hence, when minimizing this expected loss, the subgradient update at $F_t$ with $F_tF_t^\top \neq \truX$ becomes
\begin{equation*}
 F_0\in \RR^{d\times k}; \quad 
	F_{t+1} = F_t -\eta_t \nabla \EE f(F) \, \mid_{F=F_t} = F_t -\eta_t \sqrt{\frac{2}{\pi}}\frac{\left(F_tF_t^\top -\truX \right)}{\fnorm{F_tF_t^\top -\truX}}F_t, \;t=0,1,2,\dots. 
\end{equation*}
If we always have $\eta_t\sqrt{\frac{2}{\pi}} = \gamma\fnorm{F_tF_t^\top -\truX}$ for some $\gamma>0$, then the above iteration scheme becomes the gradient descent for the smooth problem $\min_{F\in \RR^{d\times k}} \fnorm{FF^\top-\truX}^2$:
\begin{equation}
 F_0\in \RR^{d\times k}; \qquad 
	F_{t+1}= F_t -\gamma\left(F_tF_t^\top -\truX \right)F_t, \;t=0,1,2,\dots. \label{eqn: gd}
\end{equation}
Recent work \cite{zhuo2021computational} showed that when no corruption presents, the gradient decent dynamic \eqref{eqn: gd} converges to the exact rank solution in the rank overspecified setting. Thus, under the robustness setting, the major intuition of our stepsize design for the subgradient methods is to ensure that subgradient dynamics \eqref{eqn: mainalgorithm} for finite samples mimic the gradient descent dynamic \eqref{eqn: gd}. To show this, we need $\eta_t g_t$ in \eqref{eqn: mainalgorithm} to approximate $\left( F_tF_t^\top -\truX\right)F_t$ up to some multiplicative constant $\gamma$, which is indeed ensured by the conditions \emph{(ii)} and \emph{(iii)} in \Cref{lem: mainlemma}. Detailed analysis appears in \Cref{sec: analysisOfMainTheorem}. 



\vspace{-.1in}
\paragraph{Three-stage analysis of the subgradient dynamics.} However, recall our condition \emph{(i)} in \Cref{lem: mainlemma} only ensures that the initial $F_0F_0^\top$ is close to the ground truth $\truX$ up to some constant scaling factor. Thus, different from the analysis for matrix sensing with no corruptions \cite{zhuo2021computational}, here we need to deal with the issue where the initial  $\norm{F_0F_0^\top-\truX}$ could be quite large. To deal with this challenge, \nct{inspired by the work ~\cite{zhuo2021computational}}, we introduce extra technicalities by decomposing the iterate $F_t$ into two parts, and conduct a three-stage analysis of the iterates. More specifically, since $\truX$ is rank $r$ and PSD, let its SVD be
\begin{equation}
	\truX = \begin{bmatrix}
		U & V
	\end{bmatrix}\begin{bmatrix}
		D_S^* & 0\\
		0   & 0
	\end{bmatrix}
	\begin{bmatrix}
		U & V
	\end{bmatrix}^\top,
\end{equation}
where $D_S^* \in \RR^{r \times r},\;U \in \RR^{d\times r}$, and $V \in \RR^{d \times (d-r)}$.
Here $U$ and $V$ are orthonormal matrices and complement to each other, i.e., $U^\top V = 0$.
Thus, for any $F_t \in \RR^{d\times r}$, we can always decompose it as 
\begin{equation}
	F_t = US_t + VT_t,
\end{equation} 
where $S_t = U^\top F_t \in \RR^{r \times k}$ stands for the signal term and $T_t = V^\top F_t \in \RR^{(d-r)\times k}$ stands for an error term. \qq{explain a bit about the meaning of S and F at high level?} Based on this, we show the following (we refer to Appendix \ref{sec: analysisOfMainTheorem} for the details.):
\begin{itemize}[leftmargin = 0.65 in]
    \item[Stage 1.] We characterize the evolution of the smallest singular value $\sigma_r(S_t)$, showing that it will increase above a certain threshold, so that $F_tF_t^\top$ is moving closer to the scaling of $\truX$.
    \item[Stage 2.] We show that the quantity $Q_t = \max\{\norm{S_t S_t^\top- D_S^*}, \norm{S_t T_t^\top}\}$ decrease geometrically. In the analysis, for first two stages, $T_t$ may not decrease, and the first two stages consist of the burn-in phase $\mathcal{T}$. \qq{could not understand it here, which stage the burn-in phase refer to? 1 or 2}
    \ld{both are burn-in}
    \item[Stage 3.] We control $E_t = \max\{\norm{S_tS_t^\top -D_S^*}, \norm{S_t T_t^\top}, \norm{T_tT_t^\top}\}$, showing that $E_t \rightarrow 0$ at a rate $O(1/t)$ due to the slower convergence of $\norm{T_tT_t^\top}$. This further implies that $\norm{F_tF_t^\top - \truX}$ converges to $0$ at a sublinear rate $O(1/t)$. For the case $k=r$, the convergence speeds up to linear due to a better control of $\norm{T_tT_t^\top}$.
\end{itemize}

\subsection{Initialization schemes}\label{subsec:init}

\begin{algorithm}[t]
		\caption{Spectral Initialization}
		\label{alg: initialization}
		\begin{algorithmic}
			\STATE {\bfseries Input:} the sensing matrix $\{A_i\}_{i=1}^m$ and observations $\{y_i\}_{i=1}^m$
			\STATE $\quad$   
			Compute the top-$k$ eigenpairs of the matrix $D \;=\;  \frac{1}{m} \sum_{i=1}^m \sign(y_i) A_i$. Collect the eigenvectors as $U\in \RR^{d\times k}$, and the eigenvalues as the diagonals of $\Sigma$ $\in \RR^{k\times k}$.
			\STATE $\quad$ Compute $B=U(\Sigma_+)^\frac{1}{2}$ where $\Sigma_+$ is the nonnegative part of $\Sigma$.
			\STATE $\quad$ Compute the median $\xi_{\frac{1}{2}}(\{|y_i|\}_{i=1}^m)$, and set 
			$\gamma  =\xi_{\frac{1}{2}}(\{|y_i|\}_{i=1}^m)/(\sqrt{2/\pi}\theta_{\frac{1}{2}})$. 
			\STATE {\bfseries Output:} $F_0=\sqrt{\gamma}B$.
		\end{algorithmic}
	\end{algorithm}
	
In this section, we introduce a spectral initialization method as presented in Algorithm \ref{alg: initialization}. The main intuition is that the cooked up matrix $D$ in Algorithm \ref{alg: initialization} has an expectation $\EE(D) = \psi(X)\frac{X}{\fnorm{X}}$, so that the obtained factor $BB^\top$ from $D$ should be pointed to the same direction as $\truX$ up to a scaling. We then multiply $B$ by a suitable quantity to produce an initialization ${F}_0$ on the same scale as $\truX$. 
More specifically, given that $\theta_\epsilon$ is the left quantile of folded normal distribution introduced before \Cref{thm: exact recovery under models}, we state our result as follows.

\begin{proposition}\label{prop: initialization for models} 
Let $F_0$ be the output of Algorithm \ref{alg: initialization}. Fix constant $c_0<0.1$. For model~\ref{md: arbitraryCorruption} with $p\le \frac{\tilde c_0}{\kappa^2 \sqrt{r}}$ where $\tilde c_0$ depends only on $c_0$, there exist constants $c_1,c_2,c_3$ depending only on $c_0$ such that whenever $m \ge c_1 dr \kappa^4 (\log\kappa+\log r) \log d$, we have
 	\begin{align}\label{eq: intiailizationinequalitytildeF}
		\|  F_0 F_0^\top - c^*\truX \| \;\leq\; c_0 \sigma_r/\kappa.
	\end{align}
with probability (w.p.) at least  $1- c_2\exp(-\frac{c_3m}{\kappa^4 r})-\exp(-(pm+d))$, where $c^* =1$. For Model \ref{md: randomCorruption} with a fixed $p<0.5$ and $m \ge c_4 dr \kappa^4 (\log\kappa+\log r)\log d $. Then \eqref{eq: intiailizationinequalitytildeF} also holds for 
	$c^* \in \Big[\sqrt{\frac{1}{2\pi}} \cdot \frac{\theta_{\frac{0.5-p}{3}}}{\theta_{\frac{1}{2}}}, \sqrt{\frac{2}{\pi}} \cdot  \frac{\theta_{1-\frac{0.5-p}{3}}}{\theta_{\frac{1}{2}}}\Big]$ 
	w.p. at least  $1- c_5\exp(-\frac{c_6m}{\kappa^4 r})$, where the constants $c_4,c_5,c_6$ depend only on $p$ and $c_0$.
\end{proposition}

The details of the proof can be found in \Cref{sec: initialization}. It should be noted in \eqref{eq: intiailizationinequalitytildeF} that our guarantee ensures $F_0F_0^\top$ is close to $\truX$ up to a scalar multiple $c^*$, instead of the usual guarantees on $\norm{F_0F_0^\top -\truX}$ \cite{li2020non,li2020nonconvex}. If $p\lesssim\frac{1}{\sqrt{r}\kappa}$, we obtain  $\norm{F_0F_0^\top -\truX} \leq c_0\sigma_r$ in \Cref{prop: initialization for arbitrary corruption} and the sample complexity $m$ improves to $\mathcal{O}(dr\kappa^2\log d(\log \kappa+\log r))$.

\section{Experiments}\label{sec:exp}
In this section, we provide numerical evidence to verify our theoretical discoveries on robust matrix recovery in the rank overspecified regime. First, for robust low-rank matrix recovery, \Cref{sec:exp-matrix} shows that the subgradient method with proposed diminishing stepsizes can $(i)$ exactly recover the underlying low-rank matrix from its linear measurements even in the presence of corruptions and the rank is overspecified, and $(ii)$ prevent overfitting and produce high-quality solution even in the highly over-parameterized regime $k = d$. Furthermore, we apply the strategy of diminishing stepsizes for robust image recovery with DIP \cite{ulyanov2018deep,heckel2019denoising,you2020robust} in \Cref{sec:exp-dip}, demonstrating its effectiveness in alleviating overfitting. The experiments are conducted with Matlab 2018a and Google Colab \cite{bisong2019google}.

\subsection{Robust recovery of low-rank matrices}
\label{sec:exp-matrix} 

\begin{figure}[]
    \centering
    \subfloat[\label{fig:matrix-exact} $k=r$]
        {\includegraphics[width=0.32\textwidth]{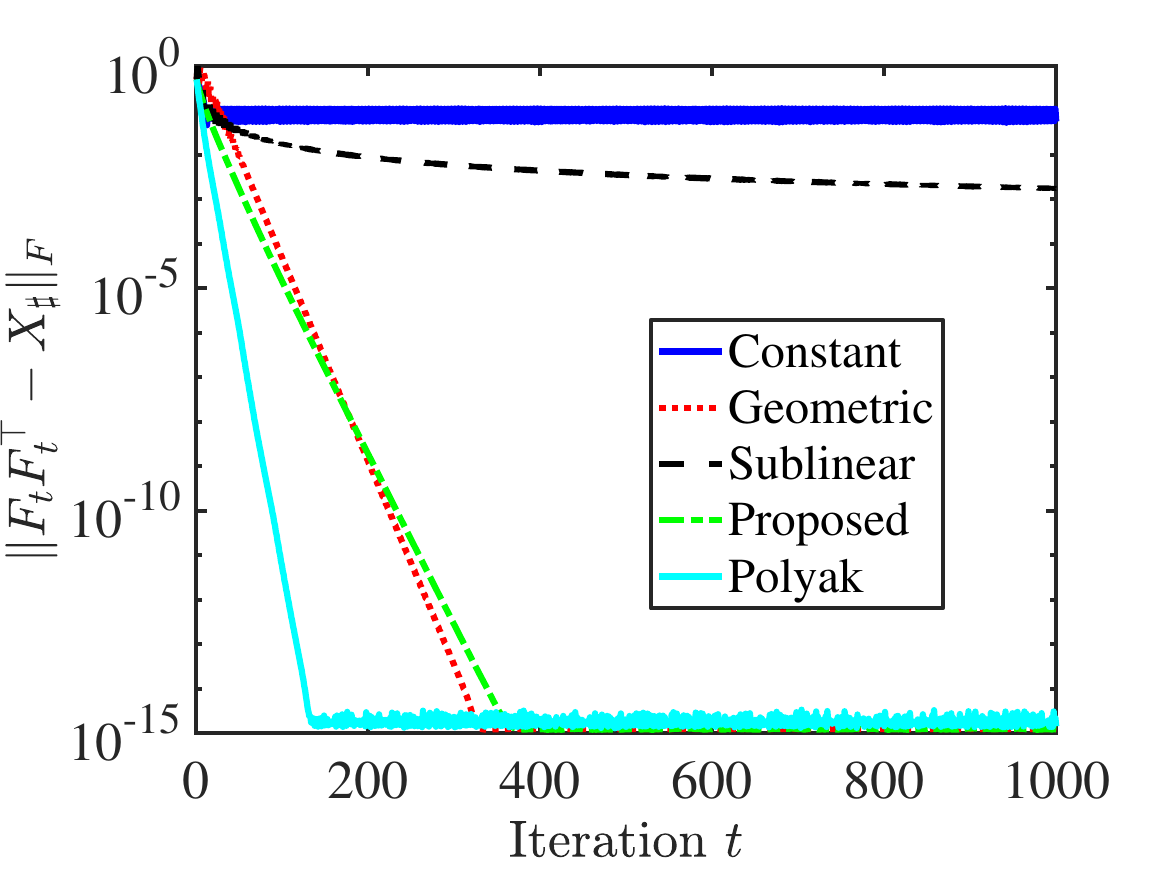}} \
      \subfloat[\label{fig:matrix-over}$k=2r$]  {\includegraphics[width=0.32\textwidth]{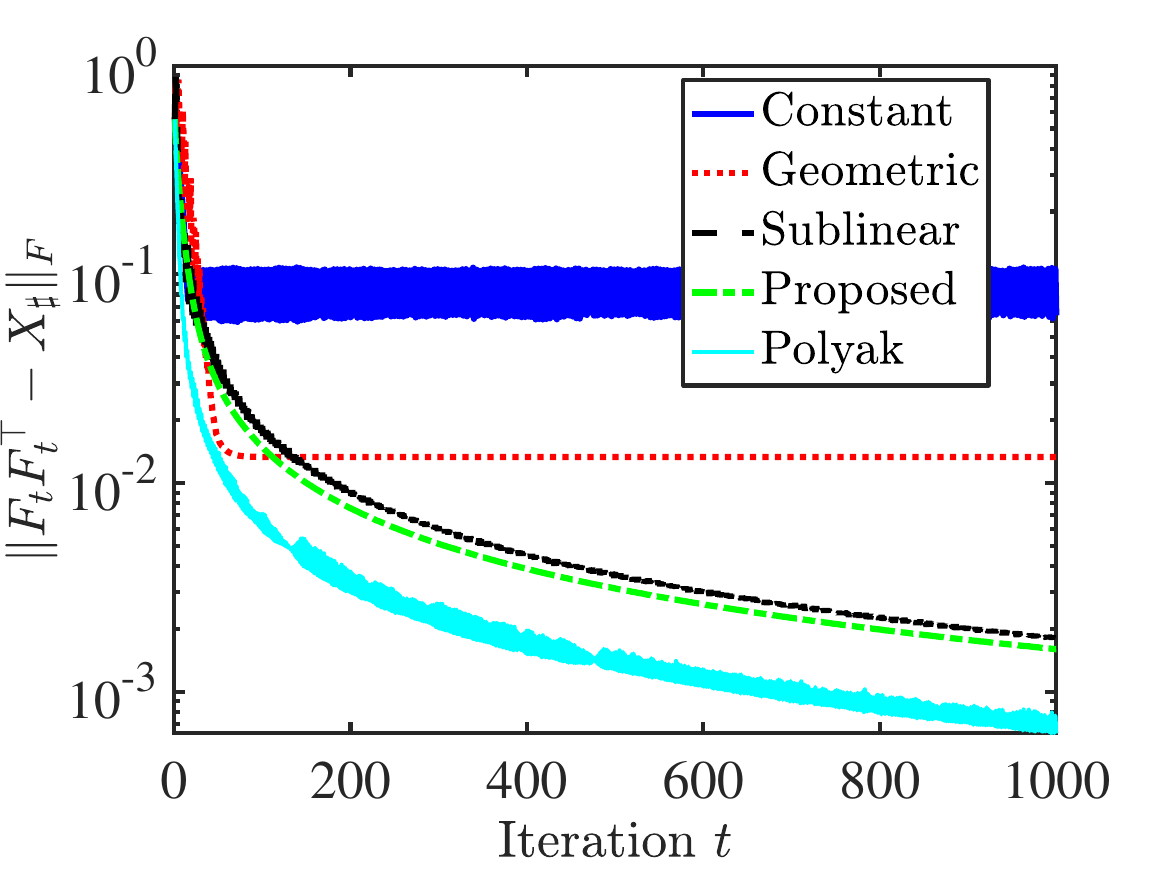}} \
        \subfloat[\label{fig:matrix-dim}$k=d$] {\includegraphics[width=0.32\textwidth]{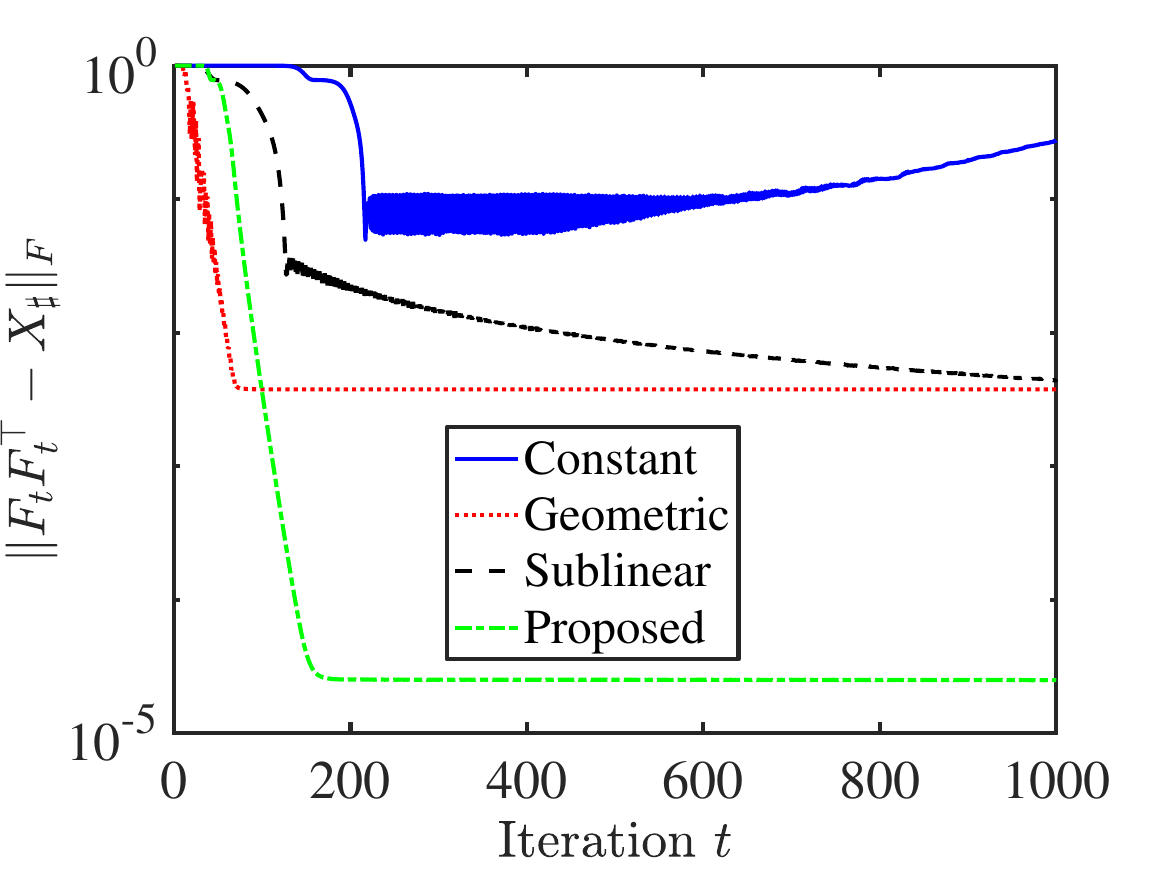}} 
    \caption{\footnotesize \textbf{Convergence of the subgradient method \eqref{eqn: mainalgorithm} with different stepsize rules} in the regimes of $(a)$ exact-parameterization $k = r$, $(b)$ rank overspecication $k = 2r$, and $(c)$ over-parameterization $k = d$. We use the following stepsizes: constant stepsizes $\eta_t = \eta_0$, geometric diminishing stepsizes $\eta_t = \eta_0 \cdot 0.9^t$, sublinear diminishing stepsizes $\eta_t = \eta_0/t$, the proposed stepsizes \eqref{eqn: stepsizerule}, and the Polyaks's stepsize.}
    \label{fig:matrix}
\end{figure}

\paragraph{Data generation and experiment setup.} For any given $d$ and $r$, we generate the ground truth $\truX = F_\natural F_\natural^\top$ with entries of $F_\natural\in\RR^{d\times r}$ \emph{i.i.d.} from standard normal distribution, and then normalize $\truX$ such that $\fnorm{\truX}=1$. Similarly, we generate the $m$ sensing matrices $A_1,\ldots,A_m$ to be GOE matrices discussed in \Cref{subsec:models}. For any given corruption ratio $p$, we generate the corruption vector $s\in\RR^m$ according to the AC model, by randomly selecting $\lfloor pm \rfloor$ locations to be nonzero and generating those entries from a \emph{i.i.d.} zero-mean Gaussian distribution with variance 100. We then generate the measurement according to \eqref{eq: observation model}, i.e., $y_i =\dotp{A_i, \truX} + s_i,  i=1,\dots,m$, and we set $n = 100$, $p = 0.2$, and $m = 10nr$. We run the subgradient methods for $10^3$ iterations using the following different stepsizes: $(i)$ the constant stepsize $\eta_t = 0.1$ for all $t\geq 0$, $(ii)$ sublinear diminishing stepsizes\footnote{For the two  diminising stepsizes, we set $\eta_0 = 2$ when $k = r$ or $ 2r$, and $\eta_0 = 5$ when $k = d$.} $\eta_t = \eta_0/t$, $(iii)$ geometrically diminishing stepsizes $\eta_t = \eta_0\cdot 0.9^t$, $(iv)$ the proposed stepsizes \eqref{eqn: stepsizerule} with $C_\eta = \frac{1}{2}$, and $(v)$ the Polyak stepsize rule \cite{polyak1969minimization} which is given by $\eta_t = \frac{f(F) - f^\star}{\fnorm{g_t}^2}$, where $f^\star$ is the optimal value of \eqref{opt: maintextmain} and $g_t$ is the subgradient in \eqref{eqn: mainalgorithm}. We use the spectral initialization illustrated in \Cref{subsec:init} to generate the initialization $F_0$.
\begin{figure}[]
\centering
  \includegraphics[width = 0.5 \textwidth]{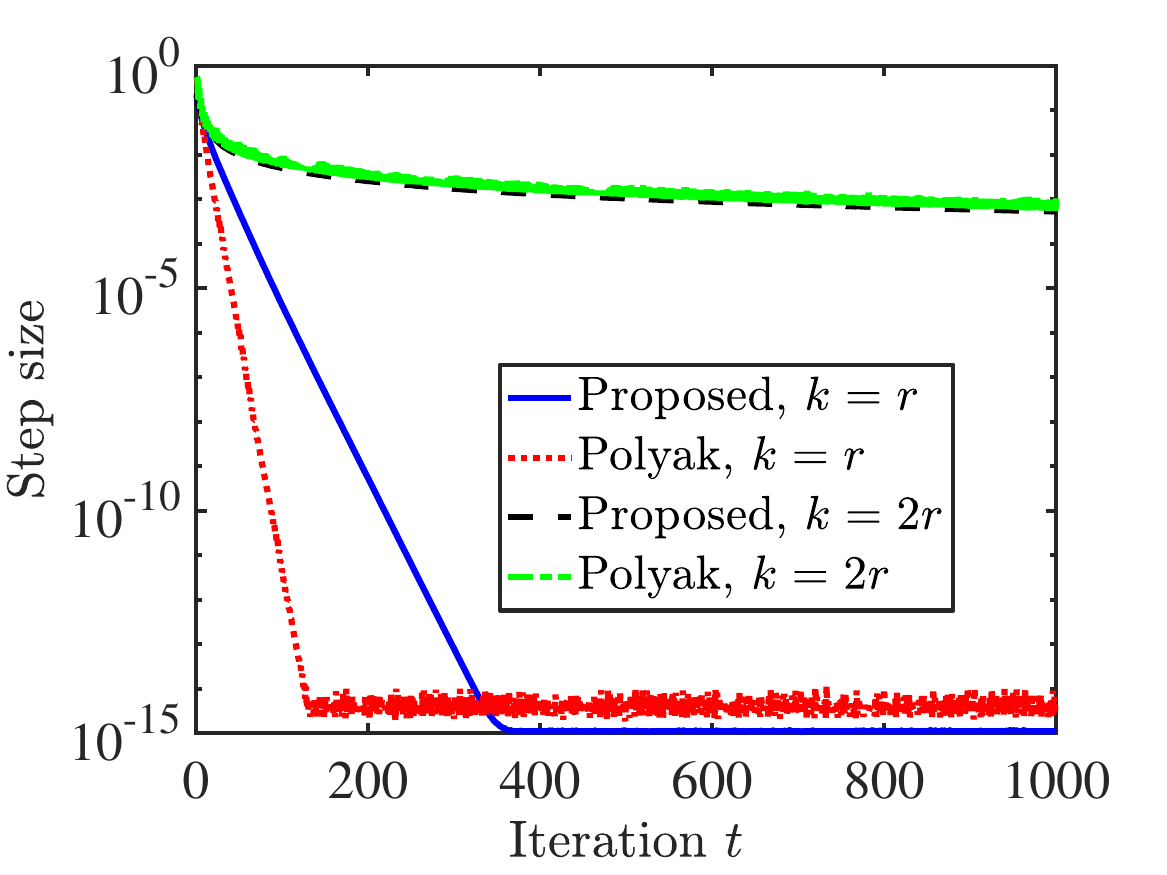}
  \caption{\footnotesize \textbf{Comparison of the proposed stepsize in \eqref{eqn: stepsizerule} and the Polyak stepsize} in the regimes of the exact-parameterization $k = r$ and the rank overspecification $k = 2r$.}
\label{fig:step-size}\end{figure}

\paragraph{Observations from the experimental results.} We run the subgradient method with different stepsize rules and different rank $k = r,\;2r,$ and $d$, and compute the reconstruction error $\fnorm{F_tF_t^\top - \truX}$ for the output of the algorithms. From \Cref{fig:matrix} and \Cref{fig:step-size}, our observations are the follows.

\begin{itemize}
    \item \emph{Exact rank case $k=r$.} As shown in \Cref{fig:matrix-exact}, we can see that \emph{(i)} the subgradient method with constant stepsizes does not converge to the target matrix, \emph{(ii)} using sublinear diminishing stepsizes results in a sublinear convergence rate, \emph{(iii)} geometrically diminishing or Polyak\footnote{As guaranteed by \Cref{thm: identifiability}, in this case, the optimal value $f^\star$ is achieved at $F_\natural$. 
    } stepsizes converge at a linear rate, consistent with the observation in \cite{li2020nonconvex}, and \emph{(iv)} the proposed stepsize rule also leads to convergence to the ground-truth at a linear rate, which is consistent with \Cref{lem: mainlemma}.
    \item \emph{Overspecified rank case $k=2r$.} As shown in \Cref{fig:matrix-over}, the subgradient method with any stepsize rule converges at most a sublinear rate, while constant or geometrically diminishing stepsizes result in poor solution accuracy. The proposed stepsize rule achieves on par convergence performance with the sublinear stepsize rule, which also demonstrates our \Cref{lem: mainlemma} on the sublinear convergence in the rank overspecified setting. As shown in \Cref{fig:step-size}, the proposed strategy \eqref{eqn: stepsizerule} gives stepsizes similar to Polyak's, and thus both resulting in a similar convergence rate (or the Polyak's is slightly better). However, the Polyak stepsize rule requires knowing the optimal value of the objective function, which is usually unknown a prior in practice.
    \item \emph{Overparameterized case $k=d$.} The results are shown in \Cref{fig:matrix-dim}.\footnote{We omit the performance of the Polyak stepsizes since in this case the optimal value $f^\star$ is not easy to compute; it may not be achieved at $F_\natural$ due to the overfitting issue. Our experiments indicate that the Polyak stepsizes do not perform well when computed by setting the optimal value as either $0$ or the value at $F_\natural$.}  In this case, inspired by \cite{li2018algorithmic,ma2021implicit}, we use an alternative tiny initialization $F_0$ by drawing its entries from \emph{i.i.d.} zero-mean Gaussian distribution with standard deviation $10^{-7}$. We first note that the constant stepsize rule results in overfitting. In contrast, the proposed diminishing stepsizes can prevent overfitting issues and find a very high-quality solution. We conjecture this is due to certain implicit regularization effects \cite{gunasekar2018implicit,li2018algorithmic,you2020robust} and we leave thorough theoretical analysis as future work.
\end{itemize}

\subsection{Robust recovery of natural images with deep image prior}\label{sec:exp-dip}

Finally, we conduct an exploratory experiment on robust image recovery with DIP \cite{ulyanov2018deep,you2020robust}. Here, the goal is to recover a natural image $\truX$ from its corrupted measurement $y = X_\natural + s$, where $s$ denotes noise corruptions. To achieve this, the DIP fits the observation by a highly overparameterized deep U-Net\footnote{Following \cite{ulyanov2018deep} (license obtained from \cite{DIP}), we use the same U-shaped architecture with skip connections, each layer containing a convolution, a nonlinear LeakyReLU and a batch normalization units,  Kaiming initialization for network parameters, and Adam optimizer \cite{kingma2014adam}. We set the network width as 192 as used in \cite{you2020robust}.
} $\phi(\Theta)$ by solving $\min_{\Theta} \|y - \phi(\Theta)\|_{\square}$ \cite{ronneberger2015u}, where $\Theta$ denotes network parameters and the norm $\norm{\cdot}_{\square}$ can be either $\ell_2$ or $\ell_1$ depending on the type of noises. As shown in 
\cite{ulyanov2018deep,heckel2019denoising,you2020robust}, by learning the parameters $\Theta$, the network first reveals the underlying clear image and then \emph{overfits} due to overparameterization, so that in many cases \emph{early stopping} is \emph{crucially} needed for its success.

\begin{figure} 
\centering
  \includegraphics[width=0.5\textwidth]{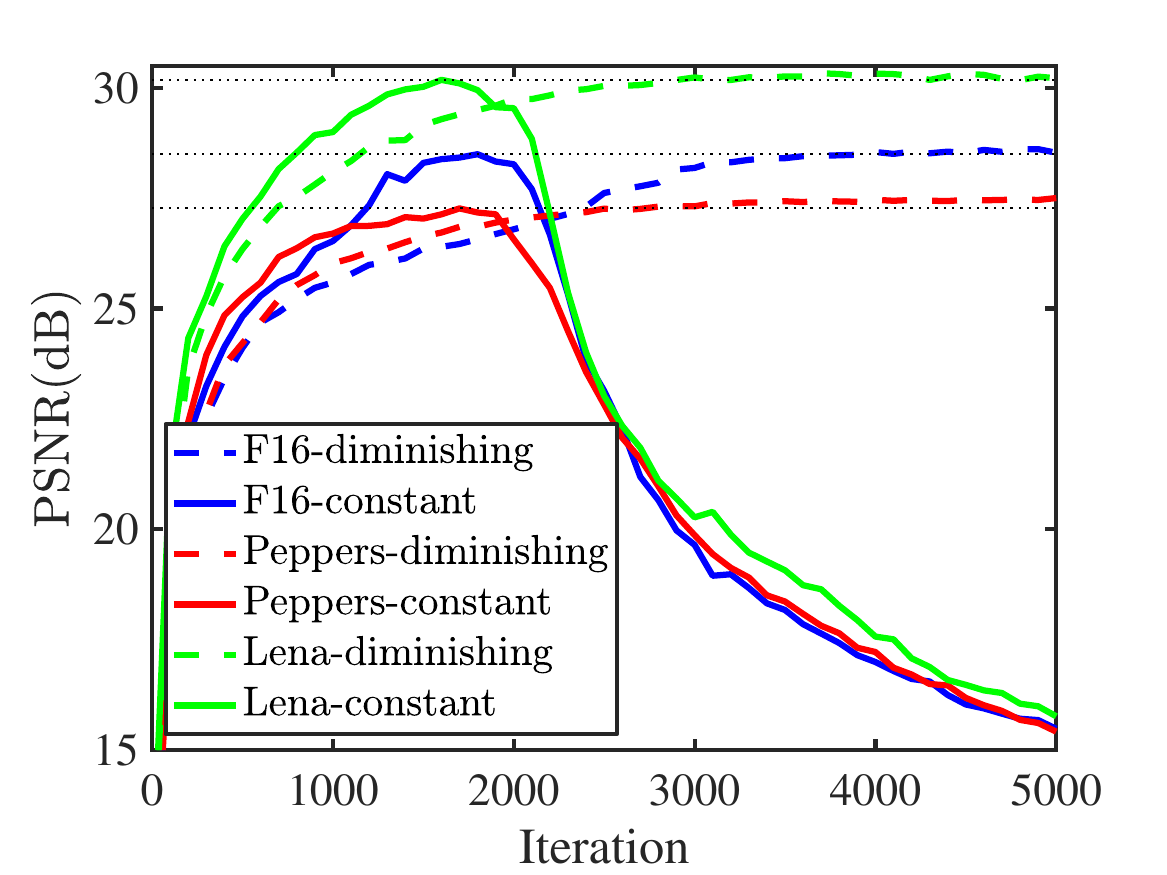}
  \caption{\footnotesize \textbf{Learning curves for robust image recovery with DIP} on different test images with 50\% salt-and-pepper noise. Here, we compare the constant and diminishing stepsize rules for DIP.}
\label{fig:DIP}
\end{figure}
In contrast to the constant stepsize rule, here we show a surprising result that, for sparse corruptions, optimizing a robust loss using the subgradient method with a diminishing stepsize rule does not overfit. In particular, we consider a sparse corruption $s$ by impulse salt-and-pepper noise, and solve the robust DIP problem $\min_{\Theta} \|y - \phi(\Theta)\|_{1}$ using the subgradient method with a diminishing stepsize rule. We test the performance using three images, ``F16'', ``Peppers'' and ``Lena'' from a standard dataset \cite{BM3D}, where we contaminate the image with salt-and-pepper noise by randomly replacing $50\%$ pixels with either $1$ or $0$ (each happens with $50\%$ probability). We compare the subgradient method with constant stepsize $0.01$ (or learning rate) and diminishing stepsize (we use an initial stepsize $0.05$ which is then scaled by $0.96$ for every $50$ iterations). As shown in \Cref{fig:DIP}, the subgradient method with constant stepsize eventually overfits so that it requires early stopping. In comparison, optimization with diminishing stepsize prevents overfitting in the sense that the performance continues to improve as training proceeds, and it achieves on par performance with that of constant stepsizes when the learning is early stopped at the optimal PSNR. However, early stopping is not easy to implement in practice without knowing the clean image \emph{a priori}.

\vspace{-0.05in}

\section{Discussion}\label{sec: conclusion}


In this work, we designed and analyzed a subgradient method for the robust matrix recovery problem under the rank overspecified setting, showing that it converges to the exact solution with a sublinear rate. Based on our results, there are several directions worth further investigation.
\begin{itemize}
\item \nct{\emph{Better sample complexity.} 
Our exact recovery result in Theorem \ref{thm: exact recovery under models} requires  $\tilde{\mathcal{O}}(dk^3)$ many samples. When the rank is exactly specified ($k=r$),
for factor approaches with corrupted measurements, the best sample complexity result is $\tilde{\mathcal{O}}(dr^2)$ \cite{li2016low,li2020non}. Hence there is an extra $k$ factor in the rank overspecified setting. 
Such a worse dependence may come from the $\delta^{-4}$ dependence of the sample complexity of RDPP in Proposition \ref{prop: main RDPP for two models}. More specifically, the Cauchy-Schwartz inequality we applied in \eqref{eq: csrdpp} might make the dependence on $\delta$ (and hence $k$) not tight enough. We leave the improvements for future work.}
    \item \emph{Implicit regularization}. As explored and demonstrated in our experiments in \Cref{sec:exp}, our stepsize choice combined with random initialization recovers $\truX$ with high accuracy, even if $k\geq r$ and the number of samples $m$ is way less than the degree of freedom (i.e., $m = O(dr)$). Theoretically justifying this phenomenon in the overparameterized regime would be of great interest.
    \item \emph{Other sensing matrix}. In this work, we assume $A_i$ to be a GOE matrix. This specific assumption on the measurement matrix still seems to be restrictive. For other measurements such as matrices with \emph{i.i.d.} sub-Gaussian entries, we believe that RDPP has to be modified. Empirically, we find that the algorithm is applicable to a broader setting (e.g., Robust PCA) even if the RDPP does not hold. Generalizing our results to more generic sensing matrices would also be an interesting future direction.
    \item \emph{Rectangular matrices}. In this work, we utilize the fact that $\truX$ is PSD, and optimize over one factor $F$. This is consistent with much existing analysis with rank overspecification \cite{li2018algorithmic,zhuo2021computational,ma2021implicit}. For rectangular $\truX$, we need to optimize over two factors $U,V$. How to combine the analysis here with an extra regularization $\fnorm{UU^\top -VV^\top}^2$ in the objective is another interesting direction to pursue.
\end{itemize}

\section*{Acknowledgement}
L. Ding would like to thank Jiacheng Zhuo for inspiring discussions. Y. Chen is partially supported by NSF grant CCF-1704828 and CAREER Award CCF-2047910. Z. Zhu is partially supported by NSF grants CCF-2008460 and CCF-2106881.

\qq{acknowledge salar here? for valuable comments}
\newpage 

{\small 
\bibliographystyle{unsrt}
\bibliography{reference}

\begin{thebibliography}{10}

\bibitem{candes2011robust}
Emmanuel~J Cand{\`e}s, Xiaodong Li, Yi~Ma, and John Wright.
\newblock Robust principal component analysis?
\newblock {\em Journal of the ACM (JACM)}, 58(3):1--37, 2011.

\bibitem{li2016low}
Yuanxin Li, Yue Sun, and Yuejie Chi.
\newblock Low-rank positive semidefinite matrix recovery from corrupted
  rank-one measurements.
\newblock {\em IEEE Transactions on Signal Processing}, 65(2):397--408, 2016.

\bibitem{rambach2021robust}
Markus Rambach, Mahdi Qaryan, Michael Kewming, Christopher Ferrie, Andrew~G
  White, and Jacquiline Romero.
\newblock Robust and efficient high-dimensional quantum state tomography.
\newblock {\em Physical Review Letters}, 126(10):100402, 2021.

\bibitem{li2020non}
Yuanxin Li, Yuejie Chi, Huishuai Zhang, and Yingbin Liang.
\newblock Non-convex low-rank matrix recovery with arbitrary outliers via
  median-truncated gradient descent.
\newblock {\em Information and Inference: A Journal of the IMA}, 9(2):289--325,
  2020.

\bibitem{li2020nonconvex}
Xiao Li, Zhihui Zhu, Anthony Man-Cho~So, and Rene Vidal.
\newblock Nonconvex robust low-rank matrix recovery.
\newblock {\em SIAM Journal on Optimization}, 30(1):660--686, 2020.

\bibitem{charisopoulos2021low}
Vasileios Charisopoulos, Yudong Chen, Damek Davis, Mateo D{\'\i}az, Lijun Ding,
  and Dmitriy Drusvyatskiy.
\newblock Low-rank matrix recovery with composite optimization: good
  conditioning and rapid convergence.
\newblock {\em Foundations of Computational Mathematics}, pages 1--89, 2021.

\bibitem{tong2021low}
Tian Tong, Cong Ma, and Yuejie Chi.
\newblock Low-rank matrix recovery with scaled subgradient methods: Fast and
  robust convergence without the condition number.
\newblock {\em IEEE Transactions on Signal Processing}, 2021.

\bibitem{vershynin2010introduction}
Roman Vershynin.
\newblock Introduction to the non-asymptotic analysis of random matrices.
\newblock {\em arXiv preprint arXiv:1011.3027}, 2010.

\bibitem{tropp2012user}
Joel~A Tropp.
\newblock User-friendly tail bounds for sums of random matrices.
\newblock {\em Foundations of computational mathematics}, 12(4):389--434, 2012.

\bibitem{vershynin2018high}
Roman Vershynin.
\newblock {\em High-dimensional probability: An introduction with applications
  in data science}, volume~47.
\newblock Cambridge university press, 2018.

\bibitem{wainwright2019high}
Martin~J Wainwright.
\newblock {\em High-dimensional statistics: A non-asymptotic viewpoint},
  volume~48.
\newblock Cambridge University Press, 2019.

\bibitem{duchi2019solving}
John~C Duchi and Feng Ruan.
\newblock Solving (most) of a set of quadratic equalities: Composite
  optimization for robust phase retrieval.
\newblock {\em Information and Inference: A Journal of the IMA}, 8(3):471--529,
  2019.

\bibitem{ma2021implicit}
Jianhao Ma and Salar Fattahi.
\newblock Implicit regularization of sub-gradient method in robust matrix
  recovery: Don't be afraid of outliers.
\newblock {\em arXiv preprint arXiv:2102.02969v2}, 2021.

\bibitem{you2020robust}
Chong You, Zhihui Zhu, Qing Qu, and Yi~Ma.
\newblock Robust recovery via implicit bias of discrepant learning rates for
  double over-parameterization.
\newblock In {\em Advances in Neural Information Processing Systems}, 2020.

\bibitem{zhuo2021computational}
Jiacheng Zhuo, Jeongyeol Kwon, Nhat Ho, and Constantine Caramanis.
\newblock On the computational and statistical complexity of over-parameterized
  matrix sensing.
\newblock {\em arXiv preprint arXiv:2102.02756}, 2021.

\bibitem{ulyanov2018deep}
Dmitry Ulyanov, Andrea Vedaldi, and Victor Lempitsky.
\newblock Deep image prior.
\newblock In {\em Proceedings of the IEEE conference on computer vision and
  pattern recognition}, pages 9446--9454, 2018.

\bibitem{recht2010guaranteed}
Benjamin Recht, Maryam Fazel, and Pablo~A Parrilo.
\newblock Guaranteed minimum-rank solutions of linear matrix equations via
  nuclear norm minimization.
\newblock {\em SIAM review}, 52(3):471--501, 2010.

\bibitem{chen2013low}
Yudong Chen, Ali Jalali, Sujay Sanghavi, and Constantine Caramanis.
\newblock Low-rank matrix recovery from errors and erasures.
\newblock {\em IEEE Transactions on Information Theory}, 59(7):4324--4337,
  2013.

\bibitem{liu2019recovery}
Guangcan Liu and Wayne Zhang.
\newblock Recovery of future data via convolution nuclear norm minimization.
\newblock {\em arXiv preprint arXiv:1909.03889}, 2019.

\bibitem{sun2016guaranteed}
Ruoyu Sun and Zhi-Quan Luo.
\newblock Guaranteed matrix completion via non-convex factorization.
\newblock {\em IEEE Transactions on Information Theory}, 62(11):6535--6579,
  2016.

\bibitem{bhojanapalli2016global}
Srinadh Bhojanapalli, Behnam Neyshabur, and Nathan Srebro.
\newblock Global optimality of local search for low rank matrix recovery.
\newblock In {\em Proceedings of the 30th International Conference on Neural
  Information Processing Systems}, pages 3880--3888, 2016.

\bibitem{ge2017no}
Rong Ge, Chi Jin, and Yi~Zheng.
\newblock No spurious local minima in nonconvex low rank problems: A unified
  geometric analysis.
\newblock In {\em International Conference on Machine Learning}, pages
  1233--1242. PMLR, 2017.

\bibitem{zhu2018global}
Zhihui Zhu, Qiuwei Li, Gongguo Tang, and Michael~B Wakin.
\newblock Global optimality in low-rank matrix optimization.
\newblock {\em IEEE Transactions on Signal Processing}, 66(13):3614--3628,
  2018.

\bibitem{chi2019nonconvex}
Yuejie Chi, Yue~M Lu, and Yuxin Chen.
\newblock Nonconvex optimization meets low-rank matrix factorization: An
  overview.
\newblock {\em IEEE Transactions on Signal Processing}, 67(20):5239--5269,
  2019.

\bibitem{zhang2020symmetry}
Yuqian Zhang, Qing Qu, and John Wright.
\newblock From symmetry to geometry: Tractable nonconvex problems.
\newblock {\em arXiv preprint arXiv:2007.06753}, 2020.

\bibitem{gunasekar2018implicit}
Suriya Gunasekar, Blake Woodworth, Srinadh Bhojanapalli, Behnam Neyshabur, and
  Nathan Srebro.
\newblock Implicit regularization in matrix factorization.
\newblock In {\em 2018 Information Theory and Applications Workshop (ITA)},
  pages 1--10. IEEE, 2018.

\bibitem{li2018algorithmic}
Yuanzhi Li, Tengyu Ma, and Hongyang Zhang.
\newblock Algorithmic regularization in over-parameterized matrix sensing and
  neural networks with quadratic activations.
\newblock In {\em Conference On Learning Theory}, pages 2--47. PMLR, 2018.

\bibitem{burke1993weak}
James~V Burke and Michael~C Ferris.
\newblock Weak sharp minima in mathematical programming.
\newblock {\em SIAM Journal on Control and Optimization}, 31(5):1340--1359,
  1993.

\bibitem{chen2020bridging}
Yuxin Chen, Jianqing Fan, Cong Ma, and Yuling Yan.
\newblock Bridging convex and nonconvex optimization in robust pca: Noise,
  outliers, and missing data.
\newblock {\em arXiv preprint arXiv:2001.05484}, 2020.

\bibitem{clarke2008nonsmooth}
Francis~H Clarke, Yuri~S Ledyaev, Ronald~J Stern, and Peter~R Wolenski.
\newblock {\em Nonsmooth analysis and control theory}, volume 178.
\newblock Springer Science \& Business Media, 2008.

\bibitem{ma2021implicitV3}
Jianhao Ma and Salar Fattahi.
\newblock Sign-rip: A robust restricted isometry property for low-rank matrix
  recovery.
\newblock {\em arXiv preprint arXiv:2102.02969v3}, 2021.

\bibitem{heckel2019denoising}
Reinhard Heckel and Mahdi Soltanolkotabi.
\newblock Denoising and regularization via exploiting the structural bias of
  convolutional generators.
\newblock {\em arXiv preprint arXiv:1910.14634}, 2019.

\bibitem{bisong2019google}
Ekaba Bisong.
\newblock Google colaboratory.
\newblock In {\em Building Machine Learning and Deep Learning Models on Google
  Cloud Platform}, pages 59--64. Springer, 2019.

\bibitem{polyak1969minimization}
Boris~Teodorovich Polyak.
\newblock Minimization of unsmooth functionals.
\newblock {\em USSR Computational Mathematics and Mathematical Physics},
  9(3):14--29, 1969.

\bibitem{DIP}
\url{https://github.com/DmitryUlyanov/deep-image-prior/blob/master/LICENSE}.

\bibitem{kingma2014adam}
Diederik~P Kingma and Jimmy Ba.
\newblock Adam: A method for stochastic optimization.
\newblock {\em arXiv preprint arXiv:1412.6980}, 2014.

\bibitem{ronneberger2015u}
Olaf Ronneberger, Philipp Fischer, and Thomas Brox.
\newblock U-net: Convolutional networks for biomedical image segmentation.
\newblock In {\em International Conference on Medical image computing and
  computer-assisted intervention}, pages 234--241. Springer, 2015.

\bibitem{BM3D}
\url{http://www.cs.tut.fi/~foi/GCF-BM3D/index.html#ref_results/}.

\bibitem{candes2011tight}
Emmanuel~J Candes and Yaniv Plan.
\newblock Tight oracle inequalities for low-rank matrix recovery from a minimal
  number of noisy random measurements.
\newblock {\em IEEE Transactions on Information Theory}, 57(4):2342--2359,
  2011.

\end{thebibliography}
}

\newpage

\appendix
\section{Analysis of algorithm under conditions of Theorem~\ref{lem: mainlemma}}
\label{sec: analysisOfMainTheorem}
To set the stage, we introduce some notations. Let the singular value decomposition of $\truX$ be 
\begin{equation}
	\truX = \begin{bmatrix}
		U & V
	\end{bmatrix}\begin{bmatrix}
		D_S^* & 0\\
		0   & 0
	\end{bmatrix}
	\begin{bmatrix}
		U & V
	\end{bmatrix}^\top,
\end{equation}
where $U \in \RR^{d\times r}, V \in \RR^{d \times d-r}, D_S^* \in \RR^{r \times r}$. $U$ and $V$ has orthonormal columns and $U^\top V = 0$. Denote $\sigma_i$ the singular values of $\truX$, thus $\sigma_1$ and $\sigma_r$ are the largest and smallest values in $D_S^*$. Since $\truX$ is assumed to have rank $r$, we have $\sigma_{r+1} =\sigma_{r+2}=\ldots =\sigma_d =0$. The condition number is defined to be $\kappa = \frac{\sigma_1}{\sigma_r}$. Since union of column space of $U$ and $V$ spans the whole space, for any $F_t \in \RR^{d\times r}$, we can write 
\begin{equation}
	F_t = US_t + VT_t,
\end{equation} 
where $S_t = U^\top F_t \in \RR^{r \times k}$ and $T_t = V^\top F_t \in \RR^{(d-r)\times k}$.\\
First of all, our initialization strategy will give us the following initialization quality.
\begin{proposition}[Initialization quality]\label{prop: initialization quality}
	Under the condition on $F_0$, we have 
	\begin{align}
		4\norm{T_0}^2 &\le 0.001\sigma_r(S_0) \frac{\sigma_r}{\sqrt{\sigma_1}}\\
		\norm{T_0}&\le \min\{0.1\sqrt{\sigma_1},\frac{\sigma_r}{200\sqrt{\sigma_1}}\}=\frac{\sigma_r}{200\sqrt{\sigma_1}}\\
		\norm{S_0} &\le 2\sqrt{\sigma_1}.
	\end{align}
	In the analysis, we will take $c_\rho$ and $\rho$ such that
	\begin{align}
		\sigma_r(S_0) &= \rho\\
		\norm{T_0} &\le c_\rho \rho =\frac{\sigma_r}{200\sqrt{\sigma_1}}\\
		\norm{S_0} & \le 2\sqrt{\sigma_1}.
	\end{align}
	The parameters satisfy $4(c_\rho\rho)^2 \le 0.001 \rho \frac{\sigma_r}{\sqrt{\sigma_1}}$ and $c_\rho \rho \le \min\{0.1\sqrt{\sigma_1},\frac{\sigma_r}{200\sqrt{\sigma_1}} \}=\frac{\sigma_r}{200\sqrt{\sigma_1}}$.
\end{proposition} 
 By assumptin, the RDPP holds with parameters $(k+r, \sqrt{\frac{1}{2\pi}}\delta)$ and $\delta = \frac{c}{\kappa^3 \sqrt{k}}$ for some small constant $c$ depending on $c_3$ in Theorem~\ref{lem: mainlemma}. Since the RDPP holds, let 
\begin{equation}\label{eq: gammatdef} 
	\gamma_t= \frac{\eta_t\psi(F_tF_t^\top -\truX)}{\fnorm{F_tF_t^\top - \truX}}, \quad D_t \in D(F_tF_t^\top -\truX),
\end{equation} 
we have
\begin{align}\label{eq: riponDt}
	\norm{\eta_t D_t -\gamma_t (F_tF_t^\top -\truX)}_{F,k+r} & \le \eta_t \sqrt{\frac{1}{2\pi}}\delta\\
	&\le \eta_t \psi(F_tF_t^\top -\truX) \delta\\
	&= \delta \gamma_t\fnorm{F_tF_t^\top -\truX}.
\end{align}
Define the following shorthand $\Delta_t$,
\begin{equation}\label{eqn: Deltatdef} 
	\Delta_t = \frac{\eta_t}{\gamma_t} D_t - (F_tF_t^\top -\truX). 
\end{equation} 
Using that fact that the subgradient we used in algorithm~\ref{eqn: mainalgorithm} can be written as $g_t = D_t F_t$, we have
\begin{align}
	F_{t+1} &= F_t - \gamma_t (F_tF_t^\top -\truX)F_t + \gamma_t\Delta_tF_t,\quad \text{and}  \label{eqn:algorithm update}\\
	\norm{\Delta}&\le  \delta \fnorm{F_tF_t^\top -\truX} \overset{(a)}{\le} \delta\sqrt{k+r}\norm{F_tF_t^\top - \truX} \label{eqn: RIP bound}
\end{align}
Here step $(a)$  is because $F_tF_t^\top -\truX$ has rank no more than $k+r$. 
Note that $\tilde{F}_{t+1} = F_t - \gamma_t \left(F_tF_t^\top -\truX\right)F_t$ is the update if we apply gradient descent to smooth function $\tilde{f}(F) = \frac{1}{4}\fnorm{FF^\top -\truX}^2$. We will leverage the properties of this population level update throughout the analysis.
\noindent The next proposition illustrates the evolution of $S_t$ and $T_t$.
\begin{proposition}[Updates of $S_t,T_t$]\label{prop: updates of S,T}
	For any $t \ge 0$, we have 
	\begin{align}
		S_{t+1} &= S_t - \gamma_t\left(S_t S_t^\top S_t + S_t T_t^\top T_t -D_S^*S_t\right) + \gamma_t U^\top \Delta_t F_t,\\
		T_{t+1} &= T_t - \gamma_t \left(T_tT_t^\top T_t + T_tS_t^\top S_t \right) + \gamma_t V^\top \Delta_t F_t.
	\end{align}
\end{proposition}
We introduce notations
\begin{align}
	\cM_t(S_t) &= S_t- \gamma_t\left(S_t S_t^\top S_t + S_t T_t^\top T_t -D_S^*S_t\right) \\
	\cN_t(T_t) &= T_t - \gamma_t \left(T_tT_t^\top T_t + T_tS_t^\top S_t \right).
\end{align}
They are "population-level" updates for $S_t$ and $T_t$.
\begin{proposition}[Uniform upper bound]\label{prop: uniform upper bound}
	Suppose $\gamma_t$ satisfies $\gamma_t \le \frac{0.01}{\sigma_1}$ for all $t\ge 0$ and $ (50\sqrt{k}\delta)^{\frac{1}{3}} \le \frac{c_\rho\rho}{2\sqrt{\sigma_1}} = \frac{\sigma_r}{400\sigma_1}$, we have 
	\begin{align}
		\norm{T_t} &\le c_\rho \rho\le 0.1 \sqrt{\sigma_r}\le 0.1 \sqrt{\sigma_1}\\
		\norm{S_t} &\le 2 \sqrt{\sigma_1}
	\end{align}
	for all $t \ge 0$.
\end{proposition}
The analysis of algorithm consists of three stages:
\begin{itemize}
	\item In stage $1$, we show at $\sigma_r(S_t)$ increases geometrically to level $\sqrt{\frac{\sigma_r}{2}}$ by time $\cT_1'$, then $\norm{S_tS_t^\top - D_S^*}$ will decrease geometrically to $\frac{100(c_\rho \rho)^2\sigma_1}{\sigma_r}(\le \frac{\sigma_r}{100})$by $\cT_1$. The iterate will then enter a good region.
	\item In stage $2$, we show that $D_t=\max\{\norm{S_tS_t^\top - D_S^*}, \norm{S_tT_t^\top}\}$ decreases geometrically if it is bigger than $10\delta \sqrt{2k}\sigma_1$, which is the computational threshold. In other words, $\norm{S_tS_t^\top - D_S^*}$ decrease to a $\frac{100(c_\rho \rho)^2\sigma_1}{\sigma_r}$ geometrically, and this will happen by $\cT_2$.
	\item  In stage $3$, after $\cT_2$, $E_t =\max\{\norm{S_tS_t^\top - D_S^*}, \norm{S_tT_t^\top},\norm{T_tT_t^\top}\}$ converges to $0$ sublinearly. 
\end{itemize} 
In the above statement, \begin{equation}
	\cT_1' = \ceil{\log \left(\frac{\sqrt{\sigma_r}}{\sqrt{2}\rho}\right)/ \log\left(1+\frac{c_\gamma\sigma_r^2}{6\sigma_1^2}\right)},
\end{equation} 
\begin{equation}
	\cT_1= \cT_1' + \ceil{\frac{\log(\frac{20  (c_\rho \rho)^2}{\sigma_r})}{\log\left(1- \frac{c_\gamma \sigma_r^2}{2\sigma_1^2}\right)}}
\end{equation}
and 
\begin{equation}
	\cT_2= \cT_1 + \ceil{\frac{\log\left( \frac{1000\delta \sqrt{k}\sigma_1}{\sigma_r}\right)}{\log\left(1-\frac{c_\gamma\sigma_r^2}{6\sigma_1^2}\right)}}.
\end{equation}
Stage $1$ consists of all the iterations up to time $\cT_1$. Stage $2$ consists of all the iterations between $\cT_1 + 1$ and $\cT_2$. Stage $3$ consists of all the iterations afterwards.
\subsection{Analysis of $\cM_t(S_t)$ and $\cN_t(T_t)$.}
In this sections we prove some facts about $\cM_t(S_t)$ and $\cN_t(T_t)$ that will be useful in the analysis.
\begin{proposition}\label{prop: bound on MM NN}
	Suppose $\gamma_t \le \min\{\frac{0.01}{\sigma_1}, \frac{0.01\sigma_r}{\sigma_1^2}\}$, $\norm{S_t} \le 2\sqrt{\sigma_1}$, $\sigma_r(S_t)\ge \sqrt{\frac{\sigma_r}{2}}$ and $\norm{T_t}\le 0.1\sqrt{\sigma_r}$, we have the following:
	\begin{enumerate}
		\item[1.] $\norm{\cM_t(S_t)\cM_t(S_t)^\top - D^*_S} \le (1-\frac{3\gamma_t\sigma_r}{4}) \norm{S_tS_t^\top - D_S^*} + 3 \gamma_t \norm{S_t T_t^\top}^2$.
	\end{enumerate}
	Suppose $\gamma_t \le \frac{0.01}{\sigma_1}$, $\norm{S_t} \le 2\sqrt{\sigma_1}$, and $\norm{T_t}\le 0.1\sqrt{\sigma_r}$, we have the following:
	\begin{enumerate}
		\item[2.] $\norm{\cN_t(T_t) \cN_t(T_t)^\top} \le  \norm{T_t}^2(1-\frac{3\gamma_t}{2} \norm{T_t}^2) = \norm{T_tT_t^\top}\left(1-\frac{3\gamma_t}{2} \norm{T_tT_t^\top}\right)$.
	\end{enumerate}
Furthermore, suppose $\gamma_t \le \frac{0.01}{\sigma_1}$, $\norm{S_t} \le 2\sqrt{\sigma_1}$, $\sigma_r(S_t)\ge \sqrt{\frac{\sigma_r}{2}}$, $\norm{T_t}\le 0.1\sqrt{\sigma_r}$, and $\norm{S_tS_t^\top - D_S^*} \le \frac{\sigma_r}{10 }$, we have same inequalities as $1$, $2$ and 
\begin{enumerate}
	\item[3.] $\norm{\cM_t(S_t)\cN_t(T_t)^\top} \le (1-\frac{\gamma_t\sigma_r}{3})\norm{S_tT_t^\top}$. 
\end{enumerate} 
\end{proposition}

\begin{proposition}\label{prop: bound on MS NT}
	Suppose $\gamma_t \le \frac{0.01}{\sigma_1}$, $\norm{S_t} \le 2\sqrt{\sigma_1}$, $\sigma_r(S_t) \ge \sqrt{\frac{\sigma_r}{2}}$ and $\norm{T_t}\le 0.1\sqrt{\sigma_r}$, we have the following:
	\begin{enumerate}
		\item $\norm{D_S^* - \cM_t(S_t)S_t^\top} \le \left(1-\frac{\gamma\sigma_r}{2}\right)\norm{D_S^*-S_tS_t^\top} + \gamma_t \norm{S_tT_t^\top}^2$.
		\item $\norm{\cM_t(S_t) T_t^\top} \le 2 \norm{S_tT_t^\top}$.
		\item $\norm{\cN_t(T_t) S_t^\top} \le \norm{T_tS_t^\top}$.
		\item $\norm{\cN_t(T_t) T_t^\top} \le  \norm{T_t}^2(1-\gamma_t \norm{T_t}^2) = \norm{T_tT_t^\top}\left(1-\gamma_t \norm{T_tT_t^\top}\right).$
	\end{enumerate}
\end{proposition}
\subsection{Analysis of Stage 1}
The following proposition characterize the evolution of $\sigma_r(S_t)$. In stage one, we start with a initialization satisfies conditions in Proposition~\ref{prop: initialization quality}.
\begin{proposition}\label{prop: stage one evolution}
	Suppose there is some constant $c_\gamma>0$ such that the parameters satisfy $\frac{c_\gamma \sigma_r}{\sigma_1^2}\le \gamma_t \le \min \{\frac{1}{100\sigma_1}, \frac{\sigma_r}{100 \sigma_1^2}\}= \frac{\sigma_r}{100 \sigma_1^2}$, $ (50\sqrt{k}\delta)^{\frac{1}{3}} \le \frac{c_\rho\rho}{2\sqrt{\sigma_1}} =\frac{\sigma_r}{400\sigma_1}$,  we have 
	\begin{equation}\label{eqn: evolution in stage one}
		\sigma_r(S_t) \ge \min\left\{(1+\frac{\sigma_r^2 c_\gamma}{6\sigma_1^2})^t \sigma_r(S_0) , \sqrt{\frac{\sigma_r}{2}}\right\}
	\end{equation}
	for all $t \ge 0$.
	In particular, we have 
	\begin{align}
		\sigma_r(S_{\cT_1'}+t) &\ge \sqrt{\frac{\sigma_r}{2}}.
	\end{align}
	for all $t \ge 0$.
\end{proposition}
\noindent Next, we show that $\norm{S_tS_t^\top - D_S^*}$ decays geometrically to $ \frac{100(c_\rho \rho)^2\sigma_1}{\sigma_r}$.
\begin{proposition}\label{prop: stagetwo phaseone}
	Suppose there is some constant $c_\gamma>0$ such that the parameters satisfy $\frac{c_\gamma \sigma_r}{\sigma_1^2}\le \gamma_t \le \min \{\frac{1}{100\sigma_1}, \frac{\sigma_r}{100 \sigma_1^2}\}= \frac{\sigma_r}{100 \sigma_1^2}$, $ (50\sqrt{k}\delta)^{\frac{1}{3}} \le \frac{c_\rho\rho}{2\sqrt{\sigma_1}}$,  we have for any $t \ge 0$, we have
	\begin{equation}\label{eqn: stagetwo phseone}
		\norm{S_{\cT_1+t}S_{\cT_1+t}^\top - D_S^*} \le \max\{5\sigma_1 (1-\frac{c_\gamma\sigma_r^2}{2\sigma_1^2})^t ,  \frac{100(c_\rho \rho)^2\sigma_1}{\sigma_r}\}.
	\end{equation}
	In particular, for $\cT_1= \cT_1' + \ceil{\frac{\log(\frac{20  (c_\rho \rho)^2}{\sigma_r})}{\log\left(1- \frac{c_\gamma \sigma_r^2}{2\sigma_1^2}\right)}}$, we have 
	\begin{equation}
		\norm{S_t S_t^\top - D_S^*} \le  \frac{100(c_\rho \rho)^2\sigma_1}{\sigma_r} \le \frac{\sigma_r}{100}, \qquad \forall t \ge \cT_1 .
	\end{equation}
\end{proposition}
\noindent When $t=\cT_1$,  by Proposition~\ref{prop: stagetwo phaseone},
\begin{equation}
	\norm{D_S^* - S_{\cT_1}S_{\cT_1}^\top} \le \frac{\sigma_r}{100}.
\end{equation} 
By Proposition~\ref{prop: uniform upper bound} and our assumption that $c_\rho \rho \le \frac{\sigma_r}{200\sqrt{\sigma_1}}$, 
\begin{equation}
	\norm{S_{\cT_1} T_{\cT_1}} \le 2(c_\rho \rho)\sqrt{\sigma_1} \le \frac{\sigma_r}{100}.
\end{equation}
Combining, we obtain
\begin{equation}
	D_{\cT_1} \le \frac{\sigma_r}{100}.
\end{equation}

\subsection{Analysis of Stage 2}
Recall
\begin{equation}
	D_t = \max\{\norm{S_t S_t^\top- D_S^*}, \norm{S_t T_t^\top}\}.
\end{equation} 
We show that $D_t$ decreases to $10\delta \sqrt{k+r}\sigma_1$ geometrically after $\cT_1$.
\begin{proposition}\label{prop: stagetwo phasetwo}
	Suppose there is some constant $c_\gamma>0$ such that the parameters satisfy $\frac{c_\gamma\sigma_r}{\sigma_1^2}\le \gamma_t \le \frac{0.01}{\sigma_1}$, $ \delta \sqrt{k+r} \le \frac{0.001\sigma_r}{\sigma_1}$. Also, we suppose $\norm{T_t}\le 0.1\sqrt{\sigma_r}$ for all $t$.  If for some $\cT_1>0$, 
	\begin{equation}
		D_{\cT_1} \le \max\{\frac{\sigma_r}{100}, 10 \delta \sqrt{k+r}\sigma_1 \}=\frac{\sigma_r}{100},
	\end{equation} then for any $t\ge 0$, we have
	\begin{equation}\label{eqn: stagetwo phasetwo}
		D_{\cT_1 +t} \le \max\left\{\left(1-\frac{c_\gamma\sigma_r^2}{6\sigma_1^2}\right)^t\cdot \frac{\sigma_r}{100} ,10\delta \sqrt{k+r} \sigma_1\right\}.
	\end{equation}
	In particular, for $\cT_2= \cT_1 + \ceil{\frac{\log\left( \frac{1000\delta \sqrt{k+r}\sigma_1}{\sigma_r}\right)}{\log\left(1-\frac{c_\gamma\sigma_r^2}{6\sigma_1^2}\right)}} $, we have 
	\begin{align}
		\norm{S_t S_t^\top - D_S^*} &\le 10\delta \sqrt{k+r}\sigma_1,\\ 
		\norm{S_t T_t^\top} &\le 10\delta \sqrt{k+r}\sigma_1, \qquad \forall t \ge \cT_2.
	\end{align}
\end{proposition}

\subsection{Analysis of Stage 3}
Define 
\begin{equation}
	E_t = \max\{\norm{S_tS_t^\top -D_S^*}, \norm{S_t T_t^\top}, \norm{T_tT_t^\top}\}.
\end{equation} 
We are going to show the sublinear convergence of $E_t$ in stage three.
\begin{proposition}\label{prop: stagethree}
	Suppose we have $\gamma_t \le \frac{0.01}{\sigma_1}$, $\delta\sqrt{k+r} \le \frac{0.001 \sigma_r}{\sigma_1}$ and $E_{t} \le 0.01\sigma_r$ for some $t>0$. Then we have
	\begin{equation}
		E_{t+1} \le  \max \{(1-\frac{\gamma_t\sigma_r}{6})E_t, E_t(1-\gamma_t E_t)\} =E_t(1-\gamma_tE_t).
	\end{equation}
\end{proposition}
Indeed, we can prove a better rate if there is no overparametrization. 
\begin{proposition}\label{prop: stagethree k=r}
	Suppose we have $\gamma_t \le \frac{0.01}{\sigma_1}$, $\delta\sqrt{k+r} \le \frac{0.001 \sigma_r}{\sigma_1}$ and $E_{t} \le 0.01\sigma_r$ for some $t>0$. If $k=r$, then we have
	\begin{equation}
		E_{t+1} \le  (1-\frac{\gamma_t\sigma_r}{3})E_t.
	\end{equation}
\end{proposition}
\subsection{Proof of Theorem~\ref{lem: mainlemma}}
The proof is a combination of all the propositions in this section. First, we show that under suitable choice of $c_0$ and $c_3$, all the assumptions are satisfied. First, if we take $c_3$ to be small enough, we know that $(50\sqrt{k}\delta)^{\frac{1}{3}} \le \frac{\sigma_r}{400 \sigma_1}$ holds. Hence, all the conditions related to $\delta$ are satisfied. Next, by definition, 
$\gamma_t = \frac{\eta_t \psi(F_tF_t^\top -\truX)}{\fnorm{F_tF_t^\top -\truX}}$. By the second assumption and the assumption on range of $\psi$, we know 
\begin{equation}
    \gamma_t \in[c_1 \sqrt{\frac{1}{2\pi}} \frac{\sigma_r}{\sigma_1^2}, c_2\sqrt{\frac{2}{\pi}} \frac{\sigma_r}{\sigma_1^2}].
\end{equation}
Since we assumed $c_2\le 0.01$, so the step size condition $\gamma_t \le \frac{\sigma_r}{100\sigma_1^2}$ is satisfied. Moreover, $c_\gamma \ge c_1 \sqrt{\frac{1}{2\pi}}$. Now, applying theorems for initialization, stage 1 and stage 2, we know that 
	\begin{align}
		\norm{S_{\cT_2} S_{\cT_2}^\top - D_S^*} &\le 10\delta \sqrt{k+r}\sigma_1 \le \frac{0.01 \sigma_r^2}{\sigma_1},\\ 
		\norm{S_{\cT_2} T_{\cT_2}^\top} &\le 10\delta \sqrt{k+r}\sigma_1 \le \frac{0.01\sigma_r^2}{\sigma_1}.
	\end{align}
In addition, by Proposition~\ref{prop: uniform upper bound}, we know 
\begin{align}
    \norm{T_{\cT_2}T_{\cT_2}^\top} = \norm{T_{\cT_2}}^2 \le (c_\rho\rho)^2 \le \frac{0.01\sigma_r^2}{\sigma_1}.
\end{align}
Hence, $E_{\cT_2}\le \frac{0.01 \sigma_r^2}{\sigma_1}$. 
Here are two cases:
\begin{itemize}
    \item $k > r$, By Proposition~\ref{prop: stagethree} and induction, we know
    \begin{equation}
        E_{t+1} \le E_t(1-\gamma_tE_t) \le E_t(1-\frac{c_\gamma\sigma_r}{\sigma_1^2} E_t),\qquad \forall t \ge \cT_2.
    \end{equation}
    where $ c_1\sqrt{\frac{2}{\pi}}\le c_\gamma \le 0.01$. Define $G_t = \frac{c_\gamma \sigma_r}{\sigma_1^2}E_t$, then we have $G_{\tau_2} <1$ and
\begin{align}
	G_{t+1} \le G_t(1-G_t),\qquad \forall t \ge \cT_2.
\end{align}
Taking reciprocal, we obtain
\begin{equation}
	\frac{1}{G_{t+1}} \ge \frac{1}{G_t} + \frac{1}{1-G_t} \ge \frac{1}{G_t} +1,\qquad \forall t \ge \cT_2
\end{equation}
So we obtain 
\begin{equation}
	G_{\cT_2+t} \le \frac{1}{\frac{1}{G_{\cT_2}}+t}, \qquad \forall t \ge 0.
\end{equation}
Plugging in the definition of $G_t$, we obtain
\begin{equation}
    E_{\tau_2 +t} \le  \frac{\sigma_1^2}{c_\gamma \sigma_r} \frac{1}{\frac{\sigma_1^2}{c_\gamma \sigma_r E_{\cT_2}} +t} \le \frac{\sigma_1^2}{c_\gamma \sigma_r} \frac{1}{\frac{100\sigma_1^3}{c_\gamma \sigma_r^3} +t} = \frac{\sigma_1}{c_\gamma} \frac{\kappa}{\frac{100}{c_\gamma}\kappa^3 +t}\le \frac{\sigma_1}{c_\gamma} \frac{\kappa}{\kappa^3 +t}.
\end{equation}
Since $c_\gamma \ge c_1\sqrt{\frac{2}{\pi}}$, we can simply take $c_5 = \frac{1}{4c_1}\sqrt{\frac{\pi}{2}}$, $\cT =\cT_2$, apply Lemma~\ref{lem: decomposition of FF-X}, and get 
\begin{equation}
    \norm{F_{\cT+t}F_{\cT+t}^\top -\truX} \le c_5\sigma_1\frac{\kappa}{\kappa^3 +t},\qquad \forall t\ge 0.
\end{equation}
The last thing to justify is $\cT_2 \lesssim \kappa^2 \log \kappa$. Recall 
\begin{equation}
	\cT_1' = \ceil{\log \left(\frac{\sqrt{\sigma_r}}{\sqrt{2}\rho}\right)/ \log\left(1+\frac{c_\gamma\sigma_r^2}{6\sigma_1^2}\right)},
\end{equation} 
\begin{equation}
	\cT_1= \cT_1' + \ceil{\frac{\log(\frac{20  (c_\rho \rho)^2}{\sigma_r})}{\log\left(1- \frac{c_\gamma \sigma_r^2}{2\sigma_1^2}\right)}}
\end{equation}
and 
\begin{equation}
	\cT_2= \cT_1 + \ceil{\frac{\log\left( \frac{1000\delta \sqrt{k}\sigma_1}{\sigma_r}\right)}{\log\left(1-\frac{c_\gamma\sigma_r^2}{6\sigma_1^2}\right)}}.
\end{equation}
Simple calculus yield that each integer above is $O(\kappa^2 \log \kappa)$. So the proof is complete in overspecified case.
\item $k=r$. By Proposition~\ref{prop: stagethree k=r} and induction,
we obtain
\begin{equation}
    E_{t+1} \le (1-\frac{\gamma_t\sigma_r}{3})E_t \le (1-\frac{c_\gamma \sigma_r^2}{\sigma_1^2})E_t, \forall t\ge \cT_2.
\end{equation}
Applying this inequality recursively and noting $c_\gamma \ge c_1\sqrt{\frac{2}{\pi}}$, we obtain
\begin{equation}
    E_{\cT_2+t} \le (1-\frac{c_\gamma \sigma_r^2}{\sigma_1^2})^t E_{\cT_2}\le \left(1-\frac{c_1\sqrt{\frac{2}{\pi}}}{\kappa^2}\right)^t \frac{0.01\sigma_r}{\kappa}, \forall t \ge 0.
\end{equation}
Thus, we can take $c_6 = 0.01/4$, $c_7 = c_1\sqrt{\frac{2}{\pi}}$,  $\cT = \cT_2$, apply Lemma~\ref{lem: decomposition of FF-X} and get
\begin{align}
     \norm{F_{\cT+t}F_{\cT+t}^\top -\truX} \le \frac{c_6\sigma_r}{\kappa}\left(1-\frac{c_7}{\kappa^2}\right)^t,\qquad \forall t\ge 0.
\end{align} 
The validity of $\cT$ is proved in the last part. The proof is complete.
\end{itemize}

\section{Proof of Propositions}
	\subsection{Proof of Proposition~\ref{prop: initialization quality}}
First, we note that the $r$-th singular value of $c^*\truX$ is at least $\epsilon \sigma_r$. By almost the same proof as Lemma~\ref{lem: decomposition of FF-X}, we get
\begin{equation}
	\max\{\norm{S_0S_0^\top -c^*D_S^*}, \norm{S_0T_0^\top}, \norm{T_0T_0^\top}\} \le \norm{F_0F_0^\top -c^*\truX} \le  \frac{\tilde c_0 \epsilon \sigma_r}{\kappa}.
\end{equation}
We take $\tilde c_0 =\left(\frac{1}{200}\right)^2$. By Weyl's inequality~\eqref{lem: weyl's ineq}, \begin{align}
	\sigma_r(S_0S_0^\top) \ge \sigma_r(c^*D_S^*) - \norm{S_0S_0^\top -c^*D_S^*} \ge \frac{c^*\sigma_r}{4}\ge \frac{\epsilon\sigma_r}{4}.
\end{align}
Hence, $\rho = \sigma_r(S_0) \ge \frac{\sqrt{\epsilon\sigma_r}}{2}$. On the other hand,
\begin{align}
    \norm{T_0}&\le \sqrt{\frac{\tilde c_0 \epsilon \sigma_r}{\kappa}}\\
    &\le \frac{\sigma_r}{200\sqrt{\sigma_1}}
\end{align}
We can simply assume $\sigma_1(S_0) \le 2\sqrt{\sigma_1}$. If not so, we can normalize $F_0$ so that $\sigma_1(S_0)= 2\sqrt{\sigma_1}$ and use normalized $F_0$ as our initialization. By Weyl's inequality~\eqref{lem: weyl's ineq},
\begin{align}
    \sigma_1(S_0S_0^\top) \le 1.01c^* \sigma_1.
\end{align}
Hence, $c^* \ge 3$. In this case, the factor we use to normalize $S_0$ is propotional to $c^*\ge3$. It's easy to show that $\rho = \sigma_r(S_0) \ge \frac{\sqrt{\epsilon\sigma_r}}{2}$ and $\norm{T_0} \le \frac{\sigma_r}{200\sqrt{\sigma_1}}$ still holds. Therefore, the initialization quality is proved.
\subsection{Proof of Proposition~\ref{prop: updates of S,T}}
The algorithm~\ref{eqn:algorithm update} updates $F_t$ by
\begin{equation}
	F_{t+1} = F_t - \gamma_t \left(F_tF_t^\top -\truX\right)F_t +\gamma_t \Delta_t F_t .
\end{equation}
Using the definition of $U,V, S_t, T_t$, we have
\begin{align*}
	S_{t+1} &= U^\top F_{t+1}\\
	&=  U^\top F_t - \gamma_t U^\top\left(F_tF_t^\top -\truX\right)F_t+ \gamma_t U^\top \Delta_tF_t\\
	&=U^\top (US_t + VT_t) - \gamma_t U^\top\left[(US_t+VT_t)(US_t + VT_t)^\top - UD^*_S U^\top  - VD^*_T V^\top\right] (US_t+ VT_t)\\
	&\qquad + \gamma_t U^\top \Delta_tF_t\\
	&\overset{(\sharp)}{=} S_t - \gamma_t(S_tS_t^\top S_t + S_tT_t^\top T_t - D_S^*S_t)+ \gamma_t U^\top \Delta_tF_t.
\end{align*} 
Here $(\sharp)$ follows from the fact that $U^\top V = 0$ and $U,V$ are orthonormal.\\
By the same token, we can show
\begin{equation}
	T_{t+1} = T_t  - \gamma_t\left(T_tT_t^\top T_t + T_t S_t^\top S_t \right)+ \gamma_t V^\top \Delta_tF_t.
\end{equation}

\subsection{Proof of Proposition~\ref{prop: uniform upper bound}}
We prove the proposition by induction. By Proposition~\ref{prop: initialization quality}, it's clear that the proposition holds for $t=0$. Suppose for $t \ge 0$, we have
\begin{align}
	\norm{T_t} &\le c_\rho \rho \le 0.1\sqrt{\sigma_1}\\
	\norm{S_t} &\le 2\sqrt{\sigma_1}.
\end{align} 
By Proposition~\ref{prop: updates of S,T}, we know 
\begin{equation}
	S_{t+1}= S_t - \gamma_t(S_tS_t^\top S_t + S_tT_t^\top T_t - D_S^*S_t)+ \gamma_t U^\top \Delta_tF_t.
\end{equation}
Since $\|T_t\| \le c_\rho \rho \le 0.1\sqrt{\sigma_1}$, $\|S_t\| \le 2\sqrt{\sigma_1}$ and our assumption that $\gamma_t \le \frac{0.01}{\sigma_1}$,  $I- \gamma_t S_t^\top S_t -\gamma_t T_t^\top T_t$ is a PSD matrix. By lemma~\ref{lem: S(I-SS) svd},
\begin{align}
	\norm{S_t\left(I - \gamma_t S_t^\top S_t - \gamma_t T_t^\top T_t\right) } &\le \norm{S_t\left(I - \gamma_t S_t^\top S_t \right) } + \gamma_t \norm{S_t}\norm{T_t}^2\\
	&= \norm{S_t} - \gamma_t \norm{S_t}^3 + 0.1\gamma_t \sigma_1^{\frac{3}{2}}.
\end{align}
On the other hand, simple triangle inequality yields
\begin{equation}
	\norm{F_t} = \norm{U S_t +V T_t} \le \norm{S_t} + \norm{T_t} \le 3\sqrt{\sigma_1}.
\end{equation}
By \ref{eqn: RIP bound} and lemma~\ref{lem: fnorm opnorm}, we get
\begin{align}
	\norm{\Delta_t F_t}&\le \norm{\Delta_t} \norm{F_t}\\
	&\le 3\sigma_1^\frac{1}{2}\delta \fnorm{F_tF_t^\top - \truX} \\
	&\le 3\sigma_1^\frac{1}{2}\delta \sqrt{k+r} \norm{F_tF_t^\top- \truX}\\
	&\le  3\sigma_1^\frac{1}{2}\delta \sqrt{k+r}\left( \norm{F_t}^2 + \norm{\truX}\right)\\
	&\le 50\delta \sqrt{k} \sigma_1^\frac{3}{2}\\
	&\le 0.1 \sigma_1^\frac{3}{2}
\end{align}
Combining, we have
\begin{align}
	\norm{S_{t+1}} &\le \norm{S_t\left(I - \gamma_t S_t^\top S_t - \gamma_t T_t^\top T_t\right) } + \gamma_t \norm{D_S^* S_t} + \gamma_t  \norm{U^\top \Delta_t F_t}\\
	&\le \norm{S_t} - \gamma_t \norm{S_t}^3  + \gamma_t\sigma_1 \norm{S_t} + 0.2\gamma_t\sigma_1^\frac{3}{2} \\
	&= \norm{S_t} \left(1 + \gamma_t\sigma_1 - \gamma_t \norm{S_t}^2\right) +0.2\gamma_t\sigma_1^\frac{3}{2}.
\end{align}
We consider two different cases:
\begin{itemize}
	\item $\|S_t\| \le 1.5 \sqrt{\sigma_1}$. By the inequality above, we have 
	\begin{equation}
	\norm{S_{t+1}} \le \norm{S_t}(1+\gamma_t \sigma_1) +0.1\gamma_t \sigma_1^\frac{3}{2} \le 1.6\sqrt{\sigma_1} + 0.2\sqrt{\sigma_1} \le 2 \sqrt{\sigma_1}
	\end{equation}
	\item $ 1.5 \sqrt{\sigma_1}< \|S_t\| \le 2 \sqrt{\sigma_1}$. In this case, we  have 
	\begin{align}
		\norm{S_{t+1}} &\le \norm{S_t}(1+\gamma_t \sigma_1 - 2.25\gamma_t \sigma_1) +0.2\gamma_t\sigma_1^\frac{3}{2}\\
		&\le \norm{S_t}(1-\gamma_t \sigma_1)  + 0.2\gamma_t \sigma_1 \norm{S_t}\\
		&\le \norm{S_t}\\
		&\le 2\sqrt{\sigma_1}.
	\end{align}
\end{itemize}
The desired bound for $S_{t+1}$ is established. For $T_{t+1}$, we note 
\begin{align}
	T_{t+1} = T_t(I - \gamma_t T_t^\top T_t - \gamma_t S_t^\top S_t) +\gamma_t V^\top \Delta_t F_t.
\end{align}
We expand $T_{t+1}T_{t+1}^\top$ and obtain
\begin{align}
	T_{t+1} T_{t+1}^\top &= (\cN_t(T_t) + \gamma_t V^\top \Delta_t F_t)(\cN_t(T_t) + \gamma_t V^\top \Delta_t F_t)^\top\\
	&\le \underbrace{\cN_t(T_t)\cN_t(T_t)^\top}_{Z_1} +\underbrace{\gamma_t V^\top \Delta_t F_t \cN_t(T_t)^\top + \gamma_t \cN_t(T_t) F_t^\top \Delta_t^\top V}_{Z_2}\\
	&\qquad + \underbrace{\gamma_t^2 V^\top \Delta_t F_tF_t^\top \Delta_t^\top V}_{Z_3}
\end{align}
By Proposition~\ref{prop: bound on MM NN}, we have $\|Z_1\|\le \|T_tT_t^\top\|(1-\frac{3\gamma_t}{2} \|T_tT_t^\top\|).$ By induction hypothesis and triangle inequality, we have
\begin{align}
    \norm{Z_2}&\le 2\gamma_t 50\delta \sqrt{k}\sigma_1^{\frac{3}{2}} \|\cN_t(T_t)\| \\
    &\le 2\gamma_t 50\delta \sqrt{k}\sigma_1^{\frac{3}{2}} \|T_t\| 
\end{align}
and
\begin{align}
    \norm{Z_3} \le \gamma_t^2 \left(50\delta \sqrt{k}\sigma_1^{\frac{3}{2}}  \right)^2 \le 0.01\gamma_t (c_\rho \rho)^4.
\end{align}
By triangle inequality, we have
	\begin{align}
		\norm{T_{t+1}T_{t+1}^\top} &\le \|Z_1\|+ \|Z_2\| + \|Z_3\|\\
		&\le \|T_tT_t^\top\|(1-\frac{3\gamma_t}{2} \|T_tT_t^\top\|) + 2\gamma_t 50\delta \sqrt{k}\sigma_1^{\frac{3}{2}} \|T_t\| +0.01\gamma_t (c_\rho \rho)^4
	\end{align}
We consider two different cases:
\begin{itemize}
	\item $\norm{T_tT_t^\top } \le  \frac{(c_\rho \rho)^2}{2}$. We have 
	\begin{align}
		\norm{T_{t+1}T_{t+1}^\top} &\le \|T_tT_t^\top\| + 2\gamma_t 50\delta \sqrt{k}\sigma_1^{\frac{3}{2}} \|T_t\| +0.01\gamma_t (c_\rho \rho)^4\\
		&\le \frac{(c_\rho \rho)^2}{2} + \frac{\gamma_t (c_\rho \rho)^4}{4} + 0.01\gamma_t (c_\rho \rho)^4\\
		&\le (c_\rho \rho)^2.
	\end{align}
	
\item $\frac{(c_\rho \rho)^2}{2} <\norm{T_tT_t^\top} \le (c_\rho \rho)^2$. We have 
\begin{align}
		\norm{T_{t+1}T_{t+1}^\top} &\le \|T_tT_t^\top\|(1-\frac{3\gamma_t}{2}\|T_tT_t^\top\|) + 2\gamma_t 50\delta \sqrt{k}\sigma_1^{\frac{3}{2}} \|T_t\| +0.01\gamma_t (c_\rho \rho)^4\\
		&\le \|T_tT_t^\top\| - \frac{3\gamma_t}{8}(c_\rho \rho)^4 +2\gamma_t 50\delta \sqrt{k}\sigma_1^{\frac{3}{2}} \|T_t\| +0.01\gamma_t (c_\rho \rho)^4\\
		&\le \|T_tT_t^\top\| - \frac{3\gamma_t}{8}(c_\rho \rho)^4 +\frac{\gamma_t (c_\rho \rho)^4}{4} +0.01\gamma_t (c_\rho \rho)^4\\
		&\le \|T_tT_t^\top\|^2\\
		&\le (c_\rho \rho)^2.
	\end{align}
\end{itemize}
Hence, we proved the inequality for $\norm{T_{t+1}}$. By induction, the proof is complete.
\subsection{Proof of Proposition~\ref{prop: bound on MM NN}}
\begin{enumerate}
	\item $\norm{\cM_t(S_t)\cM_t(S_t)^\top - D^*_S} \le (1-\frac{3\gamma_t\sigma_r}{4}) \norm{S_tS_t^\top - D_S^*} + 3 \gamma_t \norm{S_t T_t^\top}^2$.\\
First, we suppose that we have  $\gamma_t \le \min\{\frac{0.01}{\sigma_1}, \frac{0.01\sigma_r}{\sigma_1^2}\}$. By definition,
\begin{equation}
	\cM_t(S_t)  = S_t - \gamma_t \left(S_t S_t^\top S_t +S_t T_t^\top T_t - D_S^*S_t\right).
\end{equation}
This yields
\begin{align}
	&D_S^* - \cM_t(S_t)\cM_t(S_t)^\top \\
	&=  D_S^* - \left[S_t - \gamma_t \left(S_t S_t^\top S_t +S_t T_t^\top T_t - D_S^*S_t\right)\right]\left[S_t - \gamma_t \left(S_t S_t^\top S_t +S_t T_t^\top T_t - D_S^*S_t\right)\right]^\top\\
	&= Z_1+ Z_2+ Z_3,
\end{align}
where 
\begin{align}
	Z_1 &= D_S^* - S_tS_t^\top  - \gamma_t (D_S^*-S_tS_t^\top )S_tS_t^\top -\gamma_t S_tS_t^\top ( D_S^*-S_tS_t^\top),\\
	Z_2 &= 2\gamma_t S_tT_t^\top T_t S_t^\top ,
\end{align}
and 
\begin{equation}
	Z_3 = -\gamma_t^2 (S_tS_t^\top S_t + S_tT_t^\top T_t - D_S^*S_t)(S_tS_t^\top S_t + S_tT_t^\top T_t - D_S^*S_t)^\top.
\end{equation}
We bound each of them separately. For $Z_1$, by triangle inequality,
\begin{align}
	\norm{Z_1} &= \norm{\left(D_S^* - S_tS_t^\top\right)\left(\frac{1}{2}I - \gamma_t S_tS_t^\top\right) +\left(\frac{1}{2}I - \gamma_t S_tS_t^\top\right)\left(D_S^* - S_tS_t^\top\right)}\\
	&\le \norm{\left(D_S^* - S_tS_t^\top\right)\left(\frac{1}{2}I - \gamma_t S_tS_t^\top\right)} + \norm{\left(\frac{1}{2}I - \gamma_t S_tS_t^\top\right)\left(D_S^* - S_tS_t^\top\right)}\\
	&\le  (\frac{1}{2} - \gamma_t \sigma_r^2(S_t))\norm{D_S^* -S_tS_t^\top} +(\frac{1}{2} - \gamma_t \sigma_r^2(S_t))\norm{D_S^* -S_tS_t^\top}\\
	&\le (1-\gamma_t \sigma_r) \norm{D_S^* - S_tS_t^\top}
\end{align}
The norm of $Z_2$ can be simply bounded by 
\begin{equation}
	\norm{Z_2} \le 2\gamma_t \norm{S_t T_t^\top}^2 
\end{equation}
For $Z_3$,
we can split it as 
\begin{align}
	Z_3 &= - \gamma_t ^2 (S_tS_t^\top - D_S^*)S_tS_t^\top(S_tS_t^\top - D_S^*)^\top  - \gamma_t^2 S_TT_t^\top T_tT_t^\top T_t S_t^\top \\
	&\qquad - \gamma_t^2 \left(S_tT_t^\top T_t\left(S_t^\top S_tS_t^\top -S_t^\top D_S^*\right) + \left(S_tS_t^\top S_t - D_S^*S_t\right)T_t^\top T_t S_t^\top\right)
\end{align}
By triangle inequality and our assumption that $\norm{S_t} \le 2\sqrt{\sigma_1}$, we have $\norm{S_tS_t^\top -D_S^*} \le 5\sigma_1$. Hence
\begin{align}
	\norm{Z_3} &\le 20\gamma_t^2\sigma_1^2 \norm{S_tS_t^\top -D_S^*} + 0.01\gamma_t^2\sigma_1 \norm{S_tT_t^\top}^2 + 2\gamma_t^2\norm{S_tS_t^\top -D_S^*} \norm{S_tT_t^\top}^2\\
	&\overset{(\sharp)}{\le}  20\gamma_t^2\sigma_1^2 \norm{S_tS_t^\top -D_S^*} + 0.01\gamma_t^2 \sigma_1 \norm{S_tT_t^\top}^2 + \gamma_t^2\sigma_1 \sigma_r\norm{S_tS_t^\top -D_S^*}\\
	&\le  (20\gamma_t^2\sigma_1^2+ \gamma_t^2 \sigma_1\sigma_r) \norm{S_tS_t^\top -D_S^*} +\gamma_t \norm{S_tT_t^\top}^2 \\
	&\le  (\frac{20}{100}\gamma_t \sigma_r+ \frac{1}{100}\gamma_t\sigma_r) \norm{S_tS_t^\top -D_S^*} +\gamma_t \norm{S_tT_t^\top}^2
\end{align}
Here $(\sharp)$ follows from our assumption that $\norm{S_t}\le 2\sqrt{\sigma_1}$ and $\norm{T_t} \le 0.1\sqrt{\sigma_r}$.
Combining, we have
\begin{align}
	\norm{D_S^* - S_{t+1}S_{t+1}^\top} &\le \norm{Z_1} + \norm{Z_2} + \norm{Z_3}\\
	&\le (1-\gamma_t\sigma_r + \frac{21}{100}\gamma_t \sigma_r) \norm{D_S^* - S_tS_t^\top} +(2\gamma_t+\gamma_t)\norm{S_tT_t^\top}^2\\
	&\le \left(1-\frac{3\gamma_t\sigma_r}{4}\right) \norm{D_S^*- S_tS_t^\top} + 3\gamma_t \norm{S_tT_t^\top}^2
\end{align}
If we assume $\norm{S_tS_t^\top - D_S^*} \le \frac{\sigma_r}{10}$ and $\gamma_t \le \frac{0.01}{\sigma_1}$ instead, the only bound that will change is 
\begin{align}
\norm{Z_3} &\le 4\gamma_t^2\sigma_1 \norm{S_tS_t^\top -D_S^*}^2 + 0.01\gamma_t^2\sigma_r \norm{S_tT_t^\top}^2 + 2\gamma_t^2\norm{S_tS_t^\top -D_S^*} \norm{S_tT_t^\top}^2\\
&\le 4\gamma_t^2 \sigma_1 \frac{\sigma_r}{10} \norm{S_tS_t^\top - D_S^*} + 0.01 \gamma_t^2\sigma_1 \norm{S_tT_t^\top}^2 + \gamma_t^2 \sigma_1\sigma_r\norm{S_tS_t^\top-D_S^*}\\
&\le \left(\frac{4}{1000}\gamma_t \sigma_r+\frac{1}{100}\gamma_t \sigma_r\right) \norm{S_tS_t^\top - D_S^*} + \gamma_t \norm{S_tT_t^\top}^2. 
\end{align}
With this bound, we can do same argument except only with $\gamma_t \le \frac{0.01}{\sigma_1}$ to get same bound 
\begin{equation}\label{eqn: stage two error update}
\norm{D_S^* - S_{t+1}S_{t+1}^\top} \le \left(1-\frac{3\gamma_t\sigma_r}{4}\right) \norm{D_S^*- S_tS_t^\top} + 3\gamma_t \norm{S_tT_t^\top}^2 
\end{equation}
\item $\norm{\cN_t(T_t) \cN_t(T_t)^\top} \le  \norm{T_t}^2(1-\gamma_t \norm{T_t}^2) = \norm{T_tT_t^\top}\left(1-\gamma_t \norm{T_tT_t^\top}\right)$.\\
By definition, 
\begin{equation}
	\cN_t(T_t) = T_t - \gamma_t(T_tT_t^\top T_t + T_tS_t^\top S_t).
\end{equation}
Plug this into $\cN_t(T_t)\cN_t(T_t)^\top$, we obtain
\begin{align}
 \cN_t(T_t)\cN_t(T_t)^\top &=  \left(T_t - \gamma_t(T_tT_t^\top T_t + T_tS_t^\top S_t)\right)\left(T_t - \gamma_t(T_tT_t^\top T_t + T_tS_t^\top S_t)\right)^\top\\
 &=Z_4 + Z_5+Z_6
\end{align}
where 
\begin{equation}
	Z_4 = T_tT_t^\top -2\gamma_t T_tT_t^\top T_tT_t^\top,
\end{equation}
\begin{equation}
	Z_5 =- 2\gamma_t T_tS_t^\top S_tT_t^\top+ \gamma_t^2 T_tS_t^\top S_tS_t^\top S_t T_t^\top,
\end{equation}
and 
\begin{equation}
    Z_6 = \gamma_t^2\left[T_tT_t^\top \left(T_tS_t^\top S_tT_t^\top \right) + \left(T_tS_t^\top S_tT_t^\top \right) T_tT_t^\top\right] 
\end{equation}
We bound each of them separately. By lemma~\ref{lem: S(I-S) svd}, we obtain
\begin{align}
	\norm{Z_4}&\le \norm{T_tT_t^\top - 2\gamma_t T_tT_t^\top T_tT_t^\top}\\
	&\le \norm{T_tT_t^\top}(1-2\gamma_t \norm{T_tT_t^\top})\\
	&= \norm{T_t}^2(1-2\gamma_t \norm{T_t}^2).
\end{align}
On the other hand, 
\begin{align}
	Z_5 &=  - 2\gamma_t T_tS_t^\top S_tT_t^\top +\gamma_t^2 T_tS_t^\top S_tS_t^\top S_t T_t^\top\\
	&=-2\gamma_t T_tS_t^\top(I - \gamma_t S_tS_t^\top) S_tT_t^\top\\
	&\preceq 0.
\end{align}
Furthermore, 
\begin{align}
    \norm{Z_6}&\le \gamma_t^2 \norm{S_t^\top S_t}\norm{T_tT_t^\top}^2\\
    &\le \frac{\gamma_t}{2} \norm{T_tT_t^\top}^2.
\end{align}
Combining, we obtain
\begin{align}
	\norm{\cN_t(T_t)\cN_t(T_t)^\top} &= \norm{Z_4+Z_5+Z_6} \\
	&\le \norm{Z_4+Z_5} + \norm{Z_6}\\
	&\le \norm{Z_4}+\norm{Z_6}\\
	&\le \norm{T_tT_t^\top}(1-\frac{3\gamma_t}{2} \norm{T_tT_t^\top}).
\end{align}
The second inequality follows from the fact that $Z_5 \preceq 0$. In this proof, we only need $\gamma_t \le \frac{0.01}{\sigma_1}$.
\item $\norm{\cM_t(S_t)\cN_t(T_t)^\top} \le (1-\frac{\gamma_t\sigma_r}{3})\norm{S_tT_t^\top}$. \\
By definition of $\cM_t(S_t), \cN_t(T_t)$, we have
\begin{align}
	\cM_t(S_t) \cN_t(T_t)^\top &= \left(S_t - \gamma_t(S_tS_t^\top S_t + S_t T_t^\top T_t - D_S^* S_t)\right)\left(T_t - \gamma_t(T_tT_t^\top T_t + T_t S_t^\top S_t)\right)^\top\\
	&=Z_7 + Z_8 + Z_9,
\end{align}
where 
\begin{equation}
	Z_7= (I - \gamma_t S_tS_t^\top)S_t T_t^\top,
\end{equation}
\begin{equation}
	Z_8 = \gamma_t (D_S^*- S_tS_t^\top)S_tT_t^\top
\end{equation}
and 
\begin{equation}
	Z_9 = -2\gamma_t S_t T_t^\top T_t T_t^\top + \gamma_t^2(S_tS_t^\top S_t + S_t T_t^\top T_t - D_S^* S_t)(T_tT_t^\top T_t + T_tS_t^\top S_t)^\top.
\end{equation}
We bound each of them. By our assumption that $\sigma_r(S_t) \ge \sqrt{\frac{\sigma_r}{2}}$, 
\begin{equation}
	\norm{Z_7} \le (1-\frac{\gamma_t\sigma_r}{2}) \norm{S_tT_t^\top}.
\end{equation}
By the assumption that $\norm{D_S^* - S_tS_t^\top} \le \frac{\sigma_r}{10}$, 
\begin{equation}
	\norm{Z_8}\le \frac{\gamma_t\sigma_r}{10}\norm{S_tT_t^\top}.
\end{equation}
For $Z_9$, we use triangle inequality and get 
\begin{align}
	\norm{Z_9} &\le \norm{2\gamma_t S_t T_t^\top T_t T_t^\top} + \gamma_t^2\norm{(S_tS_t^\top - D_S^*)S_tT_t^\top T_t T_t^\top}\\
	&\qquad +\gamma_t^2\norm{(S_tS_t^\top - D_S^*)S_tS_t^\top S_t T_t^\top)}+ \gamma_t^2 \norm{S_tT_t^\top T_t (T_t^\top T_tT_t^\top + S_t^\top S_t T_t^\top)}\\
	&\overset{(\sharp)}{\le} \frac{2}{100}\gamma_t\sigma_r \norm{S_t T_t^\top} + \gamma_t^2 \frac{\sigma_r}{10}\frac{\sigma_r}{100}\norm{S_tT_T^\top}+4\gamma_t^2 \frac{\sigma_r}{10}\sigma_1\norm{S_tT_t^\top} \\
	&\qquad + \gamma_t^2\left(\frac{\sigma_r}{100}\right)^2 \norm{S_tT_t^\top}+\gamma_t^2\left(\frac{2\sqrt{\sigma_1\sigma_r}}{10}\right)^2 \norm{S_tT_t^\top}\\
	&\overset{(\star)}{\le} \frac{\gamma_t\sigma_r}{20 } \norm{S_t T_t^\top}
\end{align}
In $(\sharp)$, we used the bound that $\norm{S_tT_t^\top}\le \norm{S_t}\norm{T_t} \le \frac{2\sqrt{\sigma_1\sigma_r}}{10}$ and $\norm{T_tT_t^\top}\le \frac{\sigma_r}{100}$. $(\star)$ follows from our assumption that $\gamma_t \le \frac{0.01}{\sigma_1}$.
Combining, we obtain
\begin{align}
	\norm{\cM_t(S_t)\cN_t(T_t)^\top} &\le \norm{Z_7} + \norm{Z_8} + \norm{Z_9}\\
	&\le \left(1- \frac{\gamma_t\sigma_r}{3}\right) \norm{S_tT_t^\top}.
\end{align}
\end{enumerate}

\subsection{Proof of Proposition~\ref{prop: bound on MS NT}}
We prove them one by one.
\begin{enumerate}
	\item $\norm{D_S^* - \cM_t(S_t)S_t^\top} \le (1-\frac{\gamma_t\sigma_r}{2})\norm{D_S^* -S_tS_t^\top}+\gamma_t \norm{S_tT_t^\top}^2$. By definition of $\cM_t(S_t)$, we know that 
	\begin{align}
		\cM_t(S_t)S_t^\top - D_S^* &= S_tS_t^\top -\gamma_t(S_tS_t^\top S_t + S_tT_t^\top T_t - D_S^*S_t)S_t^\top - D_S^*\\
		&=(S_tS_t^\top - D_S^*)(I- \gamma_t S_tS_t^\top) - \gamma_t S_tT_t^\top T_t S_t^\top.
	\end{align}
	By our assumption that $\sigma_r(S_t) \ge \sqrt{\frac{\sigma_r}{2}}$, we know 
	\begin{equation}
		\norm{(S_tS_t^\top - D_S^*)(I- \gamma_t S_tS_t^\top)} \le (1-\frac{\sigma_r}{2})\norm{S_tS_t^\top -D_S^*}.
	\end{equation}
	By triangle inequality, the result follows.
	\item $\norm{\cM_t(S_t)T_t^\top} \le 2\norm{S_t T_t^\top}$. By definition of $\cM_t(S_t)$, we have 
	\begin{align}
		\cM_t(S_t)T_t^\top = S_tT_t^\top - \gamma_t(S_tS_t^\top S_t + S_tT_t^\top T_t - D_S^*S_t)T_t^\top
	\end{align}
	Triangle inequality yields
	\begin{align}
		\gamma_t\norm{(S_tS_t^\top S_t + S_tT_t^\top T_t - D_S^*S_t)T_t^\top}&\le \gamma_t (\norm{S_t}^2 + \norm{T_t}^2 + \norm{D_S^*})\norm{S_tT_t^\top}\\
		&\le \gamma_t (4\sigma_1 + 0.01\sigma_r  + \sigma_1)\norm{S_tT_t^\top}\\
		&\le \norm{S_tT_t^\top}
	\end{align}
	The last inequality follows from our assumption that $\gamma_t \le \frac{0.01}{\sigma_1}$. By triangle inequality again, we obtain
	\begin{equation}
		\norm{\cM_t(S_t)T_t^\top} \le \norm{S_tT_t^\top} + \norm{S_tT_t^\top} = 2\norm{S_t T_t^\top}.
	\end{equation}
	\item $\norm{\cN_t(T_t) S_t^\top} \le \norm{T_tS_t^\top}$. By definition of $\cN_t(T_t)$, 
	\begin{align}
		\cN_t(T_t)S_t^\top &= T_tS_t^\top - \gamma_t(T_tT_t^\top T_t + T_tS_t^\top S_t)S_t^\top\\
		&= \left(\frac{1}{2}I- \gamma_t T_tT_t^\top \right)T_tS_t^\top + T_tS_t^\top\left(\frac{1}{2}I - \gamma_t S_tS_t^\top\right)
	\end{align}
	By triangle inequality,
	\begin{align}
		\norm{\cN_t(T_T)S_t^\top} &\le \norm{\left(\frac{1}{2}I- \gamma_t T_tT_t^\top \right)T_tS_t^\top} +\norm{ T_tS_t^\top\left(\frac{1}{2}I - \gamma_t S_tS_t^\top\right)}\\
		&\le \norm{T_tS_t^\top}.
	\end{align}
	The last inequality follows from the choice of $\gamma_t$ and the fact that $\norm{\left(\frac{1}{2}I- \gamma_t T_tT_t^\top \right)}\le \frac{1}{2}$, $\norm{\left(\frac{1}{2}I- \gamma_t S_tS_t^\top \right)}\le \frac{1}{2}$.
	\item $\norm{\cN_t(T_t) T_t^\top} \le  \norm{T_t}^2(1-\gamma_t \norm{T_t}^2) = \norm{T_tT_t^\top}\left(1-\gamma_t \norm{T_tT_t^\top}\right).$ By definition of $\cN_t(T_t)$, we have 
	\begin{align}
		\cN_t(T_t)T_t^\top &=  T_tT_t^\top - \gamma_t(T_tT_t^\top T_t + T_tS_t^\top S_t)T_t^\top\\
		&= T_tT_t^\top(I -\gamma_t T_tT_t^\top) - \gamma_t T_tS_t^\top S_t T_t^\top \\
		&\preceq T_tT_t^\top(I -\gamma_t T_tT_t^\top) 
	\end{align}
	As a result of lemma~\ref{lem: S(I-S) svd}, we have 
	\begin{align}
		\norm{\cN_t(T_t)T_t^\top} \le \norm{T_tT_t^\top}(1-\gamma_t \norm{T_tT_t^\top}) = \norm{T_t}^2(1-\gamma_t \norm{T_t}^2).
	\end{align}
\end{enumerate}
\subsection{Proof of Proposition~\ref{prop: stage one evolution}}
We prove this proposition by induction. Note that the inequality~\ref{eqn: evolution in stage one} holds trivially when $s=0$. Suppose it holds for $t \ge 0$.
By Proposition~\ref{prop: updates of S,T}, we can write $S_{t+1}$ as 
\begin{align}
	S_{t+1} &= \left(I - \gamma_t S_tS_t^\top + \gamma_t D_S^*\right)S_t -\gamma_t S_tT_t^\top T_t +\gamma_t U^\top \Delta_t F_t \label{eqn:inverse}\\
	&= \left(I+\gamma_t D_S^*\right)S_t (I-\gamma_t S_t^\top S_t)  +\gamma_t^2 D_S^*S_tS_t^\top S_t - \gamma_t S_t T_t^\top T_t+ \gamma_t U^\top\Delta_tF_t \label{eqn: prodsigmar}
\end{align}
These two ways of expressing $S_{t+1}$ are crucial to the proof.\\
For the ease of notation, we introduce some notations. Let
\begin{align}
	H_t &= I- \gamma_t S_tS_t^\top +\gamma_t D_S^*\\
	E_t &= S_tT_t^\top T_t - U^\top \Delta_t F_t
\end{align}
By Proposition~\ref{prop: uniform upper bound} and our assumption that $ (50\sqrt{k}\delta)^{\frac{1}{3}} \le \frac{c_\rho\rho}{2\sqrt{\sigma_1}}$, we have 
\begin{align}
	\norm{E_t} &\le \norm{S_t} \norm{T_t}^2 + \norm{\Delta_t F_t}\\
	&\le 2(c_\rho \rho)^2\sqrt{\sigma_1} + 50\delta \sqrt{k}\sigma_1^\frac{3}{2}\\
	&\le  2(c_\rho \rho)^2\sqrt{\sigma_1} + \frac{(c_\rho \rho)^3}{8}\\
	&\le 3(c_\rho \rho)^2\sqrt{\sigma_1}.
\end{align}
In the last inequality, we used our assumption that $c_\rho\rho \le 0.1\sqrt{\sigma_1}$.
By lemma~\ref{lem: norm of inverse} and our choice of $\gamma_t$, we know $H_t$ is invertible and 
\begin{equation}
	\norm{H_t^{-1}} \le \frac{1}{1-\gamma_t  \norm{S_t}^2 - \gamma_t \norm {D_S^*}} \le \frac{1}{1- 0.04 - 0.01 }\le 2.
\end{equation}
By~\ref{eqn:inverse}, we can write 
\begin{equation}
	S_t = H_t^{-1}S_{t+1} + \gamma_t  H_t^{-1}E_t
\end{equation}
Plug this in to~\ref{eqn: prodsigmar} and rearrange, we get 
\begin{equation}\label{eqn: decomposition into Z1Z_2}
	\left(I- \gamma_t ^2 D_S^* S_tS_t^\top H_t^{-1}\right)S_{t+1} = \underbrace{\left(I+\gamma_t  D_S^*\right)S_t (I-\gamma_t  S_t^\top S_t)}_{Z_1} +  \underbrace{\gamma_t ^3 D_S^* S_tS_t^\top H_t^{-1}E_t - \gamma_t  E_t}_{Z_2}
\end{equation}
Let's consider the $r$-th singular value of both sides. For LHS, by lemma~\ref{lem: sigmar ineq} and lemma~\ref{lem: norm of inverse}
\begin{align}
	\sigma_r(\left(I- \gamma_t ^2 D_S^* S_tS_t^\top H_t^{-1}\right)S_{t+1}) &\le \norm{\left(I- \gamma_t ^2 D_S^* S_tS_t^\top H_t^{-1}\right)} \sigma_r(S_{t+1})\\
	&\le \frac{1}{1- \gamma_t ^2 \norm{D_S^* S_tS_t^\top H_t^{-1}}} \sigma_r(S_{t+1})\\
	&\le \frac{1}{1-8\gamma_t ^2\sigma_1^2} \sigma_r(S_{t+1}).
\end{align}
For RHS, we consider $Z_1$ and $Z_2$ separately. For $Z_1$, by lemma~\ref{lem: S(I-SS) svd}, we have 
\begin{align}
	\sigma_r(Z_1) &\ge \sigma_r(I+\gamma_t  D_S^*) \cdot \sigma_r(S_t(I-\gamma_t  S_t^\top S_t))\\
	&= (1+\gamma_t \sigma_r) \sigma_r(S_t)(1-\gamma_t \sigma_r^2(S_t))
\end{align}
For $Z_2$, by triangle inequality,
\begin{align}
	\norm{\gamma_t ^3 D_S^* S_tS_t^\top H_t^{-1}E_t - \gamma_t  E_t} &\le \gamma_t ^3\norm{D_S^* S_tS_t^\top H_t^{-1}E_t}+\gamma_t  \norm{E_t}\\
	&\le(8\gamma_t ^3\sigma_1^2+\gamma_t ) \norm{E_t}\\
	&\le 3(8\gamma_t ^3\sigma_1^2 +\gamma_t)(c_\rho \rho)^2\sqrt{\sigma_1}.
\end{align}
Combining, by  lemma~\ref{lem: weyl's ineq}, we obtain 
\begin{align}
	&\;\;\;\;\sigma_r( \left(I+\gamma_t  D_S^*\right)S_t (I-\gamma_t  S_t^\top S_t) +  \gamma_t^3 D_S^* S_tS_t^\top H_t^{-1}E_t - \gamma_t  E_t)\\
	&\ge \sigma_r\left(Z_1 \right) - \gamma_t  \norm{E_t} - \norm{Z_2}\\
	&\ge (1+\gamma_t \sigma_r) \sigma_r(S_t)(1-\gamma_t \sigma_r^2(S_t))-3(8\gamma_t ^3\sigma_1^2 +\gamma_t)(c_\rho \rho)^2\sqrt{\sigma_1}.
\end{align}
By induction hypothesis, we know $\sigma_r(S_t) \ge \rho$. Note we assumed that $4c_\rho^2 \rho \le 0.01 \frac{\sigma_r}{\sqrt{\sigma_1}} $, so we have 
\begin{equation}
	3(8\gamma_t ^3\sigma_1^2 +\gamma_t)(c_\rho \rho)^2\sqrt{\sigma_1}\le 4 \gamma_t  (c_\rho \rho)^2\sqrt{\sigma_1} \le 0.01 \gamma_t \sigma_r \sigma_r(S_t)
\end{equation}
Consequently, we get
\begin{align}
	&\;\;\;\;\sigma_r( \left(I+\gamma_t D_S^*\right)S_t (I-\gamma_t S_t^\top S_t) +  \gamma_t^3 D_S^* S_tS_t^\top H_t^{-1}E_t - \gamma_t E_t)\\
	&\ge (1+\gamma_t\sigma_r) \sigma_r(S_t)(1-\gamma_t\sigma_r^2(S_t)) - 0.01 \sigma_r \sigma_r(S_t)\\
	&=\sigma_r(S_t) \left(1+0.99\gamma_t \sigma_r - \gamma_t \sigma_r^2(S_t) - \gamma_t^2 \sigma_r \sigma_r^2(S_t)\right)
\end{align}
Combining the LHS and RHS, we finally get 
\begin{equation}\label{eqn: stageone key ineq}
	\sigma_r(S_{t+1}) \ge (1-8\gamma_t^2\sigma_1^2) \left(1+0.99\gamma_t \sigma_r - \gamma_t \sigma_r^2(S_t) - \gamma_t^2 \sigma_r \sigma_r^2(S_t)\right) \sigma_r(S_t)
\end{equation}
We consider two cases(recall $\sigma_r = \frac{1}{\kappa}$):
\begin{itemize}
	\item $\sigma_r(S_t) \ge \sqrt{\frac{3\sigma_r}{4}}$. By~\ref{eqn: stageone key ineq}, we know that 
	\begin{equation}
		\sigma_r(S_{t+1}) \ge (1-8\gamma_t^2\sigma_1^2)(1-5\gamma_t\sigma_1) \sigma_r(S_t).
	\end{equation}
	Here we used Proposition~\ref{prop: uniform upper bound} to bound $\sigma_r(S_t)$ by $2\sqrt{\sigma_1}$. Since $\gamma_t \le \frac{0.01}{\sigma_1}$, simple calculation shows that 
	\begin{equation}
		\sigma_r(S_{t+1}) \ge (1-8\gamma_t^2\sigma_1^2)(1-5\gamma_t\sigma_1) \sqrt{\frac{3\sigma_r}{4}} \ge \sqrt{\frac{\sigma_r}{2}}.
	\end{equation}
	\item $\sigma_r(S_t) < \sqrt{\frac{3\sigma_r}{4}}$. By~\ref{eqn: stageone key ineq} and induction hypothesis, we know
	\begin{align}
		\sigma_r(S_{t+1}) &\ge (1-8\gamma_t^2\sigma_1^2) \left(1+0.99\gamma_t \sigma_r - \gamma_t \sigma_r^2(S_t) - \gamma_t^2 \sigma_r \sigma_r^2(S_t)\right) \sigma_r(S_t)\\
		&\ge (1-8\gamma_t^2\sigma_1^2) \left(1+\frac{\gamma_t\sigma_r}{5 } \right) \sigma_r(S_t)\\
		&\ge \left(1+\frac{\gamma_t\sigma_r^2 }{6\sigma_1^2 } \right) \sigma_r(S_t)\\
		&\ge \min\{(1+\frac{c_\gamma\sigma_r^2}{6\sigma_1^2})^{t+1}\sigma_r(S_0), \sqrt{\frac{\sigma_r}{2}}\}.
	\end{align}
We used the bound $\gamma_t \ge \frac{c_\gamma\sigma_r}{\sigma_1^2}$ in the last inequality.
\end{itemize}
By induction, we proved inequality~\ref{eqn: evolution in stage one} for $\sigma_r(S_t)$. By our choice of $\cT_1$, it's easy to verify that 
\begin{equation}
	\sigma_r(S_{\cT_1+t}) \ge \sqrt{\frac{\sigma_r}{2}},\qquad \forall t\ge 0.
\end{equation}

\subsection{Proof of Proposition~\ref{prop: stagetwo phaseone}}
We prove it by induction. For the ease of notation, we use index $t$ for $t \ge \cT_1'$ instead of $\cT_1'+t$.
The inequality~\ref{eqn: stagetwo phseone} holds for $t=\cT_1'$ by Proposition~\ref{prop: uniform upper bound} and triangle inequality that 
\begin{equation}
	\norm{S_{\cT_1}S_{\cT_1}^\top - D_S^*} \le \norm{S_{\cT_1}}^2 + \norm{D_S^*} \le 5\sigma_1.
\end{equation}
Suppose that~\ref{eqn: stagetwo phseone} holds for some $t\ge \cT_1'$. By Proposition~\ref{prop: updates of S,T}, we have
\begin{equation}
	S_{t+1} = \cM_t(S_t) + \gamma_t U^\top \Delta_t F_t.
\end{equation}
As a result,
\begin{align}
	S_{t+1}S_{t+1}^\top - D_S^* &= \underbrace{\cM_t(S_t)\cM_t(S_t)^\top - D_S^*}_{Z_1} + \underbrace{\gamma_t(U^\top \Delta_t F_t \cM_t(S_t)^\top + \cM_t(S_t)F_t^\top \Delta_t^\top U)}_{Z_2}\\
	&\qquad + \underbrace{\gamma_t^2 U^\top \Delta_t F_tF_t^\top \Delta_t^\top U}_{Z_3}
\end{align}
By Proposition~\ref{prop: bound on MM NN}, we know 
\begin{align}
	\norm{Z_1} &\le (1-\frac{3\gamma_t\sigma_r}{4})\norm{S_tS_t^\top - D_S^*} + 3\gamma_t \norm{S_tT_t^\top}^2\\
	&\overset{(\sharp)}{\le} (1-\frac{3\gamma_t\sigma_r}{4})\norm{S_tS_t^\top - D_S^*} + 12\gamma_t\sigma_1 (c_\rho \rho)^2
\end{align}
Here $(\sharp)$ follows from Proposition~\ref{prop: uniform upper bound}.\\
On the other hand, it's easy to see $\norm{\cM_t(S_t)} \le 3\sqrt{\sigma_1}$ by its definition and Proposition~\ref{prop: uniform upper bound}. By triangle inequality,
\begin{align}
	\norm{Z_2} &\le 2 \gamma_t \norm{U^\top \Delta_t F_t \cM_t(S_t)^\top}\\
	&\le 2 \gamma_t \norm{\Delta_t} \norm{US_t + VT_t} \norm{\cM_t(S_t)}\\
	&\le 18 \gamma_t \norm{\Delta_t}\sigma_1\\
	&\overset{(\sharp)}{\le} 18\gamma_t \delta \sqrt{k+r} \norm{F_tF_t^\top - \truX}\sigma_1\\
	&\overset{(\star)}{\le} 270\gamma_t \delta \sqrt{k}\sigma_1^2\\
	&\overset{(*)}{\le} \gamma_t (c_\rho \rho)^3 \sqrt{\sigma_1}
\end{align}
Here $(\sharp)$ follows from~\ref{eqn: RIP bound}, $(\star)$ follows from uniform bound $\norm{F_t} \le 3\sqrt{\sigma_1}$, and $(*)$ follows from the assumption that $(50\sqrt{k}\delta)^{\frac{1}{3}}\le \frac{c_\rho \rho}{2\sqrt{\sigma_1}}$.\\
Furthermore, 
\begin{align}
	\norm{Z_3} &\le \gamma_t^2 \norm{\Delta_t}^2 \norm{F_t}^2\\
	&\le 9\gamma_t^2 (10\delta \sqrt{k+r})^2\sigma_1^3\\
	&\le \gamma_t^2 (c_\rho \rho)^6
\end{align}
The last inequality follows simply from our assumption that $(50\sqrt{k}\delta)^{\frac{1}{3}}\le \frac{c_\rho \rho}{2\sqrt{\sigma_1}}$.
Combining, we obtain
\begin{align}
 \norm{S_{t+1}S_{t+1}^\top - D_S^*} &\le \norm{Z_1} + \norm{Z_2} + \norm{Z_3}\\
 &\le  (1-\frac{3\gamma_t\sigma_r}{4})\norm{S_tS_t^\top - D_S^*} + 12\gamma_t(c_\rho \rho)^2\sigma_1 +\gamma_t (c_\rho \rho)^3 \sqrt{\sigma_1} + \gamma_t^2(c_\rho \rho)^6\\
 &\le (1-\frac{3\gamma_t\sigma_r}{4})\norm{S_tS_t^\top - D_S^*} + 13\gamma_t(c_\rho \rho)^2\sigma_1 
\end{align}
In the last inequality, we used $c_\rho \rho \le 0.1\sqrt{\sigma_1}$ and $\gamma_t \le \frac{0.01}{\sigma_1}$.
We consider two cases:
\begin{itemize}
	\item $\norm{S_tS_t^\top -D_S^*}\le \frac{52(c_\rho \rho)^2\sigma_1}{\sigma_r} $. By above inequality, we simply have
	\begin{equation}
		\norm{S_{t+1}S_{t+1}^\top - D_S^*} \le \norm{S_tS_t^\top -D_S^*} + 13 \gamma_t(c_\rho \rho)^2\sigma_1  \le \frac{100(c_\rho \rho)^2\sigma_1}{\sigma_r} .
	\end{equation}
The last inequality follows from the assumption that $\gamma_t \le \frac{0.01}{\sigma_1}\le \frac{0.01}{\sigma_r}$.
\item $\norm{S_tS_t^\top -D_S^*}> \frac{52(c_\rho \rho)^2\sigma_1}{\sigma_r}$. In this case, $13\gamma_t (c_\rho \rho)^2\sigma_1 \le \frac{\gamma_t \sigma_r}{4} \norm{S_tS_t^\top -D_S^*}$. Consequently, 
\begin{align}
	\norm{S_{t+1}S_{t+1}^\top - D_S^*} &\le (1-\frac{3\gamma_t\sigma_r}{4})\norm{S_tS_t^\top -D_S^*} + \frac{\gamma_t\sigma_r}{4} \norm{S_tS_t^\top -D_S^*}\\
	&\le (1-\frac{\gamma_t\sigma_r}{2})\norm{S_tS_t^\top -D_S^*}\\
	&\le \max\{5(1-\frac{c_\gamma \sigma_r^2}{2\sigma_1^2})^{t+1- \cT_1'}, \frac{100(c_\rho \rho)^2\sigma_1}{\sigma_r}\}.
\end{align}
We used the induction hypothesis in the last inequality. By induction, inequality~\ref{eqn: stagetwo phseone} is proved. Moreover, $\cT_2'$ is the smallest integer such that 
\begin{equation}
	5(1-\frac{c_\gamma \sigma_r^2}{2\sigma_1^2})^{t- \cT_1'}\le  \frac{100(c_\rho \rho)^2\sigma_1}{\sigma_r}.
\end{equation}
Therefore, the second claim in Proposition~\ref{prop: stagetwo phaseone} follows from~\ref{eqn: stagetwo phseone}.
\end{itemize}
\subsection{Proof of Proposition~\ref{prop: stagetwo phasetwo}}
We prove it by induction. For the ease of notation, we use index $t$ for $t \ge \cT_1$ instead of $\cT_1+t$. When $t=\cT_1$,~\ref{eqn: stagetwo phasetwo} holds by assumption. Now suppose~\ref{eqn: stagetwo phasetwo} holds for some $t\ge \cT_1$. By induction hypothesis, we have
\begin{equation}
	\norm{S_tT_t^\top} \le 0.01\sigma_r.
\end{equation}
Moreover, 
\begin{align}
	\norm{S_tS_t^\top} \le \norm{D_S^*} + \norm{S_tS_t^\top - D_S^*} \le 1.01 \sigma_1.
\end{align}
Therefore, $\norm{S_t} \le 2\sqrt{\sigma_1}$. Also,
\begin{align}
	\sigma_r(S_tS_t^\top) \ge \sigma_r(D_S^*) - \norm{S_tS_t^\top -D_S^*} \ge \frac{\sigma_r}{2}.
\end{align}
Hence, $\sigma_r(S_t) \ge \sqrt{\frac{\sigma_r}{2}}$ and the conditions of Proposition~\ref{prop: bound on MM NN} and Proposition~\ref{prop: bound on MS NT} are satisfied.
We consider $\norm{S_{t+1}S_{t+1}^\top -D_S^*}$ and $\norm{S_{t+1}T_{t+1}^\top}$ separately.
\begin{enumerate}
	\item For $\norm{S_{t+1}S_{t+1}^\top- D_S^*}$, we apply the same idea as proof of Proposition~\ref{prop: stagetwo phaseone} and write 
	\begin{align}
		S_{t+1}S_{t+1}^\top - D_S^* &= \underbrace{\cM_t(S_t)\cM_t(S_t)^\top - D_S^*}_{Z_1} - \underbrace{\gamma_t(U^\top \Delta_t F_t \cM_t(S_t)^\top + \cM_t(S_t)F_t^\top \Delta_t^\top U)}_{Z_2}\\
		&\qquad + \underbrace{\gamma_t^2 U^\top \Delta_t F_tF_t^\top \Delta_t^\top U}_{Z_3}
	\end{align}
	By Proposition~\ref{prop: bound on MM NN}, we know 
	\begin{align}
		\norm{Z_1} &\le (1-\frac{3\gamma_t\sigma_r}{4})\norm{S_tS_t^\top - D_S^*} + 3\gamma_t \norm{S_tT_t^\top}^2\\
		&\le (1-\frac{3\gamma_t\sigma_r}{4})\norm{S_tS_t^\top - D_S^*} +0.03\gamma_t\sigma_r \norm{S_t T_t^\top}\\
		&\le (1-\frac{3\gamma_t\sigma_r}{4} + 0.03\gamma_t\sigma_r)D_t.
	\end{align}
	On the other hand, 
	 By triangle inequality,
	\begin{align}
		\norm{Z_2} &\le 2 \gamma_t \norm{U^\top \Delta_t F_t \cM_t(S_t)^\top}\\
		&\le 2 \gamma_t \norm{\Delta_t} \norm{US_t + VT_t} \norm{\cM_t(S_t)}\\
		&\le 18 \gamma_t \sigma_1 \norm{\Delta_t}\\
		&\overset{(\sharp)}{\le} 18\gamma_t \sigma_1 \delta \sqrt{k+r} \norm{F_tF_t^\top - \truX}\\ 
	\end{align}
	Here $(\sharp)$ follows from~\ref{eqn: RIP bound}.
	By lemma~\ref{lem: decomposition of FF-X}, we see that 
	\begin{align}
		\norm{F_tF_t^\top - \truX} &\le \norm{S_tS_t^\top - D_S^*} + 2\norm{S_tT_t^\top} + \norm{T_tT_t^\top}\\
		&\le \frac{3\sigma_r}{100} + \frac{\sigma_r}{100}\\
		&\le \frac{4\sigma_r}{100}.
	\end{align}
	Hence, we obtain
	\begin{equation}
		\norm{Z_2} \le \frac{72}{100}\gamma_t\sigma_r \delta \sqrt{k+r}  \sigma_1.
	\end{equation}
	Similarly,
	\begin{align}
		\norm{Z_3} &\le \gamma_t^2 \norm{\Delta_t}^2 \norm{F_t}^2\\
		&\le 9\sigma_1 \gamma_t^2 (\delta \sqrt{k+r})^2 \norm{F_tF_t^\top -\truX}^2\\
		&\le  9\sigma_1 \gamma_t^2 (\delta \sqrt{k+r})^2 (\frac{4\sigma_r}{100})^2\\
		&\le \frac{1}{100}\gamma_t \sigma_r \delta \sqrt{k+r} \sigma_1.
	\end{align}
	In the last inequality, we used our assumption that $\gamma_t \sigma_r \le \gamma_t\sigma_1\le 0.01$ and $\delta\sqrt{k+r} \le 0.001$. Combining, we obtain
	\begin{align}
		\norm{S_{t+1}S_{t+1}^\top -D_S^*}&\le (1-\frac{\gamma_t\sigma_r}{2})D_t+ \gamma_t\sigma_r \delta \sqrt{k+r} \sigma_1.
	\end{align}
	We consider two cases:
	\begin{itemize}
		\item $D_t \le 3 \delta \sqrt{k+r} \sigma_1$. In this case, we simply have 
		\begin{align}
			\norm{S_{t+1}S_{t+1}^\top -D_S^*} \le D_t + 3 \gamma_t \sigma_r \delta \sqrt{k+r} \sigma_1 \le D_t +\delta \sqrt{k+r} \sigma_1  \le 10 \delta \sqrt{k+r}\sigma_1.
		\end{align} 
		\item $ 3 \delta \sqrt{k+r} \sigma_1< D_t \le 10 \delta \sqrt{k+r} \sigma_1$. In this case, we clearly have 
		\begin{align}
			\gamma_t \delta \sqrt{k+r}\sigma_1 \sigma_r \le \frac{\gamma_t\sigma_r}{3}D_t.
		\end{align}
		Consequently, 
		\begin{align}
			\norm{S_{t+1}S_{t+1}^\top -D_S^*} \le (1-\frac{\gamma_t \sigma_r}{6\sigma_1})D_t \le  \max\left\{\left(1-\frac{c_\gamma\sigma_r^2}{6\sigma_1^2}\right)^{t+1-\cT_1}\cdot \frac{\sigma_r}{10} ,10\delta \sqrt{k+r}\sigma_1\right\}.
		\end{align}
		Here we used the induction hypothesis on $D_t$.
	\end{itemize}
	\item For $\norm{S_{t+1}T_{t+1}^\top}$, we can expand it and get
	\begin{align}
		S_{t+1}T_{t+1}^\top &= (\cM_t(S_t) + \gamma_t U^\top \Delta_t F_t)(\cN_t(T_t) + \gamma_t V^\top \Delta_t F_t)^\top\\
		&=\underbrace{\cM_t(S_t)\cN_t(T_t)^\top}_{Z_4} + \underbrace{\gamma_t U^\top \Delta_t F_t \cN_t(T_t) + \gamma_t \cM_t(S_t)F_t^\top \Delta_t^\top V}_{Z_5} \\
		&\qquad + \underbrace{\gamma_t^2 U^\top \Delta_t F_t F_t^\top \Delta_t^\top V}_{Z_6}.
	\end{align}
	By assumption,  we have
	\begin{align}
		\norm{S_t S_t^\top - D_S^*} &\le D_t\\
		&\le \max\{\frac{\sigma_r}{100}, 10\delta \sqrt{k+r}\sigma_1\}\\
		& \le \frac{\sigma_r}{100}.
	\end{align}By Proposition~\ref{prop: bound on MM NN}, we know 
	\begin{equation}
		\norm{Z_4} \le (1-\frac{\gamma_t \sigma_r}{3})\norm{S_t T_t^\top}\le (1-\frac{\gamma_t \sigma_r}{3})D_t
	\end{equation}
	On the other hand, it's easy to see the $\norm{\cM_t(S_t)} \le 3\sqrt{\sigma_1}$ and $\norm{\cN_t(T_t)}\le \sqrt{\sigma_1}$, by triangle inequality and the same argument as $\norm{S_{t+1}S_{t+1}^\top - D_S^*}$, 
	\begin{align}
		\norm{Z_5}&\le \gamma_t \left(\norm{F_t} \norm{\cN_t(T_t)}+ \norm{F_t}\norm{\cM_t(S_t)} \right)\norm{\Delta_t}\\
		&\le 12\gamma_t \sigma_1 \norm{\Delta_t}\\
		&\le 12\gamma_t \sigma_1 \delta\sqrt{k+r}\norm{F_tF_t^\top -\truX}\\
		&\le \frac{48}{100}\gamma_t \sigma_r \delta \sqrt{k+r}\sigma_1.
	\end{align}
We used $\norm{F_tF_t^\top - \truX} \le \frac{4\sigma_r}{100}$, which was proved above.
	Similar as calculation for $\norm{S_{t+1}S_{t+1}^\top -D_S^*}$, we have 
	\begin{equation}
		\norm{Z_6} \le \frac{1}{100}\gamma_t \sigma_r\delta \sqrt{k+r}\sigma_1.
	\end{equation}
	Combining, we obtain
	\begin{align}
		\norm{S_{t+1}T_{t+1}^\top} &\le \norm{Z_4} + \norm{Z_5}+\norm{Z_6}\\
		&\le \left(1-\frac{\gamma_t\sigma_r}{3}\right) D_t + \gamma_t \sigma_r\delta \sqrt{k+r}\sigma_1.
	\end{align}
	We consider two cases:
	\begin{itemize}
		\item $D_t \le 6 \delta \sqrt{k} \sigma_1$. In this case, we simply have 
		\begin{align}
			\norm{S_{t+1}T_{t+1}^\top} \le D_t +\gamma_t \sigma_r\delta \sqrt{k+r} \sigma_1  \le D_t +\delta \sqrt{k+r} \sigma_1  \le 10 \delta \sqrt{k+r}\sigma_1.
		\end{align} 
		\item $ 6 \delta \sqrt{k+r} \sigma_1< D_t \le 10 \delta \sqrt{k+r} \sigma_1$. In this case, we clearly have 
		\begin{align}
			\gamma_t  \sigma_r \delta \sqrt{k+r}\sigma_1\le \frac{\gamma_t\sigma_r}{6}D_t.
		\end{align}
		Consequently, 
		\begin{align}
			\norm{S_{t+1}T_{t+1}^\top} \le (1-\frac{\gamma_t \sigma_r}{6})D_t \le  \max\left\{\left(1-\frac{c_\gamma\sigma_r^2}{6\sigma_1^2}\right)^{t+1-\cT_1}\cdot \frac{\sigma_r}{10} ,10\delta \sqrt{k+r}\sigma_1\right\}.
		\end{align}
		Here we used the induction hypothesis on $D_t$.
	\end{itemize}
\end{enumerate}
Combining, we see that 
\begin{equation}
	D_{t+1} \le \max\left\{\left(1-\frac{c_\gamma\sigma_r^2}{6\sigma_1^2}\right)^{t+1-\cT_1}\cdot \frac{\sigma_r}{10} ,10\delta \sqrt{k+r}\sigma_1\right\}.
\end{equation}
So the induction step is proved. Note that $\cT_2$ is chosen to be the smallest integer $t$ that  
\begin{equation}
	\left(1-\frac{c_\gamma\sigma_r^2}{6\sigma_1^2}\right)^{t-\cT_1}\cdot \frac{\sigma_r}{10} \le 10\delta\sqrt{k+r}\sigma_1,
\end{equation}
the second part of Proposition~\ref{prop: stagetwo phasetwo} follows.
\subsection{Proof of Proposition~\ref{prop: stagethree}}
The proof is inspired by~\cite{zhuo2021computational}. By our assumption that $E_{t}\le 0.01\sigma_r$, we have 
\begin{equation}
	\norm{S_tS_t^\top} \le \norm{S_t S_t^\top -D_S^*} + \norm{D_S^*} \le 1.01\sigma_1.
\end{equation}
As a result, $\norm{S_t} \le 2\sqrt{\sigma_1}$. Similarly,
\begin{equation}
	\norm{T_t}\le \sqrt{\norm{T_tT_t^\top}} \le 0.1\sqrt{\sigma_r}.
\end{equation}
Moreover, 
\begin{equation}
	\sigma_r(S_tS_t^\top) \ge \sigma_r(D_S^*) - \norm{S_tS_t^\top - D_S^*} \ge \frac{\sigma_r}{2}.
\end{equation}
We obtain 
\begin{equation}
	\sigma_r(S_t) \ge \sqrt{\frac{\sigma_r}{2}}.
\end{equation}
Thus, $S_t,T_t$ satisfy all the conditions in Proposition~\ref{prop: bound on MM NN} and Proposition~\ref{prop: bound on MS NT}.
We will bound $\norm{S_{t+1}S_{t+1}^\top -D_S^*}$, $\norm{S_{t+1}T_{t+1}^\top}$, $\norm{T_{t+1}T_{t+1}^\top}$ separately.
\begin{itemize}
	\item $\norm{S_{t+1}S_{t+1}^\top - D_S^*}$. Simple algebra yields
		\begin{align}
		S_{t+1}S_{t+1}^\top - D_S^* &= \underbrace{\cM_t(S_t)\cM_t(S_t)^\top - D_S^*}_{Z_1} + \underbrace{\gamma_t(U^\top \Delta_t F_t \cM_t(S_t)^\top + \cM_t(S_t)F_t^\top \Delta_t^\top U)}_{Z_2}\\
		&\qquad + \underbrace{\gamma_t^2 U^\top \Delta_t F_tF_t^\top \Delta_t^\top U}_{Z_3}
	\end{align}
By Proposition~\ref{prop: bound on MM NN}, we obtain
	\begin{align}
		\norm{Z_1} &\le (1-\frac{3\gamma_t\sigma_r}{4})\norm{S_tS_t^\top - D_S^*} + 3\gamma_t \norm{S_tT_t^\top}^2\\
		&\overset{(\sharp)}{\le} (1-\frac{3\gamma_t\sigma_r}{4})\norm{S_tS_t^\top - D_S^*} +0.03\gamma_t \sigma_r \norm{S_t T_t^\top}\\
		&\le (1-\frac{3\gamma_t\sigma_r}{4}+0.03\gamma_t \sigma_r)E_t.
	\end{align}
	In $(\sharp)$, we used our assumption that $\norm{S_t T_t^\top}\le 0.01\sigma_r$.
	On the other hand, it's easy to see $\norm{\cM_t(S_t)} \le 3\sqrt{\sigma_1}$ by its definition and the fact that $\norm{S_t}\le 2\sqrt{\sigma_r}$. By triangle inequality,
	\begin{align}
		\norm{Z_2} &\le 2 \gamma_t \norm{U^\top \Delta_t F_t \cM_t(S_t)^\top}\\
		&\le 2 \gamma_t \norm{\Delta_t} \norm{US_t + VT_t} \norm{\cM_t(S_t)}\\
		&\le 18 \gamma_t \sigma_1 \norm{\Delta_t}\\
		&\overset{(\sharp)}{\le} 18\gamma_t \sigma_1 \delta \sqrt{k+r} \norm{F_tF_t^\top - \truX}\\ 
		&\overset{(\star)}{\le} 0.018\gamma_t \sigma_r \norm{F_tF_t^\top - \truX} 
	\end{align}
	Here $(\sharp)$ follows from~\ref{eqn: RIP bound}. $(\star)$ follows from our assumption that $\delta \sqrt{k+r} \le \frac{0.001\sigma_r}{\sigma_1}$.
	By lemma~\ref{lem: decomposition of FF-X}, we see that 
	\begin{align}
		\norm{F_tF_t^\top - \truX} &\le \norm{S_tS_t^\top - D_S^*} + 2\norm{S_tT_t^\top} + \norm{T_tT_t^\top}\\
		&\le 4 E_t
	\end{align}
	Hence, we obtain
	\begin{equation}
		\norm{Z_2} \le 0.1 \gamma_t \sigma_r E_t.
	\end{equation}
	Similarly,
	\begin{align}
		\norm{Z_3} &\le \gamma_t^2 \norm{\Delta_t}^2 \norm{F_t}^2\\
		&\le 9\sigma_1 \gamma_t^2 (\delta \sqrt{k+r})^2 \norm{F_tF_t^\top -\truX}^2\\
		&\le  144\sigma_1 \gamma_t^2 (\delta \sqrt{k+r})^2 E_t^2\\
		&\le 0.1\gamma_t \sigma_r E_t.
	\end{align}
	In the last inequality, we used our assumption that $\delta \sqrt{k+r} \le \frac{0.
		001\sigma_r}{\sigma_1} \le 0.001$, $\gamma_t\le \frac{0.01}{\sigma_1}$ and $\norm{E_t} \le 0.01\sigma_r$. Combining, we obtain
	\begin{align}
		\norm{S_{t+1}S_{t+1}^\top -D_S^*}&\le \norm{Z_1} + \norm{Z_2} + \norm{Z_3}\\
		&\le (1-\frac{\gamma_t\sigma_r}{2})E_t
	\end{align}
\item $\norm{S_{t+1} T_{t+1}^\top}$. We can expand it and get
\begin{align}
	S_{t+1}T_{t+1}^\top &= (\cM_t(S_t) + \gamma_t U^\top \Delta_t F_t)(\cN_t(T_t) + \gamma_t V^\top \Delta_t F_t)^\top\\
	&=\underbrace{\cM_t(S_t)\cN_t(T_t)^\top}_{Z_4} + \underbrace{\gamma_t U^\top \Delta_t F_t \cN_t(T_t) + \gamma_t \cM_t(S_t)F_t^\top \Delta_t^\top V}_{Z_5} \\
	&\qquad + \underbrace{\gamma_t^2 U^\top \Delta_t F_t F_t^\top \Delta_t^\top V}_{Z_6}.
\end{align}
By Proposition~\ref{prop: bound on MM NN}, we know 
\begin{equation}
	\norm{Z_4} \le (1-\frac{\gamma_t \sigma_r}{3})\norm{S_t T_t^\top}\le (1-\frac{\gamma_t \sigma_r}{3})E_t
\end{equation}
On the other hand, we see that $\norm{\cM_t(S_t)} \le 3\sqrt{\sigma_1}$ and $\norm{\cN_t(T_t)}\le \sqrt{\sigma_1}$(by bound on $S_t$ and $T_t$ and the update rule), by triangle inequality and the same argument as $\norm{S_{t+1}S_{t+1}^\top - D_S^*}$, 
\begin{align}
	\norm{Z_5}&\le \gamma_t \left(\norm{F_t} \norm{\cN_t(T_t)}+ \norm{F_t}\norm{\cM_t(S_t)} \right)\norm{\Delta_t}\\
	&\le 12\gamma_t \sigma_1 \norm{\Delta_t}\\
	&\le 12\gamma_t \sigma_1 \delta\sqrt{k+r}\norm{F_tF_t^\top -\truX}\\
	&\le 0.05\gamma_t  \sigma_r E_t.
\end{align}
Same as calculation for $\norm{S_{t+1}S_{t+1}^\top -D_S^*}$, we have 
\begin{equation}
	\norm{Z_6} \le  0.1\gamma_t  \sigma_r E_t.
\end{equation}
Combining, we obtain
\begin{align}
	\norm{S_{t+1}T_{t+1}^\top} &\le \norm{Z_4} + \norm{Z_5}+\norm{Z_6}\\
	&\le \left(1-\frac{\gamma_t\sigma_r}{6}\right) E_t.
\end{align}
\item $\norm{T_{t+1}T_{t+1}^\top}$. We expand it and obtain
\begin{align}
	T_{t+1} T_{t+1}^\top &= (\cN_t(T_t) + \gamma_t V^\top \Delta_t F_t)(\cN_t(T_t) + \gamma_t V^\top \Delta_t F_t)^\top\\
	&\le \underbrace{\cN_t(T_t)\cN_t(T_t)^\top}_{Z_7} +\underbrace{\gamma_t V^\top \Delta_t F_t \cN_t(T_t)^\top + \gamma_t \cN_t(T_t) F_t^\top \Delta_t^\top V}_{Z_8}\\
	&\qquad + \underbrace{\gamma_t^2 V^\top \Delta_t F_tF_t^\top \Delta_t^\top V}_{Z_9}
\end{align}
By Proposition~\ref{prop: bound on MM NN}, 
\begin{align}
	\norm{Z_7} \le \norm{T_tT_t^\top}(1-2\gamma_t\norm{T_tT_t^\top}) \le E_t(1-2\gamma_t E_t).
\end{align}
The last inequality follows from the fact that $x \rightarrow x(1-2\gamma_t x)$ is non-decreasing on interval $[0,\frac{1}{4\gamma_t}]$.
On the other hand,
\begin{align}
	V^\top \Delta_t F_t \cN_t(T_t)^\top &=V^\top \Delta_t (US_t + VT_t)\cN_t(T_t)^\top\\
	&= V^\top \Delta_t U S_t \cN_t(T_t)^\top + V^\top \Delta_t VT_t\cN_t(T_t)^\top
\end{align}
By Proposition~\ref{prop: bound on MS NT}, we obtain
\begin{align}
	\norm{V^\top \Delta_t F_t \cN_t(T_t)^\top}&\le \norm{V^\top \Delta_t U S_t \cN_t(T_t)^\top} + \norm{V^\top \Delta_t V T_t \cN_t(T_t)^\top}\\
	&\le \left(\norm{S_t\cN_t(T_t)^\top} + \norm{T_t\cN_t(T_t)^\top}\right)\norm{\Delta_t}\\
	&\le \left(\norm{S_t T_t^\top} + \norm{T_tT_t^\top}\right) \delta \sqrt{k+r} \norm{F_tF_t^\top - \truX}\\
	&\le 8\delta \sqrt{k+r} E_t^2.\\
	&\le 0.01 E_t^2
\end{align}
Consequently,
\begin{align}
	\norm{Z_8} \le 2\gamma_t \norm{V^\top \Delta_t F_t\cN_t(T_t)^\top} \le 0.02 \gamma_t E_t^2.
\end{align}
Furthermore, 
\begin{align}
	\norm{Z_9}&\le \gamma_t^2 \norm{F_t}^2 \norm{\Delta_t}^2\\
	&\le 9\gamma_t^2 \sigma_1 (\delta \sqrt{k+r})^2 \norm{F_tF_t^\top - \truX}^2\\
	&\le 144 \gamma_t^2 \sigma_1 (\delta \sqrt{k+r})^2 E_t^2\\
	&\le 0.1\gamma_t E_t^2.
\end{align}
In the last inequality, we used our assumption that $\gamma_t \le 0.01\sigma_1$ and $\delta \sqrt{k+r} \le 0.001$.
Combining, we obtain
\begin{align}
	\norm{T_{t+1} T_{t+1}^\top} \le E_t(1-\gamma_t E_t).
\end{align}
\end{itemize}
The result follows.
\subsection{Proof of Proposition~\ref{prop: stagethree k=r}}
The proof of this proposition has lots of overlap with Proposition~\ref{prop: stagethree}. 
By our assumption that $E_{t}\le 0.01\sigma_r$, we have 
\begin{equation}
	\norm{S_tS_t^\top} \le \norm{S_t S_t^\top -D_S^*} + \norm{D_S^*} \le 1.01\sigma_1.
\end{equation}
As a result, $\norm{S_t} \le 2\sqrt{\sigma_1}$. Similarly,
\begin{equation}
	\norm{T_t}\le \sqrt{\norm{T_tT_t^\top}} \le 0.1\sqrt{\sigma_r}.
\end{equation}
Moreover, 
\begin{equation}
	\sigma_r(S_tS_t^\top) \ge \sigma_r(D_S^*) - \norm{S_tS_t^\top - D_S^*} \ge \frac{\sigma_r}{2}.
\end{equation}
We obtain 
\begin{equation}
	\sigma_r(S_t) \ge \sqrt{\frac{\sigma_r}{2}}.
\end{equation}
Thus, $S_t,T_t$ satisfy all the conditions in Proposition~\ref{prop: bound on MM NN} and Proposition~\ref{prop: bound on MS NT}.
We will bound $\norm{S_{t+1}S_{t+1}^\top -D_S^*}$, $\norm{S_{t+1}T_{t+1}^\top}$, $\norm{T_{t+1}T_{t+1}^\top}$ separately. Note that the proof of Proposition~\ref{prop: stagethree} doesn't use $k>r$, so it also holds for the case when $k=r$. So, we already have 
\begin{equation}
	\norm{S_{t+1}S_{t+1}^\top -D_S^*} \le (1-\frac{\gamma_t \sigma_r}{2})E_t
\end{equation}
and
\begin{equation}
	\norm{S_{t+1}T_{t+1}^\top} \le (1-\frac{\gamma_t\sigma_r}{3})E_t.
\end{equation}
Next, we obtain a better bound for $\norm{T_{t+1}T_{t+1}^\top}$. We expand $T_{t+1}T_{t+1}^\top$ and obtain
\begin{align}
	T_{t+1} T_{t+1}^\top &= (\cN_t(T_t) + \gamma_t V^\top \Delta_t F_t)(\cN_t(T_t) + \gamma_t V^\top \Delta_t F_t)^\top\\
	&= \underbrace{\cN_t(T_t)\cN_t(T_t)^\top}_{Z_1} +\underbrace{\gamma_t V^\top \Delta_t F_t \cN_t(T_t)^\top + \gamma_t \cN_t(T_t) F_t^\top \Delta_t^\top V}_{Z_2}\\
	&\qquad + \underbrace{\gamma_t^2 V^\top \Delta_t F_tF_t^\top \Delta_t^\top V}_{Z_3}
\end{align}
By definition, 
\begin{equation}
	\cN_t(T_t) = T_t - \gamma_t(T_tT_t^\top T_t + T_tS_t^\top S_t).
\end{equation}
Plug this into $\cN_t(T_t)\cN_t(T_t)^\top$, we obtain
\begin{align}
	Z_1&= \cN_t(T_t)\cN_t(T_t)^\top\\
	 &=  \left(T_t - \gamma_t(T_tT_t^\top T_t + T_tS_t^\top S_t)\right)\left(T_t - \gamma_t(T_tT_t^\top T_t + T_tS_t^\top S_t)\right)^\top\\
	&=Z_4 + Z_5,
\end{align}
where 
\begin{equation}
	Z_4 = T_tT_t^\top -2\gamma_t T_tT_t^\top T_tT_t^\top- \gamma_t T_tS_t^\top S_tT_t^\top
\end{equation}
and 
\begin{equation}
	Z_5 =- \gamma_t T_tS_t^\top S_tT_t^\top+ \gamma_t^2(T_tT_t^\top T_t + T_tS_t^\top S_t)(T_tT_t^\top T_t + T_tS_t^\top S_t)^\top.
\end{equation}
We bound each of them separately. Since $k=r$, $S_t^\top S_t$ is a $r$-by-$r$. Moreover, 
\begin{equation}
\sigma_r(S_t^\top S_t) = \sigma_r(S_t)^2\ge\frac{\sigma_r}{2},
\end{equation}
By $\gamma_t \le \frac{0.01}{\sigma_1}$,
\begin{align}
	\norm{I-\gamma_tS_t^\top S_t - 2\gamma_tT_tT_t^\top} &\le 	\norm{I-\gamma_tS_t^\top S_t}\\
	&\le 1-\frac{\gamma_t \sigma_r}{2}.
\end{align} 
Consequently,
\begin{align}
	\norm{Z_4}&= \norm{T_t(I -\gamma_t S_t^\top S_t- 2\gamma_t T_t^\top T_t)T_t^\top}\\
	&\le \norm{T_t}^2\norm{(I -\gamma_t S_t^\top S_t- 2\gamma_t T_t^\top T_t)}\\
	&\le (1-\frac{\gamma_t \sigma_r}{2} )\norm{T_t}^2.
\end{align}
In addition, 
\begin{align}
	Z_5 &=  - \gamma_t T_tS_t^\top S_tT_t^\top + \gamma_t^2\left[T_tT_t^\top \left(T_tS_t^\top S_tT_t^\top \right) + \left(T_tS_t^\top S_tT_t^\top \right) T_tT_t^\top\right] +\gamma_t^2 T_tS_t^\top S_tS_t^\top S_t T_t^\top\\
	&\preceq (-\gamma_t + \frac{2}{100}\gamma_t^2\sigma_r+ 4\sigma_1\gamma_t^2)T_tS_t^\top S_t T_t^\top\\
	&\preceq 0
\end{align}
Combining, we obtain
\begin{align}
	\norm{\cN_t(T_t)\cN_t(T_t)^\top} \le \norm{Z_4} \le  (1-\frac{\gamma_t\sigma_r}{2})\norm{T_tT_t^\top}.
\end{align}
On the other hand, we see that $\norm{\cM_t(S_t)} \le 3\sqrt{\sigma_1}$ and $\norm{\cN_t(T_t)}\le \sqrt{\sigma_1}$(by bound on $S_t$ and $T_t$ and the update rule). As a result,
\begin{align}
	\norm{V^\top \Delta_t F_t \cN_t(T_t)^\top}&\le \norm{F_t} \norm{\cN_t(T_t)} \norm{\Delta_t}\\
	&\le (\norm{S_t}+\norm{T_t})\norm{\cN_t(T_t)} \delta \sqrt{k+r} \norm{F_tF_t^\top - \truX}\\
	&\le 3\sigma_1\delta \sqrt{k+r} \norm{F_tF_t^\top -\truX}.\\
	&\le 12\sigma_1 \delta \sqrt{k+r}E_t.
\end{align}
Consequently,
\begin{align}
	\norm{Z_2} \le 2\gamma_t \norm{V^\top \Delta_t F_t\cN_t(T_t)^\top} \le 0.03 \gamma_t \sigma_r E_t
\end{align}
Furthermore, 
\begin{align}
	\norm{Z_3}&\le \gamma_t^2 \norm{F_t}^2 \norm{\Delta_t}^2\\
	&\le 9\gamma_t^2 \sigma_1 (\delta \sqrt{k+r})^2 \norm{F_tF_t^\top - \truX}^2\\
	&\le 144 \gamma_t^2 \sigma_1 (\delta \sqrt{k+r})^2 E_t^2\\
	&\le 0.01\gamma_t \sigma_rE_t.
\end{align}
In the last inequality, we used our assumption that $\gamma_t \le 0.01\sigma_1$ and $\delta \sqrt{k+r} \le 0.001$.
Combining, we obtain
\begin{align}
	\norm{T_{t+1} T_{t+1}^\top} &\le \norm{Z_1}+\norm{Z_2}+\norm{Z_3}\\
	&\le (1-\frac{\gamma_t\sigma_r}{3})\norm{T_tT_t^\top}
\end{align}

\section{Proof of RDPP}
    \label{sec: Proof of RDPP random}
    Throughout this section, we denote
\begin{align*}
	\bS \;:=\;\{X \in \mathcal{S}^{d\times d} \colon \fnorm{X}=1\},\quad \bS_r\;:=\; \{X \in \mathcal{S}^{d\times d} \colon \fnorm{X}=1, \rank(X)\le r\}.
\end{align*}
Here we split the Proposition~\ref{prop: main RDPP for two models} into two parts and prove them separately. For the ease of notation, we use $r$ to denote the rank, instead of $k'$.
\begin{proposition}\label{prop: sign-RIP}
	Assume that the sensing matrix $A_i\overset{i.i.d.}{\sim}\textit{GOE}(d)$,\footnote{Gaussian orthogonal ensemble(GOE): $A$ is symmetric with $A_{ij}=A_{ji}\sim N(0,\frac{1}{2})$ for $i\neq j$ and $A_{ii}\sim N(0,1)$ independently.} and the corruption is from model~\ref{md: randomCorruption}. Then RDPP holds with parameters $(r,\delta)$ and a scaling function $\psi(X) = \frac{1}{m}\sum_{i=1}^{m}\sqrt{\frac{2}{\pi}}\left(1-p + p\EE_{s_i\sim \mathbb{P}_i}\left[\exp(-\frac{s_i^2}{2\fnorm{X}^2})\right]\right)$ with probability at least $1-Ce^{-cm\delta^4}$, given $m \gtrsim \frac{dr \left(\log(\frac{1}{\delta}) \vee 1\right)}{\delta^4}$.
\end{proposition}
\begin{proposition}\label{prop: RDPP arbitrary}
    Assume that the sensing matrices $\{A_i\}_{i=1}^m$ have i.i.d. standard Gaussian entries, and the corruption is from model~\ref{md: arbitraryCorruption}. Moreover, we modify function $\sign(x)$ such that $\sign(x) = \begin{cases}
	\{-1\} & x<0\\
	\{-1,1\} & x=0\\
	\{1\} & x>0
\end{cases}$ . Then, RDPP-II holds with parameter $(r,\delta+ 3\sqrt{\frac{dp}{m}} +3p)$ and a scaling function $\psi(X)=\sqrt{\frac{2}{\pi}}$ with probability at least $1-\exp(-(pm+d)) - \exp(-c'm\delta^4)$, given $m \gtrsim \frac{dr \left(\log(\frac{1}{\delta}) \vee 1\right)}{\delta^4}$.
\end{proposition}
\subsection{Proof of Proposition~\ref{prop: sign-RIP}}
In the probability bounds that we obtained, the $c$ might be different from bounds to bounds, but they are all universal constants.
\nct
{
\begin{lem}\label{lem: expection of subgradient}
	Suppose that we are under Model~\ref{md: randomCorruption}. Then, for every nonzero $X\in \bS^{d\times d}$, and  every $D \in \cD(X)$, the expectation $\expect{D}$ is 
	\begin{equation}
		\expect{D} = \psi(X) \frac{X}{\fnorm{X}},\,
	\text{where $\psi(X) = \frac{1}{m}\sum_{i=1}^{m}\sqrt{\frac{2}{\pi}} \left(1-p + p\mathbb{E}_{s_i \sim \mathbb{P}_i}\left[e^{-s_i^2/2\fnorm{X}^2}\right]\right)$.}
		\end{equation}
\end{lem}
}
\begin{proof}
	We may drop the subscript under expectation when the distribution is clear. Firstly, we show that for any $X, Y \in \bS^{d\times d}$, if $s$ follows distribution $\mathbb{P}$, $A$ is GOE matrix and they are independent, then
	\begin{equation}\label{eqn: expectation of subgradient}
		\expect{\sign(\dotp{A,X} -s) \dotp{A,Y}} = \sqrt{\frac{2}{\pi}} \expect{e^{-s^2/2\fnorm{X}^2}} \dotp{\frac{X}{\fnorm{X}}, Y}.
	\end{equation}
	In this section, $\sign(\dotp{A,x}-s)$ should be thought of as any element chosen from the corresponding set. There is ambiguity when $\dotp{A,x}-s=0$, but this happens with probability 0, so it won't affect the result. Without loss of generality, we assume $\fnorm{X} = \fnorm{Y}=1$. To leverage the fact that $A$ is GOE matrix, we denote $u = \dotp{A,X}$, $v= \dotp{A,Y}$ and $\rho = \text{cov}(u,v)$. Simple calculation yields $u \sim N(0,1)$, $v\sim N(0,1)$ and $\rho = \dotp{X,Y}$. By coupling, we can write $v= \rho u + \sqrt{1-\rho^2}w$, where $w$ is another standard Gaussian independent of others. 
	Using the definition of $u,v,\rho,w$, we have 
	\nct{ 
	\begin{align}
		&\expect{\sign(\dotp{A,X} -s) \dotp{A,Y}} 
		=\expect{\sign(u-s) v} =\rho \expect{\sign(u-s) u}.\label{eq: expAX-s1}
	\end{align}
	}
	We continue the above equality using the properties of Gaussian:
\nct{	
\begin{align}
		\rho \expect{\sign(u-s) u}
		=&\rho \mathbb{E}_s\left[\int_{s}^{+\infty} u \frac{1}{\sqrt{2\pi}}e^{-u^2/2} du - \int_{-\infty}^{s}u \frac{1}{\sqrt{2\pi}}e^{-u^2/2} du\right]\\
		\overset{(a)}{=}&\rho \mathbb{E}_s\left[\int_{s}^{+\infty} u \frac{1}{\sqrt{2\pi}}e^{-u^2/2} du + \int_{-s}^{+\infty}u \frac{1}{\sqrt{2\pi}}e^{-u^2/2} du\right]\\
		\overset{(b)}{=}& 2\rho  \mathbb{E}_s\left[\int_{|s|}^{+\infty} u \frac{1}{\sqrt{2\pi}}e^{-u^2/2} du\right]\\
		=&\sqrt{\frac{2}{\pi}} \rho \mathbb{E}_s\left[ \int_{|s|}^{+\infty} d(-e^{-u^2/2})\right]
		=\sqrt{\frac{2}{\pi}} \rho \mathbb{E}_s\left[e^{-s^2/2}\right]. \label{eq: expAX-s2}
	\end{align}
}
\nct{
Here, in the steps $(a)$, we do a change of variable $u\mapsto -u$. In the step $(b)$, we use the fact that the density of standard Gaussian is symmetric. 
	Recall that $\rho = \dotp{X,Y}$. Hence, the equation~\eqref{eqn: expectation of subgradient} follows from \eqref{eq: expAX-s1} - \eqref{eq: expAX-s2}.} 
Since it holds for all symmetric $Y$, we obtain
	\begin{equation}
		\expect{\sign(\dotp{A,X} -s) A} = \sqrt{\frac{2}{\pi}} \expect{e^{-s^2/2\fnorm{X}^2}} \frac{X}{\fnorm{X}}.
	\end{equation}
	On the other hand, if we apply the above result to the case when $s\equiv 0$, we get
	\begin{equation}
		\expect{\sign(\dotp{A,X} ) A} = \sqrt{\frac{2}{\pi}}  \frac{X}{\fnorm{X}}.
	\end{equation}
	When $s_i$'s are form model~\ref{md: randomCorruption}, by tower property and results above,
	\begin{align}
		\expect{\sign(\dotp{A_i,X} -s_i) A_i }&=\expect{\EE[\sign(\dotp{A_i,X} -s_i) A_i \mid s_i]}\\
		&= (1-p)\expect{\sign(\dotp{A_i,X} ) A_i} + p\EE_{s_i\sim \mathbb{P}_i, A_i}\left[\sign(\dotp{A_i,X} -s_i) A_i\right]\\
		&=\sqrt{\frac{2}{\pi}}\left((1-p) +p\expect{e^{-s^2/2\fnorm{X}^2}} \right) \frac{X}{\fnorm{X}}
	\end{align}
	The lemma follows from the linearity of expectation.
\end{proof}
\nct{Lemma \ref{lem: expection of subgradient} is an analogue of \cite[Lemma 3]{ma2021implicit}. Note that the function $\psi$ is not necessarily the quantity $\sqrt{\frac{2}{\pi}}\left((1-p) +p\expect{e^{-s_i^2/2\fnorm{X}^2}} \right) \frac{X}{\fnorm{X}}$, which appears in \cite[Lemma 3]{ma2021implicit}, since  the corruptions are not assumed to be i.i.d in this paper.
}

Next, we prove a probability bound that holds for any fixed $X,Y \in \bS$.
\nct{
\begin{lem}\label{lem: concentration for fixed XY}
Under Model~\ref{md: randomCorruption}, there exists a universal constant $c$ such that for any $\delta >0, X\in \bS, Y\in \bS$, with probablity at most $ 2 e^{-cm\delta^2}$, the following event happens 
	\begin{equation}
		\abs{\frac{1}{m}\sum_{i=1}^{m}\sign\left(\dotp{A_i,X} -s_i\right) \dotp{A_i,Y} -\psi(X) \dotp{X,Y}} > \delta ,
	\end{equation}
		where $\psi(X) = \frac{1}{m}\sum_{i=1}^{m}\sqrt{\frac{2}{\pi}} \left(1-p + p\mathbb{E}_{s_i \sim \mathbb{P}_i}\left[e^{-s^2_i/2\fnorm{X}^2}\right]\right)$.
\end{lem}
}
\begin{proof}
	We first show that $\sign(\dotp{A_i,X} -s_i) \dotp{A_i,Y}$ is a sub-Gaussian random variable. Let consider the Orlicz norm~
	\cite{wainwright2019high} with $\psi_2(x) = e^{x^2}-1$.
	$\dotp{A_i,Y}$ is standard Gaussian, so it has sub-Gaussian parameter $1$. By property of Orlicz norm, $\norm{ \dotp{A_i,Y}}_{\psi_2}\le C$ for some constant $C$. Moreover, $\abs{\sign(\dotp{A_i,X} -s_i)}\le 1$, so
	\begin{equation}
		\norm{\sign(\dotp{A_i,X} -s_i) \dotp{A_i,Y}}_{\psi_2} \le \norm{ \dotp{A_i,Y}}_{\psi_2}\le C.
	\end{equation} 
	By property of Orlicz norm again, we know $\sign(\dotp{A_i,X} -s_i) \dotp{A_i,Y}$ is sub-Gaussian with constant sub-Gaussian parameter. By Lemma~\ref{lem: expection of subgradient}, we have 
	\begin{equation}
		\expect{\frac{1}{m}\sum_{i=1}^m\sign\left(\dotp{A_i,X}-s_i\right) \dotp{A_i,Y}} = \psi(X)\dotp{X,Y}.
	\end{equation}
	By Chernoff bound, we can find some constant $c>0$ such that 
	\begin{align}
		&P\left(\abs{\frac{1}{m}\sum_{i =1}^m\sign\left(\dotp{A_i,X}-s_i\right) \dotp{A_i,Y} - \psi(X)\dotp{X,Y}} \ge \delta\right)\\
		\qquad &\le 2e^{-cm\delta^2}
	\end{align}
\end{proof}
\nct{Lemma C.4 is an analogue of \cite[Lemma 4]{ma2021implicit}. Since the corruptions are not assumed to be i.i.d., the function $\psi$  is different from the quantity $\sqrt{\frac{2}{\pi}}\left((1-p) +p\expect{e^{-s_i^2/2\fnorm{X}^2}} \right) \frac{X}{\fnorm{X}}$, which appears in \cite[Lemma 3]{ma2021implicit}. Moreover, we need to apply a (generalized) Chernoff bound for a sum of random variables with different sub-Gaussian parameters in the end of our proof rather than a concentration bound for i.i.d. random variables as done in \cite[Lemma 4]{ma2021implicit}.}
\begin{proof}[Proof of Proposition~\ref{prop: sign-RIP}]
	Without loss of generality, we only need to prove the bound holds for all $X\in \bS_r$ with high probability. By Lemma~\ref{lem: covering lemma for low rank matrices}, we can find $\epsilon$-nets $\bS_{\epsilon,r} \subset \bS_r$, $\bS_{\epsilon, 1} \subset \bS_1$ with respect to Frobenius norm and satisfy $\abs{\bS_{\epsilon,r}} \le \left(\frac{9}{\epsilon}\right)^{(2d+1)r}$, $\abs{\bS_{\epsilon,1}} \le \left(\frac{9}{\epsilon}\right)^{2d+1}$. For any $\bar X \in \bS_{\epsilon,r}$, define $B_r(\bar X, \epsilon) =\{X \in \bb S_r \colon \fnorm{X-\bar X}\le \epsilon\}$. $B_1(\bar X, \epsilon)$ is defined similarly by $B_1(\bar X, \epsilon) =\{X \in \bb S_1 \colon \fnorm{X-\bar X}\le \epsilon\}$. 
	\nct{Then, for any $\bar X,\bar Y$ and $X\in B_r(\bar X,\epsilon)$, $Y \in B_1(\bar Y, \epsilon)$, we have $\dotp{X,Y}- \dotp{\bar X, \bar Y} = \dotp{X, Y - \bar Y} +\dotp{X-\bar X, \bar Y}$. By bounding the two terms on the RHS of the previous equality via the Cauchy-Schwarz's inequality, we have 
	\begin{equation}\label{eqn: XYbarXbarY}
		\abs{\dotp{X,Y}- \dotp{\bar X, \bar Y}} \le  2\epsilon.
	\end{equation}
	}
\nct{ Let us also decompose the quantity of interest, $R:=\frac{1}{m} \sum_{i=1}^m \sign(\dotp{A_i,X}- s_i)\dotp{A_i, Y} - \psi(X)\dotp{X,Y}$, into four terms}:
\nct{ 
\begin{align}
	   R&:\,=\frac{1}{m} \sum_{i=1}^m \sign(\dotp{A_i,X}- s_i)\dotp{A_i, Y} - \psi(X)\dotp{X,Y} \\
	   =&  \underbrace{\frac{1}{m} \sum_{i=1}^m \sign(\dotp{A_i,\bar X}- s_i)\dotp{A_i, \bar Y} - \psi(\bar X)\dotp{\bar X,\bar Y}}_{=:R_1}\\
		+& \underbrace{\frac{1}{m} \sum_{i=1}^m \sign(\dotp{A_i,\bar X}- s_i)\dotp{A_i, Y} -\sign(\dotp{A_i,\bar X}- s_i)\dotp{A_i, \bar Y}  }_{=:R_2}\\
		+& \underbrace{\frac{1}{m} \sum_{i=1}^m \sign(\dotp{A_i,X}- s_i)\dotp{A_i, Y} - \sign(\dotp{A_i,\bar X}- s_i)\dotp{A_i, Y}}_{=:R_3}\\
		+& \underbrace{\psi(\bar X)\dotp{\bar X,\bar Y} - \psi(X)\dotp{ X, Y}}_{=:R_4}
	\end{align}
}
\nct{Recall our goal is to give a high probablity bound on $\sup_{X\in \bS_r, Y\in \bS_1} \abs{R}$. To achieve this goal, we use the above decomposition and the triangle inequality, and have the following bound.
}
\nct{
	\begin{align}
		&\sup_{X\in \bS_r, Y\in \bS_1} \abs{R}
		=\sup_{\substack{\bar X\in \bS_{\epsilon,r}\\ \bar Y \in \bS_{\epsilon,1}}}\sup_{\substack{X \in B_r(\bar X,\epsilon)\\ Y\in B_1(\bar Y, \epsilon)}} \abs{R}\\
	\le & \underbrace{\sup_{\substack{\bar X\in \bS_{\epsilon,r}\\ \bar Y \in \bS_{\epsilon,1}}} \abs{R_1}}_{Z_1}
		+ \underbrace{\sup_{\substack{\bar X\in \bS_{\epsilon,r}\\ \bar Y \in \bS_{\epsilon,1}}}\sup_{Y \in B_r(\bar Y,\epsilon)} \abs{R_2}}_{Z_2}
		+\underbrace{\sup_{\substack{\bar X\in \bS_{\epsilon,r}\\ Y\in \bS_1}}\sup_{X\in B_r(\bar X, \epsilon)} \abs{R_3}}_{Z_3}
		+\underbrace{\sup_{\substack{\bar X\in \bS_{\epsilon,r}\\ \bar Y \in \bS_{\epsilon,1}}}\sup_{\substack{X\in B_r(\bar X, \epsilon)\\
					Y \in B_1(\bar Y,\epsilon)}} \abs{R_4}}_{Z_4}
	\end{align}
}
	By~\ref{eqn: XYbarXbarY} and $\psi(X)=\psi(\bar X) \le 1$, we obtain
	\begin{equation}
		Z_4 \le 2\epsilon.
	\end{equation}
	Then we hope to bound $Z_1, Z_2, Z_3$ separately. By union bound and Lemma~\ref{lem: concentration for fixed XY}, we have $Z_1 \le \delta_1$ with probability at least $1-2\abs{S_{\epsilon,r}}\abs{S_{\epsilon,1}}e^{-cm\delta_1^2}$.
	On the other hand, by $\ell_1/\ell_2$-rip~\eqref{lem:ell1/ell2},
	\begin{align}
		Z_2 &\le \sup_{\bar Y \in \bS_{\epsilon,1} ,Y \in B_1(\bar Y, \epsilon)} \frac{1}{m} \sum_{i=1}^{m}\abs{\dotp{A_i,Y-\bar Y}}\\
		&\le \epsilon \sup_{Z \in \bS_{2}} \frac{1}{m}\sum_{i=1}^{m}\abs{\dotp{A_i,Z}}\\
		&\le \epsilon \left(\sqrt{\frac{2}{\pi}} +\delta_2\right)
	\end{align}
	with probability at least $1-e^{-cm\delta_2^2}$, given $m \gtrsim d$.\\
	Moreover, by Cauchy-Schwartz inequality,
	\begin{equation}\label{eq: csrdpp}
		Z_3 \le \sup_{\substack{\bar X \in S_{\epsilon,r}\\ X\in B_r(\bar X,\epsilon)}}\left(\frac{1}{m}\sum_{i=1}^{m} \left(\sign(\dotp{A_i,\bar X}-s_i) - \sign(\dotp{A_i,X}-s_i)\right)^2\right)^{\frac{1}{2}} \sup_{Y\in \bS_1} \left(\frac{1}{m}\sum_{i=1}^{m} \dotp{A_i,Y}^2\right)^{\frac{1}{2}}.
	\end{equation}
	By $\ell_2$-rip~\eqref{lem: ell2-rip}, we know
	\begin{equation}
		\sup_{Y\in \bS_1} \frac{1}{m}\sum_{i=1}^{m} \dotp{A_i,Y}^2\le 1+\delta_3 
	\end{equation}
	with probability $1- C\exp(-Dm)$ given $m \gtrsim \frac{1}{\delta_3^2}\log(\frac{1}{\delta_3}) d$. Note that $\sign(\dotp{A_i,\bar X}-s_i) = \sign(\dotp{A_i,X}-s_i)$ if $\abs{\dotp{A_i, X-\bar X}}\le \abs{\dotp{A_i,\bar X}-s_i}$, as a result, for any $t > 0$,
	\begin{align}
		&\sup_{\substack{\bar X \in S_{\epsilon,r}\\ X\in B_r(\bar X,\epsilon)}}\frac{1}{m}\sum_{i=1}^{m} \left(\sign(\dotp{A_i,\bar X}-s_i) - \sign(\dotp{A_i,X}-s_i)\right)^2\\
		&\le \sup_{\substack{\bar X \in S_{\epsilon,r}\\ X\in B_r(\bar X,\epsilon)}} \frac{4}{m} \sum_{i=1}^{m} 1\left(\abs{\dotp{A_i,X-\bar X}} \ge \abs{\dotp{A_i,\bar X} -s_i}\right)\\
		&\le  \sup_{\substack{\bar X \in S_{\epsilon,r}\\ X\in B_r(\bar X,\epsilon)}} \frac{4}{m} \sum_{i=1}^{m} 1\left(\abs{\dotp{A_i,X-\bar X}} \ge t\right) + 1\left(\abs{\dotp{A_i,\bar X}- s_i} \le t\right)\\
		&\le \underbrace{\sup_{Z \in \epsilon \bS_{2r}} \frac{4}{m} \sum_{i=1}^{m}1\left(\abs{\dotp{A_i,Z}} \ge t\right)}_{Z_5} + \underbrace{\sup_{\bar X \in \bS_{\epsilon,r}} \frac{4}{m}\sum_{i=1}^{m} 1\left(\abs{\dotp{A_i,\bar X}- s_i} \le t \right)}_{Z_6} \label{eq: Z_5def}
	\end{align}
	For $Z_5$, we use the simple inequality $1(\abs{\dotp{A_i,Z}} \ge t) \le \frac{\abs{\dotp{A_i,Z}}}{t}$ and $\ell_1/\ell_2$-rip~\eqref{lem:ell1/ell2} and obtain
	\begin{align}
		Z_5 &\le \sup_{Z \in \epsilon \bS_{2r}}\frac{4}{m}\sum_{i=1}^{m} \frac{\abs{\dotp{A_i,Z}}}{t}\label{eq: boundZ_5}\\
		&\le \sup_{Z \in \bS_{2r}}\frac{4\epsilon}{m}\sum_{i=1}^{m} \frac{\abs{\dotp{A_i,Z}}}{t}\\
		&\le \frac{4\epsilon(1+\delta_4)}{t}\label{eq: lastboundZ_5}
	\end{align}
	with probability at least $1- e^{-cm\delta_4^2}$ given $m\gtrsim dr$.\\
	For $Z_6$, we firstly use Chernoff's bound for each fixed $\bar{X}$ and get 
	\begin{equation}
		\frac{1}{m}\sum_{i=1}^{m} 1\left(\abs{\dotp{A_i,\bar X} -s_i} \le t\right)\le \expect{ 1\left(\abs{\dotp{A_i,\bar X} -s_i} \le t\right)} + \delta_5
	\end{equation}
	with probability at least $1- e^{cm\delta_5^2}$. On the other hand, for fixed $\bar X \in \bS_{\epsilon,r}$, $\dotp{A_i,\bar X}$ is standard Gaussian. Since the density function of Gaussian is bounded above by $\frac{1}{2\pi}$,  we always have 
	\begin{equation}
		\expect{ 1\left(\abs{\dotp{A_i,\bar X} -s_i} \le t\right)} \le \frac{2}{\sqrt{2\pi}} t \le t.
	\end{equation}
	Consequently, 
	\begin{equation}
		\frac{1}{m}\sum_{i=1}^{m} 1\left(\abs{\dotp{A_i,\bar X} -s_i} \le t\right)\le t + \delta_5
	\end{equation}
	with probability at least $1-e^{-cm\delta_5^2}$.  By union bound, we have
	\begin{equation}
		Z_6 \le 4t + 4\delta_5
	\end{equation}
	with probability at least  $1-\abs{\bS_{\epsilon, r}}e^{-cm\delta_5^2}$. Combining, we have
	\begin{align}
		&\sup_{X,Y\in \bS_r} \abs{\frac{1}{m} \sum_{i=1}^m \sign(\dotp{A_i,X}- s_i)\dotp{A_i, Y} - \psi(X)\dotp{X,Y}}\\
		&\le \delta_1 + \epsilon\left(\sqrt{\frac{2}{\pi}} +\delta_2\right) + \sqrt{1+\delta_3}\sqrt{\frac{4\epsilon(1+\delta_4)}{t} + 4t +4 \delta_5} + 2\epsilon
	\end{align}
	with probability as least $1- 2\abs{S_{\epsilon,r}}\abs{S_{\epsilon,1}}e^{-cm\delta_1^2}-e^{-cm\delta_2^2} - C\exp(-Dm) - e^{-cm\delta_4^2} - \abs{\bS_{\epsilon, r}}e^{-cm\delta_5^2}$, given $m \gtrsim \max\{\frac{1}{\delta_3^2}\log\left(\frac{1}{\delta_3}\right)d, dr\}$. Take $\delta_1 = \delta, \delta_2=\delta_3=\delta_4=\frac{1}{2}, \delta_5 = \delta^2, t=\delta^2, \epsilon=\delta^4$, we have  
	\begin{equation}
		\sup_{X,Y\in \bS_r} \abs{\frac{1}{m} \sum_{i=1}^m \sign(\dotp{A_i,X}- s_i)\dotp{A_i, Y} - \psi(X)\dotp{X,Y}} \lesssim \delta
	\end{equation}
	with probability at least(given $m\gtrsim dr$)
	\begin{align}
		\qquad 1- 2\left(\frac{9}{\delta^4}\right)^{(r+1)(2d+1)}e^{-cm\delta^2} -C'\exp(-D'm) -\left(\frac{9}{\delta^4}\right)^{(2d+1)r}e^{-cm\delta^4}
	\end{align}
	Given $m \gtrsim dr\delta^4\log\left(\frac{1}{\delta}\right)$, we have
	\begin{align}
		&\qquad 2\left(\frac{9}{\delta^4}\right)^{(r+1)(2d+1)}e^{-cm\delta^2} +C'\exp(-D'm) +\left(\frac{9}{\delta^4}\right)^{(2d+1)r}e^{-cm\delta^4}\\
		&\lesssim \exp\left(8r(2d+1) \log\left(\frac{9}{\delta}\right) -cm\delta\right) + \exp\left(4r(2d+1)\log \left(\frac{9}{\delta}\right) - cm\delta^4\right)\\
		&\lesssim \exp(-c'm\delta^4)
	\end{align}
	So if $m \gtrsim dr\delta^4\log\left(\frac{1}{\delta}\right)$,
	\begin{equation}
		\sup_{X \in \bS_r, Y\in \bS_1} \abs{\frac{1}{m} \sum_{i=1}^m \sign(\dotp{A_i,X}- s_i)\dotp{A_i, Y} - \psi(X)\dotp{X,Y}} \lesssim \delta
	\end{equation}
	with probability at least $1-C \exp(-c'm\delta^4)$. This implies
	\begin{equation}
		\sup_{X \in \bS_r} \norm{\frac{1}{m} \sum_{i=1}^m \sign(\dotp{A_i,X}- s_i)A_i - \psi(X) X} \lesssim \delta
	\end{equation}
	by variational expression of operator norm. The proof is complete since we only need to prove RDPP for matrices with unit Frobenius norm.

\end{proof}
\nct{Proposition \ref{prop: sign-RIP} is an analogue of 
\cite[Proposition 5]{ma2021implicit}. Note that the function $\psi$ is different from the function $\psi$ in \cite[Proposition 5]{ma2021implicit} as the corruptions are not assumed to be i.i.d. in this paper. Our proof also deviates from the proof of \cite[Proposition 5]{ma2021implicit} in bounding the term $Z_5$, which appears in \eqref{eq: Z_5def}. This term corresponds to the first term on the RHS of the last line of eq. (38) in \cite{ma2021implicit}. In \cite{ma2021implicit}, this term is bounded by \cite[Lemma 8]{ma2021implicit} using empirical processes tools such as Talagrand's inequality. Here, we bound the term $Z_5$ using a simple contraction argument (stated as an inline inequality before \eqref{eq: boundZ_5}) and the $\ell_1/\ell_2$-RIP; see \eqref{eq: boundZ_5}-\eqref{eq: lastboundZ_5}.
}

\subsection{Proof of Proposition~\ref{prop: RDPP arbitrary}}
\label{sec: Proof of RDPP arbitrary}
We assume for simplicity that $pm$ and $(1-p)m$ are integers. Note that
\begin{align}
	&\qquad\frac{1}{m} \sum_{i=1}^m \sign(\dotp{A_i,X}- s_i)A_i- \psi(X)\frac{X}{\fnorm{X}}\\
	&=\frac{1}{m}\sum_{i\in S} \sign(\dotp{A_i,X}- s_i)A_i + \frac{1}{m}\sum_{i\notin S}\sign(\dotp{A_i,X}- s_i)A_i - \sqrt{\frac{2}{\pi}}\frac{X}{\fnorm{X}}\\
	&= \underbrace{\frac{1}{m}\sum_{i\in S} \sign(\dotp{A_i,X}- s_i)A_i}_{Z_1} + \underbrace{\frac{1}{m}\sum_{i\notin S}\sign(\dotp{A_i,X})A_i - (1-p)\sqrt{\frac{2}{\pi}}\frac{X}{\fnorm{X}}}_{Z_2}\\
	&\qquad-\underbrace{p\sqrt{\frac{2}{\pi}}\frac{X}{\fnorm{X}}}_{Z_3}
\end{align}
We bound $Z_1,Z_2,Z_3$ separately.
\begin{itemize}
	\item For $Z_1$, we observe the following fact: let $e_i\in\{-1,1\}$ be sign variables. For any fixed $\{e_i\}_{i\in S}$, $\sum_{i\in S}e_i A_i$ is a GOE matrix with $N(0, pm)$ diagonal elements and $N(0,\frac{pm}{2})$ off-diagonal elements. By lemma~\ref{lem: concentration of opnorm}, we have
	\begin{equation}
		P\left(\norm{\sum_{i\in S}e_iA_i} \ge \sqrt{pm}(\sqrt{d}+t)\right) \le e^{-\frac{t^2}{2}}.
	\end{equation}
	Take $t=2\sqrt{pm +d}$, we obtain
	\begin{equation}
		P\left(\norm{\sum_{i\in S}e_iA_i} \ge \sqrt{pm}(\sqrt{d}+2\sqrt{pm+d})\right) \le e^{-2(pm+d)}.
	\end{equation}
	As a result, by union bound(the union of all the possible signs), with probability at least $1- 2^{pm}e^{-2(pm+d)} \ge 1- e^{-(pm+d)}$,
	\begin{align}
		\norm{\sum_{i\in S}\sign(\dotp{A_i,X}- s_i)A_i} \le \sqrt{pm}(\sqrt{d}+2\sqrt{pm+d})
	\end{align}
	for any $X$.
	Note also that $\sqrt{d}+2\sqrt{pm+d} \le 3\sqrt{d}+2\sqrt{pm}$, so with probability at least $1-\exp(-(pm+d))$,
	\begin{equation}
		\norm{Z_1} \le 3\sqrt{\frac{dp}{m}} + 2p
	\end{equation}
	for any $X$.
	\item For $Z_2$, applying Proposition~\ref{prop: sign-RIP} with zero corruption and the assumption that $p<\frac{1}{2}$, we obtain that with probability exceeding $1-\exp(-cm(1-p)\delta^2) \ge 1-\exp(-c'm\delta^4)$, the following holds for all matrix $X$ with rank at most $r$,
	\begin{align}
		\norm{\frac{1}{(1-p)m} \sum_{i\notin S}\sign(\dotp{A_i,X})A_i - \sqrt{\frac{2}{\pi}}\frac{X}{\fnorm{X}}} \le \delta,
	\end{align}
	given $m\gtrsim \frac{dr\left(\log(\frac{1}{\delta}) \vee 1\right)}{\delta^4}$.
	Consequently, given $m\gtrsim \frac{dr\left(\log(\frac{1}{\delta}) \vee 1\right)}{\delta^4}$,
	with probability exceeding $1-\exp(-cm(1-p)\delta^2) \ge 1-\exp(-c'm\delta^4)$, 
	\begin{equation}
		\norm{Z_2}\le \delta
	\end{equation}
	for any $X$ with rank at most $r$.
	\item For $Z_3$, we have a deterministic bound
	\begin{equation}
		\norm{Z_3}\le \sqrt{\frac{2}{\pi}} p.
	\end{equation}
	Combining, we obtain that given $m\gtrsim  \frac{dr\left(\log(\frac{1}{\delta}) \vee 1\right)}{\delta^4}$, then with probability exceeding $1-\exp(-(pm+d)) - \exp(-c'm\delta^4)$, 
	\begin{equation}
		\norm{\frac{1}{m} \sum_{i=1}^m \sign(\dotp{A_i,X}- s_i)A_i- \psi(X)\frac{X}{\fnorm{X}}}\le 3\sqrt{\frac{dp}{m}} +3p + \delta
	\end{equation}
	for any $X$ with rank at most $r$.
\end{itemize}

\section{Choice of stepsize}
    \label{sec: Choice of Stepsize}
    First, we present a proposition that is the cornerstone for the choice of stepsize.
\begin{proposition}\label{prop: quantile concentration}
	Fix $p\in (0,1)$, $\epsilon \in (0,1)$. If $m \ge c_0(\epsilon^{-2}\log \epsilon^{-1}) dr\log d$ for some large enough constant $c_0$, then with probability at least $1-c_1\exp(-c_2m\epsilon^2)$, where $c_1$ and $c_2$ are some constants, we have for all  symmetric matrix $G \in \RR^{d\times d}$ with rank at most $r$, 
	\begin{equation}
		\xi_p\left(\{\left|\dotp{A_i,G}\right|\}_{i=1}^{m}\right) \in [\theta_p - 2\epsilon, \theta_p+2\epsilon] \fnorm{G},
	\end{equation}
where $\xi_p(\{\left|\dotp{A_i,G}\right|\}_{i=1}^m)$ is $p$-quantile of samples.~(see Definition 5.1 in \cite{li2020non})
\end{proposition}
Next, we prove a proposition that can be used to estimate $\fnorm{F_tF_t^\top-X^*}$ and $\fnorm{X^*}$ under corruption model~\ref{md: arbitraryCorruption}.
\begin{proposition}\label{prop: estimate fnorm arbitrary}
	Suppose we are under model~\ref{md: arbitraryCorruption} and  $y_i=\dotp{A_i, G} + s_i$'s are given. Fix $\epsilon <0.1$ and corruption probability $p<0.1$. Then if $m \ge c_0(\epsilon^{-2}\log \epsilon^{-1}) dr\log d$ for some large enough constant $c_0$, then with probability at least $1-c_1\exp(-c_2m\epsilon^2)$, where $c_1$ and $c_2$ are some constants, we have for any  symmetric matrix $G \in \RR^{d\times d}$ with rank at most $r$, 
	\begin{align}
		\xi_{\frac{1}{2}}\left(\{|y_i|\}_{i=1}^{m}\right) &\in [\theta_{\frac{1}{2}-p -\epsilon}, \theta_{\frac{1}{2}+p+\epsilon}] \fnorm{G}\\
		&\subset [\theta_{\frac{1}{2}} - L(p+\epsilon), \theta_{\frac{1}{2}}+ L(p+\epsilon)]\fnorm{G},
	\end{align}
where $L>0$ is some universal constant.
\end{proposition}
The following proposition can be used to estimate $\fnorm{F_tF_t^\top -X^*}$ and $\fnorm{X^*}$ under corruption model~\ref{md: randomCorruption}.
\begin{proposition}\label{prop: estimate fnorm random}
	Suppose we are under model~\ref{md: randomCorruption} and  $y_i=\dotp{A_i, G} + s_i$'s are given. Fix corruption probability $p<0.5$. Let $\epsilon = \frac{0.5-p}{3}$. Then if $m \ge c_0 dr\log d$ for some large enough constant $c_0$ depending on $p$, then with probability at least $1-c_1\exp(-c_2m\epsilon^2)$, where $c_1$ and $c_2$ are some constants, we have for all  symmetric matrix $G \in \RR^{d\times d}$ with rank at most $r$, 
	\begin{equation}
		\xi_{\frac{1}{2}}\left(\{|y_i|\}_{i=1}^{m}\right) \in [\theta_{\frac{0.5-p}{3}}, \theta_{1-\frac{0.5-p}{3}}] \fnorm{G}.
	\end{equation}

\end{proposition}

    \subsection{Proof of Proposition~\ref{prop: quantile concentration}}
	The proof is modified from Proposition 5.1 in~\cite{li2020non}. We first note $\dotp{A_i,G} \sim N(0,\fnorm{G}^2)$ and  
	\begin{equation}
		\theta_p(|N(0,\fnorm{G}^2)|) = \theta_p\cdot \fnorm{G}.
	\end{equation} 
	Here $\theta_p(|N(0,\fnorm{G}^2)|)$ denote the $p$-quantile of folded $N(0,\fnorm{G}^2)$.
It suffices to prove the bound for all  symmetric matrices that have rank at most $r$ and unit Frobenius norm. For each fixed symmetric $G_0$ with $\fnorm{G_0}=1$, we know from Lemma~\ref{lem: concentration for sample level meadian} that 
	\begin{equation}\label{eqn: p=quantile concentration}
		\xi_p\left(\{\left|\dotp{A_i,G}\right|\}_{i=1}^{m}\right) \in [\theta_p - \epsilon, \theta_p+\epsilon]
	\end{equation}
with probability at least $1-2\exp(-cm\epsilon^2)$ for some constant $c$ that depends on $p$. Next, we extend this result to all symmetric matrices with rank at most $r$ via a covering argument. Let $S_{\tau, r}$ be a $\tau$-net for all symmetric matrices with rank at most $r$ and unit Frobenius norm. By Lemma~\ref{lem: covering lemma for low rank matrices}, $\left|S_{\tau,r}\right| \le \left(\frac{9}{\tau}\right)^{r(2d+1)}$. Taking union bound, we obtain
\begin{equation}
	\xi_p\left(\{\left|\dotp{A_i,G_0}\right|\}_{i=1}^{m}\right) \in [\theta_p - \epsilon, \theta_p+\epsilon],\qquad \forall G_0\in \bS_{\tau,r}
\end{equation}
with probability at least $1-2\left(\frac{9}{\tau}\right)^{r(2d+1)}\exp(-cm\epsilon^2)$. Set $\tau=\epsilon/(2\sqrt{d(d+m)})$. Under this event and the event that
\begin{equation}
	\max_{i=1,2,\ldots,m}\fnorm{A_i} \le 2\sqrt{d(d+m)},
\end{equation}
which holds with probability at least $1- m\exp(-d(d+m)/2)$ by Lemma~\ref{lem: concetration for max fnorm}, for any rank-$r$ matrix $G$ with $\fnorm{G}=1$, there exists $G_0 \in \bS_{\tau, r}$ such that $\fnorm{G-G_0}\le \tau$, and 
\begin{align}
\abs{\xi_p\left(\{\abs{\dotp{A_i,G}}\}_{i=1}^{m}\right) - \xi_p\left(\{\abs{\dotp{A_i,G_0}}\}_{i=1}^{m}\right) } &\le \max_{i=1,2,\ldots,m}\abs{\abs{\dotp{A_i,G}} - \abs{\dotp{A_i,G_0}}}\\
&\le \max_{i=1,2,\ldots,m}\abs{\dotp{A_i,G-G_0}}\\
&\le \fnorm{G_0-G} \max_{i=1,2,\ldots,m} \fnorm{A_i}\\
&\le \tau 2\sqrt{d(d+m)}\\
&\le \epsilon.
\end{align}
The first inequality follows from Lemma~\ref{lem: order stats  bounded by infty norm}. Combining with~\eqref{eqn: p=quantile concentration}, we obtain that for all symmetric with rank at most $r$ and unit Frobenius norm,
\begin{equation}
	\xi_p\left(\{\abs{\dotp{A_i,G}}\}_{i=1}^{m}\right) \in [\theta_p -2\epsilon, \theta_p+2\epsilon].
\end{equation}
The rest of the proof is to show that the above bound holds with probability at least $1- c_1\exp(-c_2m\epsilon^2)$ for some constants $c_1$ and $c_2$ which follows exactly the same argument as proof of Proposition 5.2 in~\cite{li2020non}.

\subsection{Proof of Proposition~\ref{prop: estimate fnorm arbitrary}}
	Let $\tilde y_i = \dotp{A_i,G}$ be clean samples. By lemma~\ref{lem: corrupted/clean samples}, we have 
	\begin{equation}
		\xi_{\frac{1}{2}}\left(\{|y_i|\}_{i=1}^{m}\right) \in [\xi_{\frac{1}{2}-p}\left(\{|\tilde y_i|\}_{i=1}^{m}\right),\xi_{\frac{1}{2}+p}\left(\{|\tilde y_i|\}_{i=1}^{m}\right)].
	\end{equation}
	Moreover, applying Proposition~\ref{prop: quantile concentration} to $(\xi_{\frac{1}{2}-p}\left(\{|\tilde y_i|\}_{i=1}^{m}\right) , \frac{\epsilon}{2})$ and $(\xi_{\frac{1}{2}+p}\left(\{|\tilde y_i|\}_{i=1}^{m}\right) , \frac{\epsilon}{2})$ , we know that if $m \gtrsim (\epsilon^{-2}\log \epsilon^{-1}) dr\log d$, the we can find constants $c_1$, $c_2$ that with probability at least $1-c_1\exp(-c_2m\epsilon^2)$, 
	\begin{equation}
		\xi_{\frac{1}{2}-p}\left(\{|\tilde y_i|\}_{i=1}^{m}\right) \ge \theta_{\frac{1}{2}-p -\epsilon} \fnorm{G}, \qquad \xi_{\frac{1}{2}+p}\left(\{|\tilde y_i|\}_{i=1}^{m}\right) \le \theta_{\frac{1}{2}+p -\epsilon} \fnorm{G}
	\end{equation}
	holds for any symmetric matrix $G$ with rank at most $r$. Combining, we obtain
	\begin{equation}
		\xi_{\frac{1}{2}}\left(\{|y_i|\}_{i=1}^{m}\right) \in [\theta_{\frac{1}{2}-p -\epsilon}, \theta_{\frac{1}{2}+p+\epsilon}] \fnorm{G}.
	\end{equation}
	In addition,  we easily see that $p\rightarrow \theta_{p}$ is a Lipschitz function with some universal Lipschitz constant $L$ in interval $[0.3,0.7]$. As a result,
	\begin{equation}
		[\theta_{\frac{1}{2}-p -\epsilon}, \theta_{\frac{1}{2}+p+\epsilon}] \fnorm{G}
		\subset [\theta_{\frac{1}{2}} - L(p+\epsilon), \theta_{\frac{1}{2}}+ L(p+\epsilon)]\fnorm{G}.
	\end{equation}
	We are done.

\subsection{Proof of Proposition~\ref{prop: estimate fnorm random}}
		Let $z_i$ be the indicator random variable that 
	\begin{equation}
		z_i = \begin{cases}
			1 & \text{$s_i$ is drawn from some corruption distribution $\mathbb{P}_i$}\\
			0 & \text{$s_i =0$}
		\end{cases}.
	\end{equation}
	Under corruption model~\ref{md: arbitraryCorruption}, $z_i$'s are i.i.d. Bernoulli random variables with parameter $p$. By standard Chernoff inequality, we obtain
	\begin{align}
		P\left(\sum_{i=1}^{m}z_i -pm \ge  \frac{0.5-p}{3} m\right) &=P\left(\sum_{i=1}^{m}z_i -pm \ge  \epsilon m\right)\\
		&\le \exp(-m\epsilon^2/2).
	\end{align}
	Therefore, with probability at least $1-\exp(-m\epsilon^2/2)$, the corruption fraction is less than $p + \frac{0.5-p}{3}$. Let $\tilde y_i = \dotp{A_i, G}$ be clean samples. By Lemma~\ref{lem: corrupted/clean samples}, we have
	\begin{align}
		\xi_{\frac{1}{2}}\left(\{|y_i|\}_{i=1}^{m}\right) \in [\xi_{\frac{1}{2} -p-\frac{0.5-p}{3}}(\{|\tilde y_i|\}_{i=1}^{m}), \xi_{\frac{1}{2} + p +\frac{0.5-p}{3}}(\{|\tilde y_i|\}_{i=1}^{m})] \fnorm{G}.
	\end{align}
	In addition, applying Proposition~\ref{prop: quantile concentration} to $(\xi_{\frac{1}{2} -p-\frac{0.5-p}{3}}(\{|\tilde y_i|\}_{i=1}^{m}) , \frac{\epsilon}{2})$ and $(\xi_{\frac{1}{2} +p+\frac{0.5-p}{3}}(\{|\tilde y_i|\}_{i=1}^{m}) , \frac{\epsilon}{2})$ , we know that($\epsilon = \frac{0.5-p}{3}$) if $m \gtrsim (\epsilon^{-2}\log \epsilon^{-1}) dr\log d\gtrsim dr\log d$, the we can find constants $c_1$, $c_2$ that with probability at least $1-c_1\exp(-c_2m\epsilon^2)$, 
	\begin{align}
		\xi_{\frac{1}{2}-p-\frac{0.5-p}{3}}\left(\{|\tilde y_i|\}_{i=1}^{m}\right) &\ge \theta_{\frac{1}{2}-p -\frac{0.5-p}{3}-\epsilon} \fnorm{G}, \\ \xi_{\frac{1}{2}+p+\frac{0.5-p}{3}}\left(\{|\tilde y_i|\}_{i=1}^{m}\right) &\le \theta_{\frac{1}{2}+p +\frac{0.5-p}{3}+\epsilon}\fnorm{G}
	\end{align}
	holds for any symmetric matrix $G$ with rank at most $r$. Plug in $\epsilon = \frac{0.5-p}{3}$, we obtain 
	\begin{equation}
		\xi_{\frac{1}{2}}\left(\{|y_i|\}_{i=1}^{m}\right) \in [\theta_{\frac{0.5-p}{3}}, \theta_{1-\frac{0.5-p}{3}}] \fnorm{G}
	\end{equation}
	for any symmetric $G$ with rank at most $r$ with the desired probability.

\section{Proof of Initialization}
	\label{sec: initialization}
	Throughout this section, we denote
\begin{align*}
	\bS \;:=\;\{X \in \mathcal{S}^{d\times d} \colon \fnorm{X}=1\},\quad \bS_r\;:=\; \{X \in \mathcal{S}^{d\times d} \colon \fnorm{X}=1, \rank(X)\le r\}.
\end{align*}
Recall that, we construct the matrix
\begin{align*}
    D \;=\;  \frac{1}{m} \sum_{i=1}^m \sign(y_i) A_i.
\end{align*}
 Based on this, we consider its eigen decomposition
\begin{align*}
    D \;=\; U\Sigma U^\top 
\end{align*}
Let $\Sigma_+^k$ be the top $k\times k$ submatrix of $\Sigma$, whose diagonal entries correspond to $k$ largest eigenvalues of $\Sigma$ with negative values replaced by $0$. Accordingly, we let $U_k \in \mathbb R^{d \times k}$ be the submatrix of $U$, formed by its leftmost  $k$ columns. Then we cook up a key ingredient of initialization:
\begin{align*}
    B \;=\; U_k (\Sigma_+^k)^{1/2}.
\end{align*}
In the following, we show that the initialization is close to the ground truth solution.
\begin{proposition}[random corruption]\label{prop: initialization for random corruption}
	Let $F_0$ be the output of Algorithm \ref{alg: initialization}. Fix constant $c_0<0.1$. For Model \ref{md: randomCorruption} with a fixed $p<0.5$ and $m \ge c_1 dr \kappa^4 (\log\kappa+\log r)\log d $. Then 
	 \begin{align}
	 	\|  F_0 F_0^\top - c^*\truX \| \;\leq\; c_0 \sigma_r/\kappa.
	 \end{align} holds for 
	$c^* \in \Big[\sqrt{\frac{1}{2\pi}} \cdot \frac{\theta_{\frac{0.5-p}{3}}}{\theta_{\frac{1}{2}}}, \sqrt{\frac{2}{\pi}} \cdot  \frac{\theta_{1-\frac{0.5-p}{3}}}{\theta_{\frac{1}{2}}}\Big]$ 
	w.p. at least  $1- c_2\exp(-\frac{c_3m}{\kappa^4 r})$, where the constants $c_1,c_2,c_3$ depend only on $p$ and $c_0$.
\end{proposition}
\begin{proof}
	By Lemma~\ref{lem: random corruption} with $\delta =\frac{c_0}{3\theta_{1-\frac{0.5-p}{3}}} \frac{\sigma_r}{\sigma_1\kappa\sqrt{r}}$  and Proposition~\ref{prop: estimate fnorm random}, we know that there exists constants $c_1, c_2,c_3$ depending on $p$ and $c_0$ such that whenever $m \ge c_1 dr \kappa^4 (\log\kappa+\log r)\log d $, then with probability at least  $1- c_2\exp(-\frac{c_3m}{\kappa^4 r})$, we have 
	\begin{align}
		 \| BB^\top -\psi(\truX)  \truX / \| \truX \|_F \| \;\leq\; \frac{c_0}{\theta_{1-\frac{0.5-p}{3}}} \frac{\sigma_r}{\sigma_1\kappa\sqrt{r}}
	\end{align}
and 
\begin{align}
	\theta_{\frac{1}{2}}(\{|y_i|\}_{i=1}^{m}) \in [\theta_{\frac{0.5-p}{3}}, \theta_{1-\frac{0.5-p}{3}}]\fnorm{\truX}.
\end{align}
Combing with the fact that $\psi(\truX) \in [\sqrt{\frac{1}{2\pi}}, \sqrt{\frac{2}{\pi}}]$, we obtain 
\begin{align}
	 &\;\;\; \norm{\frac{\xi_{\frac{1}{2}}(\{|y_i|\}_{i=1}^{m})}{\theta_{\frac{1}{2}}}BB^\top - \frac{\xi_{\frac{1}{2}}(\{|y_i|\}_{i=1}^{m})}{\theta_{\frac{1}{2}}}\psi(\truX)  \truX / \| \truX \|_F}\\
	   &\le \frac{c_0}{\theta_{1-\frac{0.5-p}{3}}(F)}\frac{\sigma_r}{\sigma_1\kappa\sqrt{r}}\frac{\xi_{\frac{1}{2}}(\{|y_i|\}_{i=1}^{m})}{\theta_{\frac{1}{2}}}\\
	   &\le \frac{c_0}{\theta_{1-\frac{0.5-p}{3}}(F)}\frac{\sigma_r}{\sigma_1\kappa \sqrt{r}}  \frac{\theta_{1-\frac{0.5-p}{3}}}{\theta_{\frac{1}{2}}}\fnorm{\truX}\\
	   &\le c_0\frac{\sigma_r}{\sigma_1\kappa\sqrt{r}} \sqrt{r} \sigma_1\\
	   &\le c_0 \sigma_r/\kappa.
    \end{align}
Let $c_*= \frac{\xi_{\frac{1}{2}}(\{|y_i|\}_{i=1}^{m})\psi(\truX)}{\theta_{\frac{1}{2}}\| \truX \|_F}$, clearly we have 
\begin{align}
	c^* \in[(1-p)\theta_{\frac{0.5-p}{3}}(F), \theta_{1-\frac{0.5-p}{3}}(F)] \subset  [\frac{1}{2}\theta_{\frac{0.5-p}{3}}(F), \theta_{1-\frac{0.5-p}{3}}(F)].
\end{align}
The result follows.
\end{proof}

\begin{lemma}[random corrpution]\label{lem: random corruption}
Suppose we are under model~\ref{md: randomCorruption} with fixed $p<0.5$, and we are given  $\delta\le \frac{1}{10\kappa \sqrt{r}}$. Then we have universal constants $c_1,c_2,c_3$ such that whenever $m \ge c_1 \frac{d\left(\log(\frac{1}{\delta}) \vee 1\right)}{\delta^2}$, with probability at least $1-c_2\exp(-c_3m\delta^2)$, we have $\tilde X_0 = B B^\top $ satisfying the following
\begin{align}
    \| \tilde X_0 - \bar X \| \;\leq\; 3\delta,
\end{align}
where $\bar X =  \psi(\truX)  \truX / \| \truX \|_F$, and $\psi(X)=  \frac{1}{m}\sum_{i=1}^{m}\sqrt{\frac{2}{\pi}}\left(1-p + p\EE_{s_i\sim \mathbb{P}_i}\left[\exp(-\frac{s_i^2}{2\fnorm{X}^2})\right]\right)$.
\end{lemma}

\begin{proof}
By Lemma~\ref{lem:pertubation-bound-spectral}, we know that  with probability at least $1-C\exp(-c'm\delta^2)$,
\begin{equation}
	\norm{ D - \bar{X} }{} \leq  \delta.
\end{equation}
Here $c'$ and $C$ are some universal constants. On the other hand, $\psi(\truX) \ge (1-p)\sqrt{\frac{2}{\pi}} \ge \sqrt{\frac{1}{2\pi}}$, so $\lambda_r(\bar X) \ge \sqrt{\frac{1}{2\pi}}\frac{\sigma_r}{\sigma_1\sqrt{r}}= \sqrt{\frac{1}{2\pi}}\frac{1}{\kappa\sqrt{r}}$. By Lemma~\ref{lem: weyl's ineq} and our assumption that $\delta \le \frac{1}{10\kappa \sqrt{r}}$, we know that the top $r$ eigenvalues of $D$ are positive.  
Let $C$ be the best symmetric rank $r$ approximation of $D$ with $\lambda_r(C)>0$ and
\begin{align}
    U_k \;=\; \begin{bmatrix}
    U_r & U_{k-r}
    \end{bmatrix},
\end{align}
then we can write
\begin{align}
    B B^\top \;=\; C + U_{k-r} \Sigma_{k-r} U_{k-r}^\top,
\end{align}
where $\Sigma_{k-r} = \diag ( (\lambda_{r+1}(D))_+, \cdots, (\lambda_{k}(D))_+ ) $. Then we have
\begin{align}
    \| BB^\top - \bar{ X}  \| \;\leq\; \| C - \bar{X} \| \;+\; \| \Sigma_{k-r} \| .
\end{align}
Finally, given that $C$ is the best symmetric rank-$r$ approximation of $D$, we have
\begin{align}
	\norm{C-D} \;\leq\; \sigma_{r+1}(D) = \abs{ \sigma_{r+1}(D) - \sigma_{r+1}(\bar{X}) }\;\leq\; \norm{ D - \bar{X} } \;\leq\; \delta,
\end{align}
where for the equality, we used the fact that $\sigma_{r+1}(\bar{X}) = 0$. 
Combining, we obtain
\begin{align}
    \norm{C - \bar{X}} \;\leq\; \norm{ C - D } + \norm{ D - \bar{X} } \;\leq \;  2\delta,
\end{align}
and
\begin{align}
    \| \Sigma_{k-r} \| \;\leq\; \norm{D - \bar{X}}{} \;\leq \; \delta.
\end{align}
Therefore, we have
\begin{align}
     \norm{ BB^\top - \bar{ X}  }{} \;\leq\; 3 \delta
\end{align}
with probability at least $1-C\exp(-c'm\delta^2)$, given $m \gtrsim \frac{d\left(\log(\frac{1}{\delta}) \vee 1\right)}{\delta^2}$.
\end{proof}

\begin{lemma}[perturbation bound under random corruption]\label{lem:pertubation-bound-spectral}
For any $\delta>0$, whenever $m \gtrsim \frac{d\left(\log(\frac{1}{\delta}) \vee 1\right)}{\delta^2}$, we have
\begin{align}
      \norm{ D - \bar{X} }{} \leq  \delta
\end{align}
holds with probability at least $1-C\exp(-c'm\delta^2)$. Here, $\bar{X} =  \psi(\truX)  \truX / \| \truX \|_F$, and $c'$, and $C>0$ are some positive numerical constants.
\end{lemma}

\begin{proof}
Without loss of generality, we assume $\fnorm{\truX} =1$. First, we prove $\norm{ D - \bar{X} }{} \leq  \delta$ by invoking Lemma \ref{lem: concentration for fixed XY}, then follow by a union bound. For each $Y \in \bb S_1$, let $B_1(\bar Y, \epsilon) =\{Z \in \bb S_1 \colon \fnorm{Z-\bar Y}\le \epsilon\}$. By Lemma~\ref{lem: covering lemma for low rank matrices}, we can always find an $\epsilon$-net  $\bS_{\epsilon, 1} \subset \bS_1$ with respect to Frobenius norm and satisfy $\abs{\bS_{\epsilon,1}} \le \left(\frac{9}{\epsilon}\right)^{2d+1}$.	Based on the $\epsilon$-net and triangle inequality, one has
\begin{align}
	\norm{D-\bar X}&= \sup_{Y\in \bS_1} \abs{\frac{1}{m} \sum_{i=1}^m \sign(\dotp{A_i,\truX}+ s_i)\dotp{A_i, Y} - \psi(\truX)\dotp{\truX,Y}}\\
	&=\sup_{ \bar Y \in \bS_{\epsilon,1}}\sup_{ Y\in B_1(\bar Y, \epsilon)} \abs{\frac{1}{m} \sum_{i=1}^m \sign(\dotp{A_i,\truX}+ s_i)\dotp{A_i, Y} - \psi(\truX)\dotp{\truX,Y}}\\
	&\le \underbrace{\sup_{ \bar Y \in \bS_{\epsilon,1}} \abs{\frac{1}{m} \sum_{i=1}^m \sign(\dotp{A_i,\truX}+ s_i)\dotp{A_i, \bar Y} - \psi( \truX)\dotp{\truX,\bar Y}}}_{Z_1}\\
	&+ \underbrace{\sup_{\bar Y \in \bS_{\epsilon,1}}\sup_{Y \in B_r(\bar Y,\epsilon)} \abs{\frac{1}{m} \sum_{i=1}^m \sign(\dotp{A_i,\truX}+ s_i)\dotp{A_i, Y} -\sign(\dotp{A_i, \truX}+ s_i)\dotp{A_i, \bar Y}  }}_{Z_2}\\
	&+\underbrace{\sup_{ \bar Y \in \bS_{\epsilon,1}}\sup_{
				Y \in B_1(\bar Y,\epsilon)} \abs{\psi(\truX)\dotp{\truX,\bar Y} - \psi(\truX)\dotp{ \truX, Y}}}_{Z_3}
\end{align}
Since $\psi(X)=\psi(\bar X) \le 1$, we obtain
\begin{equation}
	Z_3 \le \fnorm{\truX} \fnorm{\bar Y- Y} \le \epsilon.
\end{equation}
Then we hope to bound $Z_1, Z_2$ separately. By union bound and Lemma~\ref{lem: concentration for fixed XY}, we have $Z_1 \le \tilde \delta$ with probability at least $1-2\abs{S_{\epsilon,1}}e^{-cm \tilde \delta^2}$.
On the other hand, by $\ell_1/\ell_2$-rip~\eqref{lem:ell1/ell2},
\begin{align}
	Z_2 &\le \sup_{\bar Y \in \bS_{\epsilon,1} ,Y \in B_1(\bar Y, \epsilon)} \frac{1}{m} \sum_{i=1}^{m}\abs{\dotp{A_i,Y-\bar Y}}\\
	&\le \epsilon \sup_{Z \in \bS_{2}} \frac{1}{m}\sum_{i=1}^{m}\abs{\dotp{A_i,Z}}\\
	&\le \epsilon \left(\sqrt{\frac{2}{\pi}} +1\right)
\end{align}
with probability at least $1-e^{-cm}$, given $m \gtrsim d$.\\
 Combining, we have
\begin{align}
	\norm{D-\bar X}
	\le \delta_1 + \epsilon\left(\sqrt{\frac{2}{\pi}} +1\right) +\epsilon
\end{align}
with probability as least $1- 2\abs{S_{\epsilon,1}}e^{-cm\tilde \delta^2}-e^{-cm} $, given $m \gtrsim d$. Take $\tilde \delta = \delta/3, \epsilon=\delta/10$, we have  
\begin{equation}
	\norm{D-\bar X} \le\delta
\end{equation}
with probability at least(given $m\gtrsim d$)
\begin{align}
	\qquad 1- 2\left(\frac{90}{\delta}\right)^{(2d+1)}e^{-cm\delta^2} -e^{-cm}
\end{align}
Given $m \gtrsim \frac{d\left(\log(\frac{1}{\delta}) \vee 1\right)}{\delta^2}$, we have
\begin{align}
	&\qquad 2\left(\frac{90}{\delta}\right)^{(2d+1)}e^{-cm\delta^2} +e^{-cm}\\
	&\lesssim \exp\left((2d+1) \log\left(\frac{90}{\delta}\right) -cm\delta^2\right) + \exp(-cm)\\
	&\lesssim \exp(-c'm\delta^2)
\end{align}
So if $m \gtrsim \frac{d\log\left(\frac{1}{\delta}\right)}{\delta^2}$,
\begin{equation}
	\norm{D- \bar X} \le \delta
\end{equation}
with probability at least $1-C \exp(-c'm\delta^2)$. 
\end{proof}
\begin{proposition}[arbitrary corruption]\label{prop: initialization for arbitrary corruption}
Let $F_0$ be the output of Algorithm \ref{alg: initialization}. Fix constant $c_0<0.1$. For model~\ref{md: arbitraryCorruption} with $p\le \frac{\tilde c_0}{\kappa^2 \sqrt{r}}$ where $\tilde c_0$ depends only on $c_0$, there exist constants $c_1,c_2,c_3$ depending only on $c_0$ such that whenever $m \ge c_1 dr \kappa^2 \log d (\log \kappa +\log r) $, we have
\begin{align}
	\|  F_0 F_0^\top - c^*\truX \| \;\leq\; c_0 \sigma_r/\kappa.
\end{align}
with probability at least   $1-c_2\exp(-c_3\frac{m}{\kappa^4 r})-\exp(-(pm+d))$. Here $c^* = 1$.
\end{proposition}
\begin{proof}
	Taking  $\epsilon = \frac{c_0 \theta_{\frac{1}{2}}}{4L\kappa^2}$  in Proposition~\ref{prop: estimate fnorm arbitrary}, where $L$ is a universal constant doesn't depend on anything from Proposition~\ref{prop: estimate fnorm arbitrary}, we know that with probability at least $1-c_5\exp(-c_6\frac{m}{\kappa^4})$
	\begin{align}
		\xi_{\frac{1}{2}}\left(\{|y_i|\}_{i=1}^{m}\right) \in [\theta_{\frac{1}{2}} - L(p+\epsilon), \theta_{\frac{1}{2}}+ L(p+\epsilon)]\fnorm{\truX},
	\end{align}
given $m \ge c_7 dr \kappa^4 \log d \log \kappa $. Here $c_5,c_6,c_7$ are constants depending only on $c_0$.
Given $\tilde c_0 = \frac{c_0}{4L}$, the above inclusion implies that 
\begin{align}
	\left|1-\frac{	\xi_{\frac{1}{2}}\left(\{|y_i|\}_{i=1}^{m}\right) }{\fnorm{\truX}\theta_{\frac{1}{2}}}\right| \le \frac{L(p+\epsilon)}{\theta_{\frac{1}{2}}} \le \frac{c_0}{2\kappa^2}.
\end{align}

Take $\delta = \frac{c_0\sqrt{\frac{2}{\pi}}}{12(1+\frac{L}{\theta_{\frac{1}{2}}})}\frac{\sigma_r}{\sigma_1 \kappa\sqrt{r}}$ in lemma~\ref{lem: arbitrary corruption}, we know that with probability at least  $1-c_8\exp(-c_9\frac{m}{\kappa^4 r})-\exp(-(pm+d))$ for constants $c_8,c_9$ depending only on $c_0$,
\begin{equation}
	\norm{BB^\top -\sqrt{\frac{2}{\pi}}  \truX / \| \truX \|_F} \le \frac{c_0\sqrt{\frac{2}{\pi}}}{2(1+\frac{L}{\theta_{\frac{1}{2}}})}\frac{\sigma_r}{\sigma_1\kappa \sqrt{r}},
\end{equation}
given $m \gtrsim dr \kappa^4  (\log \kappa +\log r) $.
The above inequality implies that 
\begin{align}
	\norm{\frac{ \theta_{\frac{1}{2}}(\{|y_i|\}_{i=1}^{m})}{\sqrt{\frac{2}{\pi}}\theta_{\frac{1}{2}}}B B^\top - \frac{ \theta_{\frac{1}{2}}(\{|y_i|\}_{i=1}^{m})\truX}{\fnorm{\truX}\theta_{\frac{1}{2}}}} &\le \frac{1+\frac{L}{\theta_{\frac{1}{2}}}}{\sqrt{\frac{2}{\pi}}}  \frac{c_0\sqrt{\frac{2}{\pi}}}{2(1+\frac{L}{\theta_{\frac{1}{2}}})}\frac{\sigma_r}{\sigma_1 \sqrt{r}}\fnorm{\truX} \\
	&\le \frac{c_0 \sigma_r}{2}.
\end{align}
\end{proof}
Combining, we can find some constants $c_1,c_2,c_3$ depending only on $c_0$ such that whenever $m \ge c_1 dr \kappa^4 \log d (\log \kappa +\log r) $, then with probability at least $1-c_2\exp(-c_3\frac{m}{\kappa^4 r})-\exp(-(pm+d))$ ,
\begin{align}
	&\;\;\;\norm{\frac{ \theta_{\frac{1}{2}}(\{|y_i|\}_{i=1}^{m})}{\sqrt{\frac{2}{\pi}}\theta_{\frac{1}{2}}}B B^\top - \truX}\\
	 &\le \norm{\frac{ \theta_{\frac{1}{2}}(\{|y_i|\}_{i=1}^{m})}{\sqrt{\frac{2}{\pi}}\theta_{\frac{1}{2}}}B B^\top - \frac{ \theta_{\frac{1}{2}}(\{|y_i|\}_{i=1}^{m})\truX}{\fnorm{\truX}\theta_{\frac{1}{2}}}} + \norm{\left(1-\frac{	\theta_{\frac{1}{2}}\left(\{|y_i|\}_{i=1}^{m}\right) }{\fnorm{\truX}\theta_{\frac{1}{2}}}\right)\truX}\\
	&\le \frac{c_0 \sigma_r}{2\kappa}+ \frac{L(p+\epsilon)}{\theta_{\frac{1}{2}}}\sigma_1\\
	&\le \frac{c_0\sigma_r}{\kappa}.
\end{align}

\begin{lemma}[arbitrary corrpution]\label{lem: arbitrary corruption}
	Suppose we are given  $\delta\le \frac{1}{10\kappa \sqrt{r}}$. Suppose we are under model~\ref{md: arbitraryCorruption} with fixed $p<\delta/10$.  Then we have universal constants $c_1,c_2,c_3$ such that whenever $m \ge c_1 \frac{d\left(\log(\frac{1}{\delta}) \vee 1\right)}{\delta^2}$, with probability at least $1-c_2\exp(-c_3m\delta^2)-\exp(-(pm+d))$, we have $\tilde X_0 = B B^\top $ satisfying the following
	\begin{align}
		\| \tilde X_0 - \bar X \| \;\leq\; 6\delta,
	\end{align}
	where $\bar X =  \psi(\truX)  \truX / \| \truX \|_F$, and $\psi(X)=  \frac{1}{m}\sum_{i=1}^{m}\sqrt{\frac{2}{\pi}}\left(1-p + p\EE_{s_i\sim \mathbb{P}_i}\left[\exp(-\frac{s_i^2}{2\fnorm{X}^2})\right]\right)$.
\end{lemma}

\begin{proof}
	By Lemma~\ref{lem: perturbation bound under arbitrary corruption}, given $m \ge c_1 \frac{d\left(\log(\frac{1}{\delta}) \vee 1\right)}{\delta^2}$, we know that  with probability at least $1-\exp(-(pm+d)) - \exp(-c_2m\delta^2)$,
	\begin{equation}
		\norm{ D - \bar{X} }{} \leq  2\delta.
	\end{equation}
 On the other hand, $\psi(\truX) \ge (1-p)\sqrt{\frac{2}{\pi}} \ge \sqrt{\frac{1}{2\pi}}$, so $\lambda_r(\bar X) \ge \sqrt{\frac{1}{2\pi}}\frac{\sigma_r}{\sigma_1\sqrt{r}}= \sqrt{\frac{1}{2\pi}}\frac{1}{\kappa\sqrt{r}}$. By Lemma~\ref{lem: weyl's ineq} and our assumption that $\delta \le \frac{1}{10\kappa \sqrt{r}}$, we know that the top $r$ eigenvalues of $D$ are positive.  
	Let $C$ be the best symmetric rank $r$ approximation of $D$ with $\lambda_r(C)>0$ and
	\begin{align}
		U_k \;=\; \begin{bmatrix}
			U_r & U_{k-r}
		\end{bmatrix},
	\end{align}
	then we can write
	\begin{align}
		B B^\top \;=\; C + U_{k-r} \Sigma_{k-r} U_{k-r}^\top,
	\end{align}
	where $\Sigma_{k-r} = \diag ( (\lambda_{r+1}(D))_+, \cdots, (\lambda_{k}(D))_+ ) $. Then we have
	\begin{align}
		\| BB^\top - \bar{ X}  \| \;\leq\; \| C - \bar{X} \| \;+\; \| \Sigma_{k-r} \| .
	\end{align}
	Finally, given that $C$ is the best symmetric rank-$r$ approximation of $D$, we have
	\begin{align}
		\norm{C-D} \;\leq\; \sigma_{r+1}(D) = \abs{ \sigma_{r+1}(D) - \sigma_{r+1}(\bar{X}) }\;\leq\; \norm{ D - \bar{X} } \;\leq\; 2\delta,
	\end{align}
	where for the equality, we used the fact that $\sigma_{r+1}(\bar{X}) = 0$. 
	Combining, we obtain
	\begin{align}
		\norm{C - \bar{X}} \;\leq\; \norm{ C - D } + \norm{ D - \bar{X} } \;\leq \;  4\delta,
	\end{align}
	and
	\begin{align}
		\| \Sigma_{k-r} \| \;\leq\; \norm{D - \bar{X}}{} \;\leq \; 2\delta.
	\end{align}
	Therefore, we have
	\begin{align}
		\norm{ BB^\top - \bar{ X}  }{} \;\leq\; 6 \delta
	\end{align}
	with probability at least $1-\exp(-(pm+d)) - \exp(-c_2m\delta^2)$, given $m \ge  c_1\frac{d\left(\log(\frac{1}{\delta}) \vee 1\right)}{\delta^2}$.
\end{proof}
\begin{lemma}[perturbation bound under arbitrary corruption]\label{lem: perturbation bound under arbitrary corruption}
	Given a fixed constant $\delta>0$. Suppose the measurements ${A_i}$'s are i.i.d. GOE, $s_i$'s are from model~\ref{md: arbitraryCorruption} with fixed $p \le \delta/10$. There exist universal constants $c_1$ and $c_2$ such that whenever $m \ge c_1   \frac{d\left(\log(\frac{1}{\delta}) \vee 1\right)}{\delta^2}$, with probability with probability at least $1-\exp(-(pm+d)) - \exp(-c_2m\delta^2)$, we have  $   D = \frac{1}{m} \sum_{i=1}^m \sign(y_i) A_i.$ satisfying the following
	\begin{align}
		\| D - \bar X \| \;\leq\;2 \delta,
	\end{align}
	where $\bar X =  \sqrt{\frac{2}{\pi}}  \truX / \| \truX \|_F$.
\end{lemma}
\begin{proof}
	Let $S$ be the set of indices that the corresponding observations are corrupted. We assume for simplicity that $pm$ and $(1-p)m$ are integers. Note that
	\begin{align}
		D-\bar X&=\frac{1}{m} \sum_{i=1}^m \sign(\dotp{A_i,\truX}+ s_i)A_i- \sqrt{\frac{2}{\pi}}\frac{\truX}{\fnorm{\truX}}\\
		&=\frac{1}{m}\sum_{i\in S} \sign(\dotp{A_i,\truX}+s_i)A_i + \frac{1}{m}\sum_{i\notin S}\sign(\dotp{A_i,\truX})A_i - \sqrt{\frac{2}{\pi}}\frac{\truX}{\fnorm{\truX}}\\
		&= \underbrace{\frac{1}{m}\sum_{i\in S} \sign(\dotp{A_i,\truX}+ s_i)A_i}_{Z_1} + \underbrace{\frac{1}{m}\sum_{i\notin S}\sign(\dotp{A_i,\truX})A_i - (1-p)\sqrt{\frac{2}{\pi}}\frac{\truX}{\fnorm{\truX}}}_{Z_2}\\
		&\qquad-\underbrace{p\sqrt{\frac{2}{\pi}}\frac{\truX}{\fnorm{\truX}}}_{Z_3}
	\end{align}
	We bound $Z_1,Z_2,Z_3$ separately.
	\begin{itemize}
		\item For $Z_1$, we observe the following fact: let $e_i\in\{-1,1\}$ be sign variables. For any fixed $\{e_i\}_{i\in S}$, $\sum_{i\in S}e_i A_i$ is a GOE matrix with $N(0, pm)$ diagonal elements and $N(0,\frac{pm}{2})$ off-diagonal elements. By lemma~\ref{lem: concentration of opnorm}, we have
		\begin{equation}
			P\left(\norm{\sum_{i\in S}e_iA_i} \ge \sqrt{pm}(\sqrt{d}+t)\right) \le e^{-\frac{t^2}{2}}.
		\end{equation}
		Take $t=2\sqrt{pm +d}$, we obtain
		\begin{equation}
			P\left(\norm{\sum_{i\in S}e_iA_i} \ge \sqrt{pm}(\sqrt{d}+2\sqrt{pm+d})\right) \le e^{-2(pm+d)}.
		\end{equation}
		As a result, by union bound(the union of all the possible signs), with probability at least $1- 2^{pm}e^{-2(pm+d)} \ge 1- e^{-(pm+d)}$,
		\begin{align}
			\norm{\sum_{i\in S}\sign(\dotp{A_i,\truX}- s_i)A_i} \le \sqrt{pm}(\sqrt{d}+2\sqrt{pm+d}).
		\end{align}
		Note also that $\sqrt{d}+2\sqrt{pm+d} \le 3\sqrt{d}+2\sqrt{pm}$, so with probability at least $1-\exp(-(pm+d))$,
		\begin{equation}
			\norm{Z_1} \le 3\sqrt{\frac{dp}{m}} + 2p
		\end{equation}
		for any $X$.
		\item For $Z_2$, by the proof of  Lemma~\ref{lem:pertubation-bound-spectral} with zero corruption and the assumption that $p<\frac{1}{2}$, we obtain that with probability exceeding $1-\exp(-cm(1-p)\delta^2) \ge 1-\exp(-c'm\delta^2)$, the following holds,
		\begin{align}
			\norm{\frac{1}{(1-p)m} \sum_{i\notin S}\sign(\dotp{A_i,\truX})A_i - \sqrt{\frac{2}{\pi}}\frac{\truX}{\fnorm{\truX}}} \le \delta,
		\end{align}
		given $m\gtrsim \frac{d\left(\log(\frac{1}{\delta}) \vee 1\right)}{\delta^2}$.
		Consequently, given $m\gtrsim \frac{d\left(\log(\frac{1}{\delta}) \vee 1\right)}{\delta^2}$,
		with probability exceeding $1-\exp(-cm(1-p)\delta^2) \ge 1-\exp(-c'm\delta^2)$, 
		\begin{equation}
			\norm{Z_2}\le \delta
		\end{equation}
		for any $X$ with rank at most $r$.
		\item For $Z_3$, we have a deterministic bound
		\begin{equation}
			\norm{Z_3}\le \sqrt{\frac{2}{\pi}} p.
		\end{equation}
		Combining, we obtain that given $m\gtrsim  \frac{d\left(\log(\frac{1}{\delta}) \vee 1\right)}{\delta^2}$, then with probability exceeding $1-\exp(-(pm+d)) - \exp(-c'm\delta^2)$, 
		\begin{equation}
			\norm{\frac{1}{m} \sum_{i=1}^m \sign(\dotp{A_i,\truX}- s_i)A_i- \psi(\truX)\frac{\truX}{\fnorm{\truX}}}\le 3\sqrt{\frac{dp}{m}} +3p + \delta.
		\end{equation}
	Take $\delta = \frac{c_0}{3\kappa\sqrt{r}}$ and let $m \gtrsim   \frac{d\left(\log(\frac{1}{\delta}) \vee 1\right)}{\delta^2}$, we know that if $p \le \delta/10$, we have 
	\begin{equation}
		\norm{\frac{1}{m} \sum_{i=1}^m \sign(\dotp{A_i,\truX}- s_i)A_i- \psi(\truX)\frac{\truX}{\fnorm{\truX}}}\le2 \delta
	\end{equation}
with probability at least $1-\exp(-(pm+d)) - \exp(-c'm\delta^2)$.
	\end{itemize}
\end{proof}

\section{Proof of \Cref{thm: identifiability}}
    \label{sec: proofofidentifiability}
	Here we prove the identifiability result in Section \ref{sec:models}.
\begin{proof}
Using Lemma \ref{lem:ell1/ell2}, we know that the  $\ell_1/\ell_2$-RIP conditions holds for $\mathcal{A}$: for some universal $c>0$, with probability at least
$1-\exp(-cm\delta^2)$, there holds. 
\begin{align*}
    \left| \frac{1}{m}\norm{\mathcal{A}(X)}_1 - \sqrt{\frac{2}{\pi}} \norm{X}_F \right|
    \le \delta \norm{X}_F,\quad \forall X \in \RR^{d\times d}: \textup{rank}(X) \le k+r.
\end{align*}
Now for any subset $L\subset \{1,\dots,m\}$, we can define $\mathcal{A}_L$ as $[\mathcal{A}(X)]_i = \dotp{A_i,X}$ if $i\in L$ and $0$ otherwise. Then if the size of $L$ satisfies that  $|L|\geq C d(r+k)\log d$ for some universal constant, using Lemma \ref{lem:ell1/ell2} again, we have
with probability at least
$1-\exp(-c|L|\delta^2)$, there  holds
\begin{align*}
    \left| \frac{1}{|L|}\norm{\mathcal{A}_{L}(X)}_1 - \sqrt{\frac{2}{\pi}} \norm{X}_F \right|
    \le \delta \norm{X}_F,\quad \forall X \in \RR^{d\times d}: \textup{rank}(X) \le k+r,
\end{align*}
Note that the above holds for each fixed $L$. If we choose $S$ to be the set of indices of nonzero $s_i$. Using Bernstein's inequality, we know with probability at least $1-\exp(-c\epsilon^2 m(1-p))$, $|L|\geq (1-\epsilon)(1-p)m$. Due to our model assumptions, $S$ is independent of $\mathcal{A}$. Hence, the above displayed inequality does hold for $L=S^c$ with probability at least $1- \exp(-c_1(\epsilon^2 +\delta^2)m(1-p))$.

Let us assume the above two displayed inequalities, the second one with $L=S^c$ in the following derivation. 
Let $F$ is optimal for \eqref{opt: maintextmain}. Starting from the optimality of $F$ and $\truX$ has rank $r\leq k$, we have 
\begin{align*}
    0 &\ge \frac{1}{m} \norm{\mathcal{A}(FF^\top)- y}_1 -  \frac{1}{m} \norm{\mathcal{A}(\truX)- y}_1\\
    &= \frac{1}{m}\norm{\mathcal{A}(FF^\top - \truX) -s }_1 -  \frac{1}{m} \norm{s}_1 \\
    &= \frac{1}{m}\norm{ \left[\mathcal{A}_{S^c}(FF^\top - \truX) \right] }_1 + \frac{1}{m}\norm{ \left[\mathcal{A}_{S}(FF^\top - \truX)\right] -s }_1-  \frac{1}{m} \norm{s}_1 \\
    &\ge \frac{1}{m}\norm{ \left[\mathcal{A}_{S^c}(FF^\top - \truX) \right]}_1 - \frac{1}{m}\norm{ \left[\mathcal{A}_{S}(FF^\top - \truX)\right] }_1\\
    &=  \frac{2}{m}\norm{ \left[\mathcal{A}_{S^c}(FF^\top - \truX) \right] }_1 - \frac{1}{m}\norm{ \mathcal{A}(FF^\top - \truX) }_1 \\
    &\ge \left( 2(1-p)(1-\epsilon) \left(\sqrt{\frac{2}{\pi}} - \delta\right) - \left(\sqrt{\frac{2}{\pi}} + \delta \right)  \right) \norm{FF^\top - \truX}_F.
\end{align*}
Hence so long as $2(1-p)(1-\epsilon) \left(\sqrt{\frac{2}{\pi}} - \delta\right) - \left(\sqrt{\frac{2}{\pi}} + \delta \right)>0$, we know  $FF^\top = \truX$. The condition $2(1-p)(1-\epsilon) \left(\sqrt{\frac{2}{\pi}} - \delta\right) - \left(\sqrt{\frac{2}{\pi}} + \delta \right)>0$ is satisfied with probability at least $1-\exp(-c'm)$ and $m\geq C'(r+k)d\log d$ for some $c'$ and $C'$ depending on $p$.

\end{proof}

\section{Results under better initialization}
    \label{sec: better convergence result}
    As indicated in remarks under Theorem~\ref{thm: exact recovery under models}, we can show that the sample complexity for provable convergence is indeed $O(dk^3\kappa^4(\log \kappa + \log k) \log d$, given $p \lesssim \frac{1}{\kappa \sqrt{r}}$ in either model. The proof consists of two theorems stated below.
\begin{thm}\label{lem: mainlemma faster}
Suppose the following conditions hold: 
\begin{itemize}[leftmargin=0.3in,topsep=0em,noitemsep]
      \item[(i)] Suppose $F_0$ satisfies
    \begin{align}\label{eqn: initialconditionmainlemma faster}
		\|  F_0 F_0^\top - \truX \| \;\leq\; c_0 \sigma_r
	\end{align}
for small sufficiently small universal constant $c_0$.
    \item[(ii)] The stepsize satisfies $0<\frac{c_1}{\sigma_1}\leq \frac{\eta_t}{\fnorm{F_tF_t^\top -\truX}} \leq  \frac{c_2}{\sigma_1}$ for some small numerical constants $c_1 < c_2 \le 0.01$ and all $t\ge 0$. 
    \qq{consistent} 
    \item[(iii)] $(r+k,\delta)$-RDPP holds for $\{A_i,s_i\}_{i=1}^m$ with $\delta \leq \frac{c_3}{\kappa\sqrt{k+r}}$ and a scaling function $\psi\in \big[\sqrt{\frac{1}{2\pi}}, \sqrt{\frac{2}{\pi}}\big].$ Here $c_3$ is some sufficiently small universal constant.
\end{itemize}
Then, we have a sublinear convergence in the sense that for any $t \ge 0$,
\[
\norm{F_{t}F_{t}^\top -\truX} \leq c_5\sigma_1\frac{1}{\kappa+t}.
\]
Moreover, if $k=r$, then under the same set of condition, we have convergence at a linear rate
\[
\norm{F_{t}F_{t}^\top -\truX}  \leq c_6\sigma_r\left(1-\frac{c_7}{\kappa}\right)^t,\qquad \forall t\ge 0.
\]
Here $c_5,c_6$ and $c_7$ are universal constants.
\end{thm} 
\begin{proof}
Take $c_0 =0.01$ and $c_3= 0.001$. Next, by definition, 
$\gamma_t = \frac{\eta_t \psi(F_tF_t^\top -\truX)}{\fnorm{F_tF_t^\top -\truX}}$. By the second assumption and the assumption on range of $\psi$, we know 
\begin{equation}
    \gamma_t \in[\sqrt{\frac{1}{2\pi}} \frac{c_1}{\sigma_1}, \sqrt{\frac{2}{\pi}} \frac{c_2}{\sigma_1}].
\end{equation}
Since we assumed $c_2\le 0.01$, so the stepsize condition $\gamma_t \le \frac{0.01}{\sigma_1}$ is satisfied. Hence, both Proposition~\ref{prop: stagethree} and Proposition~\ref{prop: stagethree k=r} hold for $t=0$. We consider two cases separately. 
\begin{itemize}
    \item $k > r$, By Proposition~\ref{prop: stagethree} and induction, we know
    \begin{equation}
        E_{t+1} \le E_t(1-\gamma_tE_t) \le E_t(1-\frac{c_\gamma}{\sigma_1} E_t),\qquad \forall t \ge 0.
    \end{equation}
    where $ c_1\sqrt{\frac{2}{\pi}}\le c_\gamma \le 0.01$. Define $G_t = \frac{c_\gamma}{\sigma_1}E_t$, then we have $G_0 <1$ and
\begin{align}
	G_{t+1} \le G_t(1-G_t),\qquad \forall t \ge 0.
\end{align}
Taking reciprocal, we obtain
\begin{equation}
	\frac{1}{G_{t+1}} \ge \frac{1}{G_t} + \frac{1}{1-G_t} \ge \frac{1}{G_t} +1,\qquad \forall t \ge 0.
\end{equation}
So we obtain 
\begin{equation}
	G_{t} \le \frac{1}{\frac{1}{G_{0}}+t}, \qquad \forall t \ge 0.
\end{equation}
Plugging in the definition of $G_t$, we obtain
\begin{equation}
    E_{\tau_2 +t} \le  \frac{\sigma_1}{c_\gamma} \frac{1}{\frac{\sigma_1}{c_\gamma  E_{0}} +t} \le \frac{\sigma_1}{c_\gamma } \frac{1}{\frac{100\sigma_1}{c_\gamma \sigma_r} +t} = \frac{\sigma_1}{c_\gamma} \frac{1}{\frac{100}{c_\gamma}\kappa +t}\le \frac{\sigma_1}{c_\gamma} \frac{1}{\kappa +t}.
\end{equation}
Since $c_\gamma \ge c_1\sqrt{\frac{2}{\pi}}$, we can simply take $c_5 = \frac{1}{4c_1}\sqrt{\frac{\pi}{2}}$, apply Lemma~\ref{lem: decomposition of FF-X}, and get 
\begin{equation}
    \norm{F_tF_t^\top -\truX} \le c_5\sigma_1\frac{1}{\kappa +t},\qquad \forall t\ge 0.
\end{equation}
So the proof is complete in overspecified case.

\item $k=r$. By Proposition~\ref{prop: stagethree k=r} and induction,
we obtain
\begin{equation}
    E_{t+1} \le (1-\frac{\gamma_t\sigma_r}{3})E_t \le (1-\frac{c_\gamma \sigma_r}{\sigma_1})E_t, \forall t\ge 0.
\end{equation}
Applying this inequality recursively and noting $c_\gamma \ge c_1\sqrt{\frac{2}{\pi}}$, we obtain
\begin{equation}
    E_{\cT_2+t} \le (1-\frac{c_\gamma \sigma_r}{\sigma_1})^t E_{\cT_2}\le \left(1-\frac{c_1\sqrt{\frac{2}{\pi}}}{\kappa}\right)^t 0.01\sigma_r, \forall t \ge 0.
\end{equation}
Thus, we can take $c_6 = 0.01/4$, $c_7 = c_1\sqrt{\frac{2}{\pi}}$, apply Lemma~\ref{lem: decomposition of FF-X} and get
\begin{align}
     \norm{F_tF_t^\top -\truX} \le c_6\sigma_r\left(1-\frac{c_7}{\kappa}\right)^t,\qquad \forall t\ge 0.
\end{align}
\end{itemize}
The proof is complete.
\end{proof}
\begin{thm}\label{thm: exact recovery under models faster}
 Suppose under either Model \ref{md: arbitraryCorruption} or \ref{md: randomCorruption}, we have $m\geq c_1' dk^2\kappa^{4}(\log \kappa + \log k)\log d$ and  $p\le \frac{c_2'}{\kappa\sqrt{k}}$ for some constants $c_1',c_2'$ depending only on $c_0$ and $c_3$. Then under both models, with probability at least $1- c_4' \exp(-c_5' \frac{m}{k^2\kappa^{4}})-\exp(-(pm+d))$ for some constants $c_4',c_5'$ depending only on $c_0$ and $c_3$, our subgradient method \eqref{eqn: mainalgorithm} 
 with the initialization  in Algorithm \ref{alg: initialization}
 and the adaptive stepsize choice \eqref{eqn: stepsizerule} with  $C_\eta \in [ \frac{c_6'}{\theta_{\frac{1}{2}}\sigma_1}, \frac{c_7'}{\theta_{\frac{1}{2}}\sigma_1}]$ with some universal $c_6',c_7'\leq 0.001$, converges as stated in \Cref{lem: mainlemma faster}.
\end{thm}
\begin{proof}
We can WLOG only prove this for model~\ref{md: arbitraryCorruption} because model~\ref{md: randomCorruption} can be reduced to model~\ref{md: arbitraryCorruption} by adding a small failure probability. Taking  $\epsilon = \frac{c_0 \theta_{\frac{1}{2}}}{4L\kappa}$  in Proposition~\ref{prop: estimate fnorm arbitrary}, where $L$ is a universal constant doesn't depend on anything from Proposition~\ref{prop: estimate fnorm arbitrary}, we know that with probability at least $1-c_8'\exp(-c_9'\frac{m}{\kappa^2})$
	\begin{align}
		\xi_{\frac{1}{2}}\left(\{|y_i|\}_{i=1}^{m}\right) \in [\theta_{\frac{1}{2}} - L(p+\epsilon), \theta_{\frac{1}{2}}+ L(p+\epsilon)]\fnorm{\truX},
	\end{align}
given $m \ge c_{10}' dr \kappa^2 \log d \log \kappa $. Here $c_8',c_9',c_{10}'$ are constants depending only on $c_0$.
Given $ c_2' \le \frac{c_0\theta_{\frac{1}{2}}}{4L}$, the above inclusion implies that 
\begin{align}
	\left|1-\frac{	\xi_{\frac{1}{2}}\left(\{|y_i|\}_{i=1}^{m}\right) }{\fnorm{\truX}\theta_{\frac{1}{2}}}\right| \le \frac{L(p+\epsilon)}{\theta_{\frac{1}{2}}} \le \frac{c_0}{2\kappa}.
\end{align}

Take $\delta = \frac{c_0\sqrt{\frac{2}{\pi}}}{12(1+\frac{L}{\theta_{\frac{1}{2}}})}\frac{\sigma_r}{\sigma_1 \sqrt{r}}$ in lemma~\ref{lem: arbitrary corruption}, we know that with probability at least  $1-c_{11}'\exp(-c_{12}'\frac{m}{\kappa^2 r})-\exp(-(pm+d))$ for constants $c_{11}',c_{12}'$ depending only on $c_0$,
\begin{equation}
	\norm{BB^\top -\sqrt{\frac{2}{\pi}}  \truX / \| \truX \|_F} \le \frac{c_0\sqrt{\frac{2}{\pi}}}{2(1+\frac{L}{\theta_{\frac{1}{2}}})}\frac{\sigma_r}{\sigma_1\sqrt{r}},
\end{equation}
given $m \ge c_{13}' dr \kappa  (\log \kappa +\log r) $ with $c_{13}'$ depending only on $c_0$.
The above inequality implies that 
\begin{align}
	\norm{\frac{ \theta_{\frac{1}{2}}(\{|y_i|\}_{i=1}^{m})}{\sqrt{\frac{2}{\pi}}\theta_{\frac{1}{2}}}B B^\top - \frac{ \theta_{\frac{1}{2}}(\{|y_i|\}_{i=1}^{m})\truX}{\fnorm{\truX}\theta_{\frac{1}{2}}}} &\le \frac{1+\frac{L}{\theta_{\frac{1}{2}}}}{\sqrt{\frac{2}{\pi}}}  \frac{c_0\sqrt{\frac{2}{\pi}}}{2(1+\frac{L}{\theta_{\frac{1}{2}}})}\frac{\sigma_r}{\sigma_1 \sqrt{r}}\fnorm{\truX} \\
	&\le \frac{c_0 \sigma_r}{2}.
\end{align}
Combining, we can find some constants $c_{14}',c_{15}',c_{16}'$ depending only on $c_0$ such that whenever $m \ge c_{14}' dr \kappa^2 \log d (\log \kappa +\log r) $, then with probability at least $1-c_{15}'\exp(-c_{16}'\frac{m}{\kappa^4 r})-\exp(-(pm+d))$ ,
\begin{align}
	&\;\;\;\norm{\frac{ \theta_{\frac{1}{2}}(\{|y_i|\}_{i=1}^{m})}{\sqrt{\frac{2}{\pi}}\theta_{\frac{1}{2}}}B B^\top - \truX}\\
	 &\le \norm{\frac{ \theta_{\frac{1}{2}}(\{|y_i|\}_{i=1}^{m})}{\sqrt{\frac{2}{\pi}}\theta_{\frac{1}{2}}}B B^\top - \frac{ \theta_{\frac{1}{2}}(\{|y_i|\}_{i=1}^{m})\truX}{\fnorm{\truX}\theta_{\frac{1}{2}}}} + \norm{\left(1-\frac{	\theta_{\frac{1}{2}}\left(\{|y_i|\}_{i=1}^{m}\right) }{\fnorm{\truX}\theta_{\frac{1}{2}}}\right)\truX}\\
	&\le \frac{c_0 \sigma_r}{2}+ \frac{L(p+\epsilon)}{\theta_{\frac{1}{2}}}\sigma_1\\
	&\le c_0\sigma_r.
\end{align}
Thus, $\norm{F_0F_0^\top - \truX} \le c_0\sigma_r$, which is the first condition.\\
Recall stepsize rule~\eqref{eqn: stepsizerule},
\begin{equation}
 \tau_{\mathcal{A},y}(F) =\xi_{\frac{1}{2}}\left(\{|\dotp{A_i,FF^\top}-y_i|\}_{i=1}^{m}\right), \quad \text{and} \quad \eta_t = C_\eta \tau_{\mathcal{A},y}(F_t).
\end{equation}
By Proposition~\ref{prop: estimate fnorm arbitrary}, with same choice of $\epsilon$, we know that with probability at least least $1-c_{17}'\exp(-c_{18}'\frac{m}{\kappa^2})$
	\begin{align}
	\tau_{\mathcal{A},y}(F_t) \in [\theta_{\frac{1}{2}} - L(p+\epsilon), \theta_{\frac{1}{2}}+ L(p+\epsilon)]\fnorm{F_tF_t^\top-\truX},\qquad \forall t\ge 0,
	\end{align}
given $m \ge c_{19}' dr \kappa^2 \log d \log \kappa $. Here $c_{17}',c_{18}',c_{19}'$ are constants depending only on $c_0$. By our condition on $C_\eta$, we know 
\begin{equation}
    \frac{\eta_t}{\fnorm{F_tF_t^\top -\truX}} \in \left[\frac{c_6' (1-\frac{c_0}{2\kappa})}{\sigma_1},\frac{c_7' (1+\frac{c_0}{2\kappa})}{\sigma_1}\right]
\end{equation}
Hence, the second condition in Theorem~\ref{lem: mainlemma faster} is satisfied.\\
By Proposition~\ref{prop: main RDPP for two models}, we know that  whenever $m \gtrsim c_{20}'dk^2\kappa^4 (\log k +\log \kappa)$ for some constant depending on $c_3$ and $c_2'\le  c_3/10$, $(r+k,\delta)$-RDPP holds with $\delta \le \frac{c_3}{\kappa \sqrt{k+r}}$ and $\psi(X)=\sqrt{\frac{2}{\pi}}$ with probability at least $1-\exp(-(pm+d)) - \exp(-\frac{c_{21}m}{k^2\kappa^4})$ for some constant $c_{21}$ depending only on $c_3$.\\
Since all the constants we introduced in this proof depend only on $c_0$ and $c_3$, so we can combing them and find desired $c_i', i\ge 1$.
\end{proof}
\section{RDPP and $\ell_1/\ell_2$-RIP}
	\label{sec: RDPP and l1l2RIP}
	Recall our definition of $(k',\delta)$ RDPP states that for all rank at most $k'$ matrix $X$, the following holds: 
\begin{equation}\label{eq: appendixRDPP}
	D(X):=\frac{1}{m}\sum_{i=1}^{m}{\sign}(\dotp{A_i,X} -s_i)A_i,\quad \text{and}\quad
		\norm{D(X) - \psi(X) \frac{X}{\fnorm{X}}} \le \delta.
\end{equation}

The $(k',\delta)\;\ell_1/\ell_2$-RIP states that for all rank $k'$ matrix $X$, the following holds.
	\begin{equation}
		\left(\sqrt{\frac{2}{\pi}}-\delta\right)\fnorm{X} \le \frac{1}{m}\sum_{i=1}^{m}\abs{\dotp{A_i,X}} \le \left(\sqrt{\frac{2}{\pi}}+\delta\right)\fnorm{X}. 
	\end{equation}
We shall utilize the following top $k'$ Frobenius norm: for an matrix $Y\in \RR^{d\times d}$
\[
\fknorm{Y}{k'} := \sqrt{\sum_{i=1}^{k'}\sigma_i^2(Y)} = \sup_{\rank(Z)\leq k',\fnorm{Z}= 1}\dotp{Y,Z}.
\]
Here $\sigma_i(Y)$ is the $i$-th largest singular value of $Y$. The second variational characterization can be proved by considering the orthogonal projection of the rank $k'$ singular vector space of $Y$ and its complement.

Now suppose there holds the $(k',\frac{\delta}{\sqrt{k'}})$ RDPP with corruption always $0$ and scale function being $\sqrt{\frac{2}{\pi}}$. Then we have 
\begin{equation}
    \begin{aligned}
   \frac{\delta}{\sqrt{k'}} 
   & \overset{(a)}{\geq} 	\norm{D(X) - \psi(X) \frac{X}{\fnorm{X}}} \\
   & \overset{(b)}{=} \norm{\frac{1}{m}\sum_{i=1}^{m}{\sign}(\dotp{A_i,X} )A_i - \sqrt{\frac{2}{\pi}} \frac{X}{\fnorm{X}} } \\ 
   & \overset{(c)}{\geq}\frac{1}{\sqrt{k'}} 
   \fknorm{\frac{1}{m}\sum_{i=1}^{m}{\sign}(\dotp{A_i,X})A_i - \sqrt{\frac{2}{\pi}} \frac{X}{\fnorm{X}} }{k'}\\
   & \overset{(d)}{=}\frac{1}{\sqrt{k'}}  \sup_{\rank(Y)\leq k',\fnorm{Y}\leq 1}
   \dotp{\frac{1}{m}\sum_{i=1}^{m}{\sign}(\dotp{A_i,X} )A_i - \sqrt{\frac{2}{\pi}} \frac{X}{\fnorm{X}},Y}\\
   & \overset{(e)}{\geq} \frac{1}{\sqrt{k'}} 
   \dotp{\frac{1}{m}\sum_{i=1}^{m}{\sign}(\dotp{A_i,X} )A_i - \sqrt{\frac{2}{\pi}} \frac{X}{\fnorm{X}},\frac{X}{\fnorm{X}}}\\
   & =\frac{1}{\sqrt{k'}}  \left(\frac{1}{m}\sum_{i=1}^m \frac{\abs{\dotp{A_i,X}}}{\fnorm{X}} - \sqrt{\frac{2}{\pi}}
   \right).
    \end{aligned}
\end{equation}
Here in the step $(a)$, we use the definition of RDPP. In the step $(b)$, we use the assumption on $s_i=0$ always and $\psi = \sqrt{\frac{2}{\pi}}$. In $(c)$, we use the relationship between operator norm and top $k'$ Frobenius norm. In step $(d)$, we use the variational characterization of top $k'$ Frobenius norm. In step $(e)$, we use the fact that $X$ is rank at most $k'$. The above derivation completes one side of the $\ell_1/\ell_2$-RIP. The other side can be proved by taking $Y=-\frac{X}{\fnorm{X}}$ in the above step $(e)$.
\section{Auxiliary Lemmas}
	This section contains lemmas that will be useful in the proof.
\begin{lemma}\label{lem: S(I-S) svd} 
	Let $A$ be an $n\times n$ symmetric matrix. Suppose that $\|A\| \le \frac{1}{2\eta}$, the largest singular value and the smallest singular value of $A(I-\eta A)$ are $\sigma_1(A) - \eta \sigma_1^2(A)$ and $\sigma_m(A) - \eta \sigma_m^2(A)$.
\end{lemma}
\begin{proof}
	Let $U_A \Sigma_A U_A^\top$ be the SVD of $A$. Simple algebra shows that
	\begin{equation}
		A\left(I - \eta A\right) = U_A\left(\Sigma_A - \eta\Sigma_A^2\right) U_A^\top.
	\end{equation}
	This is exactly the SVD of $A\left(I - \eta A\right)$. Let $g(x) = x - \eta x^2$. By taking derivative, $g$ is monotone increasing in interval $[-\infty, \frac{1}{2\eta}]$. Since the singular values of $A(I-\eta A)$ are exactly the singular values of $A$ mapped by $g$, the result follows.
\end{proof}
\begin{lemma}\label{lem: S(I-SS) svd}
	Let $A$ be an $m\times n$ matrix. Suppose that $\|A\| \le \sqrt{\frac{1}{3\eta}}$, the largest singular value and the smallest singular value of $A(I-\eta A^\top A)$ are $\sigma_1(A) - \eta \sigma_1^3(A)$ and $\sigma_m(A) - \eta \sigma_m^3(A)$.
\end{lemma}
\begin{proof}
	Let $U_A \Sigma_A V_A^\top$ be the SVD of $A$. Simple algebra shows that
	\begin{equation}
		A\left(I - \eta A^\top A\right) = U_A\left(\Sigma_A - \eta\Sigma_A^3\right) V_A^\top.
	\end{equation}
	This is exactly the SVD of $A\left(I - \eta A^\top A\right)$. Let $g(x) = x - \eta x^3$. By taking derivative, $g$ is monotone increasing in interval $[-\sqrt{\frac{1}{3\eta}}, \sqrt{\frac{1}{3\eta}}]$. Since the singular values of $A(I-\eta A^\top A)$ are exactly the singular values of $A$ mapped by $g$, the result follows.
\end{proof}
\begin{lem}\label{lem: norm of inverse}
	Let $A$ be an $n\times n$ matrix such that $\norm{A} <1$. Then $I +A$ is invertible and 
	\begin{equation}
		\norm{(I+A)^{-1}} \le \frac{1}{1-\norm{A}}.
	\end{equation}
\end{lem}
\begin{proof}
	Since $\norm{A}<1$, the matrix $B= \sum_{i=0}^{\infty}(-1)^iA^i$ is well defined and indeed $B$ is the inverse of $I+A$. By continuity, subaddivity and submultiplicativity of operator norm,
	\begin{align}
		\norm{(I+A)^{-1}} = \norm{B} \le \sum_{i=0}^{\infty} \norm{A^{i}} \le \sum_{i=0}^{\infty} \norm{A}^i = \frac{1}{1-\norm{A}}.
	\end{align}
\end{proof}
\begin{lem}\label{lem: sigmar ineq}
	Let $A$ be an $r\times r$ matrix and $B$ be an $r\times k$ matrix. Then
	\begin{equation}
		\sigma_r(AB) \le \norm{A} \sigma_r(B).
	\end{equation}
\end{lem}
\begin{proof}
	For any $r \times k$ matrix $C$, the variational expression of $r$-th singular value is 
	\begin{equation}
		\sigma_r(C) = \sup_{\substack{\text{subspace } S \subset \RR^k\\ \dim(S)=r}}\inf_{\substack{x \in S\\ x\neq 0}} \frac{\norm{Cx}}{\norm{x}}. 
	\end{equation}
Applying this variational result twice, we obtain
\begin{align}
	\sigma_r(AB)&= \sup_{\substack{\text{subspace } S \subset \RR^k\\ \dim(S)=r}}\inf_{\substack{x \in S\\ x\neq 0}} \frac{\norm{ABx}}{\norm{x}}\\
	&\le  \sup_{\substack{\text{subspace } S \subset \RR^k\\ \dim(S)=r}}\inf_{\substack{x \in S\\ x\neq 0}} \frac{\norm{A}\norm{Bx}}{\norm{x}}\\
	&= \norm{A}\sigma_r(B).
\end{align}
\end{proof}
\begin{lem}[Weyl's Inequality]\label{lem: weyl's ineq}
	Let $A$ and $B$ be any $m\times n$ matrices. Then 
	\begin{equation}
		\sigma_i(A-B) \le \norm{A-B},\qquad \forall 1\le i \le \min\{m,n\}.
	\end{equation}
When both $A$ and $B$ are symmetric matrices, the singular value can be replaced by eigenvalue.
\end{lem}
\begin{lem}\label{lem: fnorm opnorm}
	Let $A$ be any $m\times n$ matrix with rank $r$. Then 
	\begin{equation}
		\norm{A} \le \fnorm{A} \le \sqrt{r}\norm{A}.
	\end{equation}
\end{lem}
\begin{proof}
	Let $\sigma_1 \ge\sigma_2\ge\ldots \ge \sigma_r>0$ be singular values of $A$. Then we know $\norm{A}= \sigma_1$ and $\fnorm{A} = \sqrt{\sum_{i=1}^{r}\sigma_i^2}$. The result follows from Cauchy's inequality. 
\end{proof}

\begin{lem}\label{lem: decomposition of FF-X}
	Let $F_t$ be the iterates defined by algorithm~\ref{eqn: mainalgorithm}. Then we have 
	\begin{equation}
		\norm{F_tF_t^\top -\truX} \le \norm{S_tS_t^\top - D_S^*} + 2\norm{S_tT_t^\top} + \norm{T_tT_t^\top}.
	\end{equation}
Moreover,
\begin{equation}
	\max\{\norm{S_tS_t^\top -D_S^*}, \norm{S_tT_t^\top}, \norm{T_tT_t^\top}\} \le \norm{F_tF_t^\top -\truX}.
\end{equation}
\end{lem}
\begin{proof}
	Recall that $F_t = US_t + VT_t$ and $\truX = UD_S^*U^\top$, so we have
	\begin{align}
		F_tF_t^\top -\truX &= (US_t + VT_t)(US_t+VT_t)^\top -UD_S^*\\
		&= U(S_tS_t^\top-D_S^*)U^\top + US_tV_t^\top V^\top + VT_tS_t^\top U^\top + VT_tT_t^\top V^\top 
	\end{align}
By triangle inequality and the fact that $\norm{U}=\norm{V}=1$, we obtain
\begin{align}
	\norm{F_tF_t^\top -\truX} \le \norm{S_tS_t^\top -D_S^*} + 2\norm{S_tT_t^\top} + \norm{T_tT_t}.
\end{align}
For the second statement, we observe 
\begin{align}
	\norm{S_tS_t^\top -D_S^*} &= \sup_{ x \in \RR^r, \norm{x}=1} x^T \left(S_tS_t^\top -D_S^*\right)x\\
	& \le \sup_{ y \in \RR^d, \norm{y}=1} y^\top U \left(S_tS_t^\top -D_S^*\right) U^\top y
\end{align}
The last inequality follows from the fact that for any $x\in \RR^r$, we can find a $y\in\RR^d$ such that $U^\top y =x$ and $\norm{y}=\norm{x}$. Indeed, we can simply take $y=Ux$. On the other hand,
\begin{align}
	\sup_{ y \in \RR^d, \norm{x}=1} y^\top U \left(S_tS_t^\top -D_S^*\right) U^\top y &= \norm{U(S_tS_t^\top -D_S^*)Y^\top}\\
	&\le \norm{S_tS_t^\top -D_S^*},
\end{align}
so actually we have equality
\begin{equation}
	\sup_{ x \in \RR^r, \norm{x}=1} x^T \left(S_tS_t^\top -D_S^*\right)x= \sup_{ y \in \RR^d, \norm{y}=1} y^\top U \left(S_tS_t^\top -D_S^*\right) U^\top y.
\end{equation}
Clearly, the sup can be attained, let $y_* = \argmax_{ y \in \RR^d, \norm{y}=1} y^\top U \left(S_tS_t^\top -D_S^*\right) U^\top y$. Then we claim that $y_*$ must lie in the column space of $U$. If not so, we can always take the projection of $y_*$ onto the column space of $U$ and normalize it, which will give a larger objective value, contradiction. As a result, $V^\top y^* =0$ and we obtain
\begin{align}
	\norm{S_tS_t^\top -D_S^*} & = y_*^\top U(S_tS_t^\top -D_S^*)U^\top y_*\\
	&= y_*^\top \left(F_tF_t^\top -\truX\right) y_*\\
	&\le\norm{F_tF_t^\top -\truX}.
\end{align}
We can apply the same argument to get $\norm{S_tT_t^\top}\le \norm{F_tF_t^\top-\truX}$ and $\norm{T_tT_t^\top}\le \norm{F_tF_t^\top-\truX}$.
\end{proof}

\begin{lem}[$\ell_1/\ell_2$-RIP, { \cite[Proposition 1 ]{li2020nonconvex}}]\label{lem:ell1/ell2} 
	Let $r \ge 1$ be given, suppose sensing matrices $\{A_i\}_{i=1}^{m}$ have i.i.d. standard Gaussian entries with $m \gtrsim dr$. Then for any $0<\delta <\sqrt{\frac{2}{\pi}}$, there exists a universal constant $c>0$, such that with probability exceeding $1-\exp(-cm\delta^2)$, we have 
	\begin{equation}
		\left(\sqrt{\frac{2}{\pi}}-\delta\right)\fnorm{X} \le \frac{1}{m}\sum_{i=1}^{m}\abs{\dotp{A_i,X}} \le \left(\sqrt{\frac{2}{\pi}}+\delta\right)\fnorm{X} 
	\end{equation}
	for any rank $2r$-matrix $X$.
\end{lem}

\begin{lem}[$\ell_2$-RIP,  {\cite[Theorem 2.3]{candes2011tight}}]\label{lem: ell2-rip}
	Fix $0<\delta<1$, suppose that sensing matrices $\{A_i\}_{i=1}^{m}$ have i.i.d. standard Gaussian entries with $m \gtrsim \frac{1}{\delta^2} \log\left(\frac{1}{\delta}\right)dr$. Then with probability exceeding $1-C\exp\left(-Dm\right)$, we have 
	\begin{equation}
		(1-\delta)\fnorm{X}^2 \le \frac{1}{m}\sum_{i=1}^{m}\dotp{A_i,X}^2 \le (1+\delta)\fnorm{X}^2
	\end{equation}
	for any rank-$r$ matrix $X$. Here $C, D$ are universal constants.
\end{lem}
\begin{proof}
This lemma is not exactly as Theorem 2.3 in~\cite{candes2011tight} stated, but it's straight forward from the proof of this theorem. All we need to note in that paper is that the sample complexity we need is $m \gtrsim \frac{\log(\frac{1}{\delta})}{c}$, where $c$ is the constant defined in Theorem 2.3~\cite{candes2011tight} for $t$ chosen to be $\delta$. By standard concentration, $c \lesssim \frac{1}{\delta^2}$ and the result follows.
\end{proof}

\begin{lem}[Covering number for symmetric low rank matrices]\label{lem: covering lemma for low rank matrices}
	Let $\mathbb{S}_r = \{X \in \bS^{d\times d} \colon \rank(X) \le r, \fnorm{X} =1\}$. Then, there exists an $\epsilon$-net $\mathbb{S}_{\epsilon,r}$ with respect to the Frobenius norm satisfying $|\mathbb{S}_{\epsilon, r}|\le \left(\frac{9}{\epsilon}\right)^{(2d+1)r}$.
\end{lem}
\begin{proof}
	The proof is the same as the proof of lemma 3.1 in \cite{candes2011tight}, except that we will do eigenvalue decomposition, instead of SVD.
\end{proof}
\begin{lem}[ {\cite[Lemma A.1]{li2020non}}]\label{lem: concentration for sample level meadian}
	Suppose $F(\cdot)$ is cumulative distribution function with continuous density function $f(\cdot)$. Assume the samples $\{x_i\}_{i=1}^{m}$ are i.i.d. drawn from $f$. Let $0<p<1$. If $l<f(\theta)<L$ for all $\theta$ in $\{\theta \colon |\theta-\theta_p| \le \epsilon\}$, then
	\begin{equation}
		|\theta_p(\{x_i\}_{i=1}^m) - \theta_p(F)| < \epsilon
	\end{equation}
holds with probability at least $1-2\exp(-2m\epsilon^2l^2)$. Here $\theta_p(\{x_i\}_{i=1}^m)$ and $\theta_p(F)$ are $p$-quantiles of samples and distribution $F$~(see Definition 5.1 in \cite{li2020non})
\end{lem}
\begin{lem}[Concentration of operator norm]\label{lem: concentration of opnorm}
	Let $A$ be a $d$-by-$d$ GOE matrix having $N(0,1)$ diagonal elements and $N(0,\frac{1}{2})$ off-diagonal elements. Then we have
	\begin{equation}
		\expect{\norm{A}} \le \sqrt{d}
	\end{equation}
and
\begin{equation}
	P\left(\norm{A} - \expect{\norm{A}} \ge t\right) \le e^{-\frac{t^2}{2}}.
\end{equation}
\end{lem}
\begin{proof}
	We will use the following two facts~\cite{wainwright2019high}:
	\begin{enumerate}
		\item For a $d$-by-$d$ matrix $B$ with i.i.d. $N(0,1)$ entries, 
		\begin{equation}
			\expect{\norm{B}} \le 2\sqrt{d}. 
		\end{equation}
	\item Suppose $f$ is $L$-Lipschitz(with respect to the Euclidean norm) function and $a$ is a standard normal vector, then
	\begin{equation}
		P\left(f(a) - \expect{f(a)} \ge t \right) \le e^{-\frac{t^2}{2L}}.
	\end{equation}
	\end{enumerate}
Now we can prove the lemma. Firstly, we note that $A$  has the same distribution as $\frac{B+B^\top}{4}$ where $B$ has i.i.d. standard normal entries. By the first fact, we obtain
\begin{equation}
	\expect{\norm{A}}\le \expect{ \frac{\norm{B}+ \norm{B^\top}}{4}} \le \sqrt{d}.
\end{equation}
On the other hand, $\norm{A}$ can be written as a function of $\{A_{ii}\}$ and $\{\sqrt{2}A_{ij}\}_{i<j}$, which are i.i.d. standard normal random variables. Simple algebra yields that this function is $1$-Lipschitz. By the second fact,
\begin{align}
	P\left(\norm{A}  \ge \sqrt{d} +t\right)&\le P\left( \norm{A} \ge \expect{\norm{A}} +t \right)\\
	&\le e^{-\frac{t^2}{2}}.
\end{align}
\end{proof}
\begin{lem}[Concentration for $\chi^2$ distribution]\label{lem: concentration for chisquare distribution}
	Let $Y \sim \chi^2(n)$ be a $\chi^2$ random variable. Then we have
	\begin{equation}
		P\left(Y  \ge (1+2\sqrt{\lambda} + 2\lambda) n\right) \le \exp(-\lambda n/2).
	\end{equation}
\end{lem}
\begin{proof}
	It follows from standard sub-exponential concentration inequality and the fact that the square of a standard normal random variable is sub-exponential~\cite{wainwright2019high}.
\end{proof}
\begin{lem}[{\cite[Lemma A.8]{li2020non}}]\label{lem: concetration for max fnorm}
	Suppose $A_i \in \RR^{d\times d}$'s are independnet GOE sensing matrices having $N(0,1)$ diagonal elements and $N(0,\frac{1}{2})$ off-diagonal elements, for $i=1,2,\ldots,m $ and $m \ge d$. Then
	\begin{equation}
		\max_{i=1,2,\ldots,m} \fnorm{A_i} \le 2\sqrt{d(d+m)}
	\end{equation} 
holds with probability exceeding $1-m\exp(-d(d+m)/2)$.
\end{lem}
\begin{proof}
	Let $A$ be a GOE sensing matrix described in this lemma, and $A_{ij}$ be the $ij$-th entry of $A$. Since 
	\begin{equation}
		\fnorm{A}^2 =\sum_{i=1}^{d}A_{ii}^2 + 2\sum_{i<j}A_{ij}^2 = \sum_{i=1}^{d}A_{ii}^2+\sum_{i<j} (\sqrt{2}A_{ij})^2,
	\end{equation}
we see that $\fnorm{A}^2$ is a $\chi^2(d(d+1)/2)$ random variable. By Lemma~\ref{lem: concentration for chisquare distribution}, we have 
\begin{equation}
	P\left(\fnorm{A}^2 \ge \left(1+2\sqrt{\lambda}+2\lambda\right)d^2 \right) \le \exp(-\lambda d^2/2)
\end{equation}
for any $\lambda>0$. Take $\lambda = \frac{d+m}{d} \ge 2.$ Simple calculus shows $2\lambda \ge 2\sqrt{\lambda}+1$. Thus, we obtain
\begin{equation}
	P\left(\fnorm{A}^2 \ge 4d(d+m)\right) \le \exp(-d(d+m)/2).
\end{equation}
Therefore, the proof is completed by applying the union bound.
\end{proof}
\begin{lem}[{\cite[Lemma A.2]{li2020non}}]\label{lem: order stats  bounded by infty norm}
	Given vectors $x= [x_1,x_2,\ldots,x_n]$ and $y=[y_1,y_2,\ldots,y_n]$. We reorder them so that 
	\begin{equation}
		x_{(1)} \le x_{(2)}\le\ldots\le x_{(n)}, \qquad \text{and} \qquad  y_{(1)}\le y_{(2)}\le\ldots \le y_{(n)}.
	\end{equation}
Then
\begin{align}
	\abs{x_{(k)} - y_{(k)}} \le \norm{x-y}_{\infty}, \qquad \forall k=1,2,\ldots,n.
\end{align}

\end{lem}
\begin{lem}[{\cite[Lemma A.3]{li2020non}}]\label{lem: corrupted/clean samples}
	Consider corrupted samples $y_i = \dotp{A_i, \truX}+ s_i$ and clean samples $\tilde y_i = \dotp{A_i, \truX}$, $i=1,2,\ldots,m$. If $\mu<\frac{1}{2}$ is the fraction of samples that are corrupted by outliers, for $\mu < p < 1-\mu$, we have
	\begin{align}
		\theta_{p-\mu}(\{|\tilde y_i|\}_{i=1}^{m}) \le \theta_p(\{|y_i|\}_{i=1}^{m}) \le \theta_{p+\mu}(\{|\tilde y_i|\}_{i=1}^{m})
	\end{align}
\end{lem}
\end{document}